\documentclass[times,sort&compress,3p]{elsarticle}

\journal{Journal of Multivariate Analysis}

\usepackage{amsmath,amssymb,amsthm,hyperref}
\usepackage[labelfont=bf]{caption} % DO NOT TOUCH (this is the standard for JMVA)
\usepackage{appendix} % For proper labeling of appendices
\usepackage{dsfont} % For the \ind \newcommand (indicator functions)
\usepackage{mathtools} % For \vcentcolon to define := and =:
\usepackage{booktabs} % For table options
\newcounter{pstep} % Define a new counter named 'pstep'
\newcommand{\step}[1]{%
\par\medskip\noindent% Start a new paragraph with space
\refstepcounter{pstep}% Increment counter and make it visible to \label
\textbf{Step \thepstep: #1.}% Format the text
}

\theoremstyle{plain}
\newtheorem{theorem}{Theorem}
\newtheorem{proposition}[theorem]{Proposition}
\newtheorem{lemma}[theorem]{Lemma}

\theoremstyle{definition}

\newtheorem{remark}{Remark}

\newcommand{\N}{\mathbb{N}}
\newcommand{\R}{\mathbb{R}}
\newcommand{\PP}{\mathsf{P}} % Russian style, do not change
\newcommand{\EE}{\mathsf{E}} % Russian style, do not change
\newcommand{\Var}{\mathsf{Var}} % Russian style, do not change
\newcommand{\Cov}{\mathsf{Cov}} % Russian style, do not change
\newcommand{\bb}[1]{\boldsymbol{#1}}
\newcommand{\OO}{\mathcal{O}}
\newcommand{\oo}{\mathrm{o}}
\newcommand{\rd}{\mathrm{d}}
\newcommand{\ind}{\mathds{1}}
\newcommand{\e}{\varepsilon}
\newcommand{\leqdef}{\vcentcolon=}
\newcommand{\tr}{\mathrm{tr}}
\newcommand{\etr}{\mathrm{etr}}

\allowdisplaybreaks

\begin{document}

\begin{frontmatter}

\title{A two-sample test for symmetric positive definite matrix distributions \\using Wishart kernel density estimators\vspace{-3mm}}

\author[a1]{Fr\'ed\'eric Ouimet\vspace{-2mm}}\ead{frederic.ouimet2@uqtr.ca}

\address[a1]{Universit\'e du Qu\'ebec \`a Trois-Rivi\`eres, Trois-Rivi\`eres, QC, Canada\vspace{-6mm}}

\begin{abstract}
We develop a nonparametric two-sample test for distributions supported on the cone of symmetric positive definite matrices. The procedure relies on the Wishart kernel density estimator (KDE) introduced by \citet{BelzileGenestOuimetRichards2025WishartKDE}, whose support-adaptive kernel alleviates boundary bias by remaining confined to the cone. Our test statistic is the rescaled integrated squared difference between two Wishart KDEs and can be expressed as a two-sample $V$-statistic via an explicit closed-form overlap of Wishart kernels, avoiding numerical integration. Under the null hypothesis of equal densities, we derive the asymptotic distribution in both the common shrinking-bandwidth and fixed-bandwidth regimes. The proposed method provides a kernel-based competitor to the empirical Laplace-transform two-sample test of \citet{Lukic2024Laplace}. Unlike the orthogonally invariant Hankel-transform test of \citet{LukicMilosevic2024AISM}, our statistic can detect alternatives that differ only through eigenvector structure, for instance, Wishart models with the same shape parameter and the same scale eigenvalues but different orientations. A comparative simulation study examines the empirical Type I error rates and power of the proposed and Laplace-transform tests.
\end{abstract}


\end{frontmatter}

\vspace{-4mm}
\section{Introduction}\label{sec:intro}

Positive definite matrices arise throughout statistics and applied probability as natural summaries of dependence and variability. Prominent examples include covariance and precision matrices in multivariate analysis \citep{Muirhead1982,Anderson2003,GuptaNagar2000}, realized covariance matrices computed from high-frequency returns in finance \citep{AndersenBollerslevDieboldLabys2003,BarndorffNielsenShephard2004,GourierouxJasiakSufana2009}, diffusion tensors in brain imaging \citep{BasserMattielloLeBihan1994,BasserPierpaoli1996,PierpaoliJezzardBasserBarnettDiChiro1996,ConturoEtAl1999,Beaulieu2002}, and kernel (Gram) matrices in machine learning \citep{ScholkopfSmola2002,HofmannScholkopfSmola2008,RasmussenWilliams2006}. In some of these applications, one may be confronted with the fundamental inferential task of comparing two populations (or two time periods) of (symmetric) positive definite matrices. Statistically, this leads to a two-sample problem on the cone $\mathcal{S}_{++}^d$ of positive definite matrices: based on two $\mathcal{S}_{++}^d$-valued samples, one seeks to assess whether the corresponding generating laws coincide. This question has been approached through geometry-aware comparisons of Fréchet-type means \citep{OsborneEtAl2013,EllingsonEtAl2017,Schwartzman2016Lognormal,YouPark2021}, eigenstructure-focused procedures \citep{SchwartzmanDoughertyTaylor2010}, and more general distributional comparisons in diffusion tensor and brain connectivity settings \citep{WhitcherEtAl2007,YouPark2022}, as well as through recent distribution-level tests on $\mathcal{S}_{++}^d$ built from empirical integral transforms \citep{Lukic2024Laplace,LukicMilosevic2024AISM,LukicMilosevic2024ChangePoint}.

In the $\R^d$ setting, a classical and influential approach to two-sample testing is to compare nonparametric density estimators through an $L^2$-type discrepancy, typically the integrated squared difference between two kernel density estimators (KDEs). The formulation of such statistics, together with a detailed derivation of their null limiting distribution and practical considerations, was developed in \citet{AndersonHallTitterington1994}. A broader literature has developed related $L^2$-based procedures and variants. This includes multivariate one- and two-sample formulations based on $L^2$ distances \citep{AhmadCerrito1993}, unequal sample sizes \citep{Li1996}, shrinking- and fixed-bandwidth regimes, bootstrap calibration, and local power analysis \citep{Li1999}, extensions to mixed or conditional data \citep{LiMaasoumiRacine2009}, and multi-sample homogeneity tests \citep{BabiluaNadaraya2019}. In parallel, methodological work has emphasized that global two-sample rejections are often complemented by diagnostics that highlight where the densities differ, as in the local difference tools advocated by \citet{Duong2013}. These developments motivate the present work: we seek an analogue of the integrated-squared-difference KDE test on the cone $\mathcal{S}_{++}^d$ that respects the positive-definiteness constraint and avoids spill-over beyond the support.

A key difficulty in transporting classical KDE-based tests to $\mathcal{S}_{++}^d$ is that fixed symmetric kernels allocate mass outside the cone, creating boundary bias and complicating both theory and computation. To address this issue, we build on the Wishart KDE introduced by \citet{BelzileGenestOuimetRichards2025WishartKDE}, an asymmetric, support-adaptive KDE that remains confined to $\mathcal{S}_{++}^d$ and mitigates boundary effects by locally tuning the Wishart kernel parameters. Using this estimator, we propose a nonparametric two-sample test whose statistic is the rescaled integrated squared difference between two Wishart KDEs. A notable feature of the Wishart kernel is that the overlap integral of two kernels admits a closed form, allowing the test statistic to be expressed as a two-sample $V$-statistic and evaluated without numerical integration. Under the null hypothesis of equal densities, we derive asymptotic null distributions in two regimes: in a common shrinking-bandwidth setting we obtain a centered Gaussian limit with explicit centering and variance constants, while in a fixed-bandwidth setting we obtain a weighted chi-square series limit characterized by the spectrum of a covariance operator on $L^2(\mathcal{S}_{++}^d)$.

Integral-transform approaches provide a complementary line of work for distributional testing, including $L^2$-type Cram\'er--von Mises statistics built from integrated squared differences of empirical transforms. In the univariate nonnegative setting, empirical Hankel-transform methodology for goodness-of-fit problems was initiated by \citet{BaringhausTaherizadeh2010}; \citet{BaringhausTaherizadeh2013} later developed a Kolmogorov--Smirnov-type exponentiality test based on the same transform. For the two-sample problem, a Cram\'er--von Mises procedure based on weighted integrated squared differences of empirical Hankel transforms was studied in detail by \citet{BaringhausKolbe2015}, building on a two-sample statistic already noted in \citet{BaringhausTaherizadeh2010}. In the setting of positive definite matrices, a recent and particularly relevant contribution is the orthogonally invariant two-sample test of \citet{LukicMilosevic2024AISM}, which leverages orthogonally invariant Hankel transforms of matrix argument and provides an explicit asymptotic null distribution. This direction builds on the Wishart goodness-of-fit method based on the orthogonally invariant Hankel transform, as developed by \citet{Hadjicosta2019PhD} and \citet{MR4169380}. A central distinction, however, is that the resulting discrepancy in that test depends only on eigenvalue information, and therefore may have limited sensitivity to alternatives that differ primarily through eigenvector structure. By contrast, for $d \geq 2$, our Wishart-KDE-based two-sample test statistic retains information about relative orientation and can detect, for example, differences between Wishart laws that share the same shape parameter and the same eigenvalues of the scale matrix but have different orientations.

To the best of our knowledge, among procedures developed specifically for positive definite matrix distributions, the only direct competitor to our proposed test that targets equality of the full matrix distributions is the Laplace-transform two-sample test of \citet{Lukic2024Laplace}. There, the discrepancy is defined as a weighted integrated squared difference between empirical Laplace transforms, with integration taken with respect to a noncentral Wishart measure, leading again to a $V$-statistic representation and enabling resampling-based calibration.

Finally, two-sample discrepancies of Cram\'er--von Mises type naturally feed into change-point methodology via scan statistics that compare distributions before and after candidate split points. In the scalar and multivariate settings, a representative body of work constructs change-point tests from $L^2$ distances between empirical characteristic functions (including rank-based variants) and develops corresponding theory under independence and dependence; see \citet{HuskovaMeintanis2006ECF,HuskovaMeintanis2006Ranks,HlavkaHuskovaMeintanis2020Paired}. Closely related approaches include energy-based multiple change-point procedures \citep{MattesonJames2014Nonparametric}, as well as kernel-based segmentation and scan statistics formulated through feature-space criteria and reproducing kernel Hilbert space discrepancies \citep{GarreauArlot2018Consistent,LiXieDaiSong2019ScanB}. In the setting of positive definite matrices, \citet{LukicMilosevic2024ChangePoint} adapt the Hankel-transform two-sample discrepancy to change-point inference, underscoring how integral-transform and $L^2$ discrepancy ideas on $\mathcal{S}_{++}^d$ extend beyond the static two-sample problem.

The paper is organized as follows. Section~\ref{sec:definitions} introduces the proposed two-sample test statistic together with preliminary definitions and notation. Section~\ref{sec:main.results} contains the main theoretical results, where we derive its asymptotic null distribution in both the common shrinking-bandwidth regime (Theorem~\ref{thm:null.shrinking.bandwidth}) and the fixed-bandwidth regime (Theorem~\ref{thm:null.fixed.bandwidth}). Section~\ref{sec:power.study} presents a simulation study comparing the proposed test with the Laplace-transform test of \citet{Lukic2024Laplace}. Section~\ref{sec:outlook} summarizes our findings and outlines several directions for future research. Finally, the proofs of the main theoretical results are gathered in Appendix~\ref{app:proofs}, with some technical lemmas deferred to Appendix~\ref{app:tech.lemmas}.

\section{Definitions and notation}\label{sec:definitions}

Fix an integer $d \in \N = \{1,2,\ldots\}$. Let $\mathcal{S}^d$ denote the space of $d\times d$ real symmetric matrices and let $\mathcal{S}_{++}^d$ denote the cone of (symmetric) positive definite matrices. Throughout, $\tr(\cdot)$ denotes the trace, $\mathrm{etr}(\cdot) \equiv \exp\{\tr(\cdot)\}$ the exponential trace, and $|\cdot|$ the determinant.

Throughout, integration on $\mathcal{S}^d$ and $\mathcal{S}_{++}^d$ is with respect to coordinate Lebesgue measure. In particular, for $X = (X_{ij}) \in \mathcal{S}^d$, $\rd X = \prod_{1\leq i\leq j\leq d}\rd X_{ij}$. The multivariate gamma function $\Gamma_d$ is defined by
\begin{equation}\label{eq:multivariate.gamma}
\Gamma_d (\alpha)
= \int_{\mathcal{S}_{++}^d} \etr(-X) \, |X|^{\alpha - (d+1)/2} \, \rd X
= \pi^{d(d-1)/4} \prod_{i=1}^d \Gamma(\alpha - (i - 1)/2), \qquad \Re(\alpha) > (d - 1)/2.
\end{equation}
where $\Gamma_1(\cdot)$ reduces to the usual gamma function $\Gamma(\cdot)$.

The Wishart distribution with shape parameter $\nu \in (d - 1,\infty)$ and scale matrix $\Sigma \in \mathcal{S}_{++}^d$ has density, relative to $\rd X$,
\begin{equation}\label{eq:Wishart.kernel}
K_{\nu,\Sigma}(X)
 = \frac{|X|^{\nu/2 - (d+1)/2} \, \etr(-\Sigma^{-1}X/2)}{|2\Sigma|^{\nu/2} \, \Gamma_d(\nu/2)}, \qquad X\in \mathcal{S}_{++}^d.
\end{equation}
If a $d\times d$ random matrix $\mathfrak{X}$ follows this distribution, we write $\mathfrak{X}\sim \mathrm{Wishart}_d(\nu,\Sigma)$.

Let $\mathfrak{X}_1,\ldots,\mathfrak{X}_n$ be a random sample of independent and identically distributed (iid) $d\times d$ positive definite matrices with unknown density $f$ supported on $\mathcal{S}_{++}^d$. For a bandwidth parameter $b\in(0,\infty)$, the Wishart KDE of $f$ is given by
\begin{equation}\label{eq:Wishart.KDE}
\hat{f}_{n,b}(S) = \frac{1}{n} \sum_{t=1}^n K_{\nu(b,d), \, bS}(\mathfrak{X}_t), \qquad S\in \mathcal{S}_{++}^d,
\end{equation}
where $\nu(b,d) = 1/b+d+1$ for brevity.

This estimator was introduced by \citet{BelzileGenestOuimetRichards2025WishartKDE} (see also \citet{MR4358612} for a slightly different version) as a matrix-variate extension of the gamma KDE on $(0,\infty)$ proposed by \citet{Chen2000}. \citet{BelzileGenestOuimetRichards2025WishartKDE} established the asymptotic properties (mean squared error, uniform strong consistency of growing sequences of compact sets, and asymptotic normality) of the Wishart KDE \eqref{eq:Wishart.KDE} for strongly mixing observations.

The Wishart KDE provides an adaptive approach to density estimation on $\mathcal{S}_{++}^d$: by locally adjusting the Wishart kernel's shape and scale parameters at each evaluation point $S$, it mitigates boundary bias while remaining supported on $\mathcal{S}_{++}^d$. This adaptivity avoids the spill-over effect seen with fixed symmetric kernels near the boundary and ensures that $\smash{\hat{f}_{n,b}}$ stays confined to $\mathcal{S}_{++}^d$. Under the assumption that $f$ and its first- and second-order partial derivatives are bounded and uniformly continuous, \citet[Remark~3.2]{BelzileGenestOuimetRichards2025WishartKDE} derive a pointwise bias expansion of order $\OO(b)$ and describe how its leading term behaves along paths approaching the boundary. The specific parameterization above is chosen so that the kernel is peaked at the evaluation point $S$ and becomes increasingly concentrated around $S$ as $b \downarrow 0$. More precisely, as a density in the variable $X$, the map $X \mapsto K_{\nu(b,d), \, bS}(X)$ has mode $X = S$; see \citet[Eq.~(2.4)]{BelzileGenestOuimetRichards2025WishartKDE}.

In this section, a two-sample test is proposed using the Wishart KDE \eqref{eq:Wishart.KDE}. Let $f^{(1)}$ and $f^{(2)}$ be the respective density functions of two populations under study, each of which is supported on $\mathcal{S}_{++}^d$. In order to test the hypothesis that the two laws are equal, viz.,
\[
H_0 : f^{(1)} = f^{(2)} \qquad \text{versus} \qquad H_1 : f^{(1)} \neq f^{(2)},
\]
the two-sample test presented by \citet{AndersonHallTitterington1994} \citep[see also][]{BabiluaNadaraya2019} is adapted to our matrix setting. Let
\[
\mathfrak{X}^{(1)}_1,\ldots,\mathfrak{X}^{(1)}_{n_1}\stackrel{\mathrm{iid}}{\sim} f^{(1)}, \qquad
\mathfrak{X}^{(2)}_1,\ldots,\mathfrak{X}^{(2)}_{n_2}\stackrel{\mathrm{iid}}{\sim} f^{(2)},
\]
be two independent random samples from $f^{(1)}$ and $f^{(2)}$, respectively, where
\begin{equation}\label{eq:n1.n2.n}
n_1 + n_2 = n, \qquad \frac{n_1}{n}\stackrel{n \to \infty}{\longrightarrow} \varrho, \qquad \varrho \in (0,1).
\end{equation}
For all $\ell \in \{1,2\}$, let
\begin{equation}\label{eq:Donald.1}
\hat{f}_{n_{\ell}, b_{\ell}}^{(\ell)} (S)
 = \frac{1}{n_{\ell}} \sum_{i=1}^{n_{\ell}} K_{\nu(b_{\ell},d), \, b_{\ell} S}(\mathfrak{X}_i^{(\ell)}), \qquad S\in \mathcal{S}_{++}^d,
\end{equation}
be the Wishart KDE for $f^{(\ell)}$. Our proposed test statistic for testing $H_0$ is
\begin{equation}\label{eq:Donald.2}
T(n_1,n_2)
= \frac{n_1 n_2}{n} \int_{\mathcal{S}_{++}^d} \left\{\hat{f}_{n_1,b_1}^{(1)}(S) - \hat{f}_{n_2,b_2}^{(2)}(S)\right\}^2 \rd S \\
= \sum_{i=1}^2 n_i \int_{\mathcal{S}_{++}^d} \Bigg\{\hat{f}_{n_i,b_i}^{(i)}(S) - \sum_{\ell=1}^2 \frac{n_{\ell}}{n} \hat{f}_{n_{\ell},b_{\ell}}^{(\ell)}(S)\Bigg\}^2 \rd S.
\end{equation}

\begin{remark}\label{rem:bandwidth.selection}
For bandwidth selection in practical applications, a simple choice is to apply either the least-squares cross-validation (LSCV) or likelihood cross-validation (LCV) method developed by \citet{BelzileGenestOuimetRichards2025WishartKDE} to the pooled sample obtained by concatenating the two samples, and then set $b_1 = b_2$ equal to the selected bandwidth. This gives the two KDEs the same smoothing level and, under $H_0$, selects the bandwidth from the total sample drawn from the common density. Moreover, the selected bandwidth depends only on the pooled observations and not on their sample labels. This specifies the numerical value of the test statistic; its null calibration is described next.
\end{remark}

\begin{remark}\label{rem:permutation.calibration}
Once the common bandwidth has been selected, the test can be calibrated by permutation. Let $T_{\mathrm{obs}}$ denote the observed statistic. For every $\mathcal{I} \subseteq \{1,\ldots,n\}$ with $|\mathcal{I}| = n_1$, let $T_{\mathcal{I}}$ denote the statistic obtained by assigning the observations indexed by $\mathcal{I}$ to the first sample and the remaining observations to the second, while keeping the selected bandwidth fixed. The full permutation $p$-value is
\[
p_{\mathrm{perm}} \leqdef \binom{n}{n_1}^{-1} \sum_{\substack{\mathcal{I} \subseteq \{1,\ldots,n\} \\ |\mathcal{I}| = n_1}} \ind_{\{T_{\mathcal{I}} \geq T_{\mathrm{obs}}\}}.
\]
Under $H_0$, the sample labels are exchangeable conditionally on the pooled observations, while the pooled-sample bandwidth is unchanged by relabeling. Consequently, rejection when $p_{\mathrm{perm}} \leq \alpha$ defines a finite-sample valid test of level at most $\alpha$. In practice, the full permutation distribution can be approximated using $B \in \N$ allocations sampled independently and uniformly. If $T^{(1)},\ldots,T^{(B)}$ denote the corresponding statistics, the Monte Carlo permutation $p$-value is
\begin{equation}\label{eq:add.one.p.value}
\hat{p}_{\mathrm{perm}} \leqdef \frac{1 + \sum_{m=1}^B \ind_{\{T^{(m)} \geq T_{\mathrm{obs}}\}}}{B + 1}.
\end{equation}
This calibration applies whether the common bandwidth shrinks with the sample size or is held fixed, and it avoids estimating the unknown quantities appearing in the corresponding asymptotic null distributions derived below.
\end{remark}

Throughout the paper, the following notational conventions are adopted. The notation $u = \OO(v)$ means that $\limsup |u / v| \leq C < \infty$ as $n \to \infty$ or $b \downarrow 0$, depending on the context. The constant $C$ may depend on the target density $f$, the dimension $d$, and any other parameters explicitly held fixed in the relevant statement, but not on the asymptotic variables or on variables over which the estimate is asserted uniformly, unless such dependence is explicitly written as a subscript. The alternative notation $u \ll v$ is also used to mean $u,v \geq 0$ and $u = \OO(v)$, with the same subscript rule. If both $u \ll v$ and $u \gg v$ hold, we write $u \asymp v$. Similarly, the notation $u = \oo(v)$ means that $\lim |u / v| = 0$ as $n \to \infty$ or $b \downarrow 0$. The symbol $\rightsquigarrow$ denotes convergence in distribution. The notation $\|\cdot\|_F$ denotes the Frobenius norm. In the shrinking-bandwidth regime, $b = b(n)$ is implicitly a function of the number of observations, while $b$ is held fixed in the fixed-bandwidth regime. In Appendix~\ref{app:tech.lemmas}, the limit $b \downarrow 0$ is considered independently of $n$.

For any $d,k \in \N$, the shorthands
\[
r \equiv r(d) = \frac{d(d + 1)}{2} \qquad \text{and} \qquad [k] = \{1,\ldots,k\}
\]
will be used frequently.

\section{Main results}\label{sec:main.results}

This section collects the main analytical results for the proposed two-sample statistic $T(n_1,n_2)$ introduced in \eqref{eq:Donald.2}. We first show that the integrated squared difference of the two Wishart KDEs admits an explicit closed-form evaluation, so that $T(n_1,n_2)$ can be computed without any numerical integration over $\mathcal{S}_{++}^d$. We then derive the asymptotic null distribution in two complementary bandwidth regimes: (i) a common shrinking-bandwidth setting ($b_1 = b_2 = b_n\downarrow 0$), which yields a centered Gaussian limit after an explicit centering and scaling, and (ii) a fixed-bandwidth setting ($b_1 = b_2 = b$ fixed), which yields a weighted $\chi^2$ series limit described by the spectrum of a covariance operator on $L^2(\mathcal{S}_{++}^d)$. See Theorem~\ref{thm:null.shrinking.bandwidth} for (i) and Theorem~\ref{thm:null.fixed.bandwidth} for (ii).

We begin by rewriting $T(n_1,n_2)$ in a form that makes both computation and subsequent asymptotic analysis transparent. The key observation is that the overlap integral of two Wishart kernels can be evaluated in closed form, which turns the $L^2$-discrepancy between two Wishart KDEs into an explicit two-sample $V$-statistic of order $2$. This representation will be used throughout the remainder of the paper.

\begin{proposition}\label{prop:test.statistic}\addcontentsline{toc}{subsection}{Proposition~\theproposition}
For $i,\ell\in \{1,2\}$, let $\delta_{i,\ell}$ denote the usual Kronecker delta, i.e., $\delta_{i,\ell}=1$ or $0$ according to whether $i = \ell$ or $i \neq \ell$, and let $a_{i,\ell} = \delta_{i,\ell} - (n_{\ell}/n)$.
We have
\[
\begin{aligned}
T(n_1,n_2)
&= \frac{1}{2^{r(d)}} \sum_{i=1}^2 n_i \sum_{\ell_1=1}^2 \sum_{\ell_2=1}^2 \frac{a_{i,\ell_1} a_{i,\ell_2}}{n_{\ell_1} n_{\ell_2}} b_{\ell_1}^{-d/(2 b_{\ell_1}) - r(d)} b_{\ell_2}^{-d/(2 b_{\ell_2}) - r(d)} \frac{\Gamma_d(1/(2 b_{\ell_1}) + 1/(2 b_{\ell_2}) + (d+1)/2)}{\Gamma_d(1/(2 b_{\ell_1}) + (d+1)/2) \Gamma_d(1/(2 b_{\ell_2}) + (d+1)/2)} \\
&\hspace{55mm}\times \sum_{i_1=1}^{n_{\ell_1}} \sum_{i_2=1}^{n_{\ell_2}} \frac{|\mathfrak{X}_{i_1}^{(\ell_1)}|^{1/(2 b_{\ell_1})} |\mathfrak{X}_{i_2}^{(\ell_2)}|^{1/(2 b_{\ell_2})}}{|b_{\ell_1}^{-1} \mathfrak{X}_{i_1}^{(\ell_1)} + b_{\ell_2}^{-1} \mathfrak{X}_{i_2}^{(\ell_2)}|^{1/(2 b_{\ell_1}) + 1/(2 b_{\ell_2}) + (d+1)/2}}.
\end{aligned}
\]
\end{proposition}

Next, we consider the common shrinking-bandwidth setting. In this regime the statistic diverges under $H_0$ because it contains an integrated ``self-overlap'' contribution coming from squaring the KDEs; accordingly, we must subtract an explicit centering term of order $b_n^{-r/2}$ and rescale by $b_n^{r/4}$ to obtain a non-degenerate limit. The resulting fluctuation is asymptotically Gaussian, with variance proportional to a weighted $L^2$-moment of the common density $f$.

\begin{theorem}[Null limit in the common shrinking-bandwidth regime]\label{thm:null.shrinking.bandwidth}\addcontentsline{toc}{subsection}{Theorem~\thetheorem}
Assume $H_0$ holds, i.e., $f^{(1)} = f^{(2)} \equiv f$. Let the bandwidths be common and shrinking: $b_1(n) = b_2(n) \equiv b_n \downarrow 0$ as $n \to \infty$, and assume
\begin{equation}\label{eq:nb.condition}
n \, b_n^{r/2}\longrightarrow \infty.
\end{equation}
Assume furthermore that $f$ is bounded and continuous on $\mathcal{S}_{++}^d$, and define the constants
\begin{equation}\label{eq:moment.conditions}
V_{1,1} \leqdef \int_{\mathcal{S}_{++}^d} f(S) \, |S|^{-(d+1)/2} \, \rd S < \infty, \qquad
V_{2,2} \leqdef \int_{\mathcal{S}_{++}^d} f(S)^2 \, |S|^{-(d+1)/2} \, \rd S < \infty,
\end{equation}
(the notation here comes from the more general notation used in \eqref{eq:Mp.diag.def} and \eqref{eq:M2p.offdiag.def} of Appendix~\ref{app:tech.lemmas}) and assume in addition that
\begin{equation}\label{eq:moment.conditions.2}
\int_{\mathcal{S}_{++}^d} f(S) \, |S|^{-(d+1)} \, \rd S < \infty, \qquad
\int_{\mathcal{S}_{++}^d} f(S)^2 \, |S|^{-3(d+1)/2} \, \rd S < \infty.
\end{equation}
Define the $b$-dependent centering constant
\[
A_d(b) \leqdef b^{-r} \, 2^{-d/b-2r} \, \frac{\Gamma_d\left(1/b + (d+1)/2\right)} {\Gamma_d\left(1/(2b) + (d+1)/2\right)^2}, \qquad b > 0,
\]
and the dimensional constants
\[
c_d \leqdef 2^{-d(d + 2)/2} \, \pi^{-r/2}, \qquad
I_d \leqdef \int_{\mathcal{S}^{d}} \exp\left\{-\frac{\tr(U^2)}{4}\right\} \, \rd U, \qquad
v_d \leqdef 2 \, c_d^2 \, I_d.
\]
Then, as $n \to \infty$,
\begin{equation}\label{eq:clt.null.shrinking}
b_n^{r/4}\left\{T(n_1,n_2) - A_d(b_n) \, V_{1,1}\right\} \rightsquigarrow \mathcal{N}\left(0, \, v_d \, V_{2,2}\right).
\end{equation}
Moreover, the Gaussian integral $I_d$ evaluates explicitly as $I_d = (2\sqrt{\pi})^{d} \, (\sqrt{2\pi})^{d(d - 1)/2}$, and hence $v_d = 2^{-(3r + d - 2)/2} \, \pi^{-r/2}$.
\end{theorem}

\begin{remark}\label{rem:moment.conditions}
The integrability conditions in \eqref{eq:moment.conditions}--\eqref{eq:moment.conditions.2} principally control the behavior of $f$ near the singular boundary of $\mathcal{S}_{++}^d$, where the negative powers of the determinant diverge. Since $f$ is bounded, $V_{2,2} \leq \|f\|_\infty V_{1,1}$, so the two conditions in \eqref{eq:moment.conditions} amount to the single negative determinant moment $V_{1,1} = \EE\{|\mathfrak{X}|^{-(d + 1)/2}\} < \infty$. For example, all the conditions in \eqref{eq:moment.conditions}--\eqref{eq:moment.conditions.2} are automatic for a bounded target density supported on $\{S \in \mathcal{S}_{++}^d : \lambda_{\min}(S) \geq \e\}$ for some $\e > 0$. They also cover standard full-support models. Specifically, every inverse Wishart density $f(S) \propto |S|^{-(\nu + d + 1)/2}\etr(-\Psi S^{-1}/2)$, with $\nu > d - 1$ and $\Psi \in \mathcal{S}_{++}^d$, satisfies all the integrability conditions in \eqref{eq:moment.conditions}--\eqref{eq:moment.conditions.2}; for the Wishart density $f = K_{\nu,\Sigma}$, \eqref{eq:moment.conditions} holds when $\nu > 2d$, and all those integrability conditions hold when $\nu > 3d + 1$. These claims follow directly from \eqref{eq:multivariate.gamma}, after the change of variables $S \mapsto S^{-1}$ in the inverse Wishart case.
\end{remark}

\begin{remark}\label{rem:nb.condition}
If $b_n = n^{-\gamma}$, condition \eqref{eq:nb.condition} is equivalent to $\gamma < 2/r = 4/[d(d + 1)]$, so the bandwidth must decrease increasingly slowly as $d$ grows. The MSE-optimal order $b_n \asymp n^{-2/(r + 4)}$ derived by \citet{BelzileGenestOuimetRichards2025WishartKDE} satisfies this condition, but then $n b_n^{r/2} \asymp n^{4/(r + 4)}$, whose growth becomes slow in moderate dimension. Thus the shrinking-bandwidth Gaussian approximation may require a large sample, and bandwidths substantially smaller than the density-estimation order should be treated cautiously; the fixed-bandwidth limit considered in Theorem~\ref{thm:null.fixed.bandwidth} below provides an alternative that does not require \eqref{eq:nb.condition}.
\end{remark}

\begin{remark}\label{rem:fixed.alternatives}
The uniform strong consistency result of \citet[Theorem~3.4]{BelzileGenestOuimetRichards2025WishartKDE} supports consistency against fixed alternatives in the common shrinking-bandwidth regime. Suppose that its conditions hold for each population and that $b_n \asymp n^{-2/(r + 4)}$. If $H_1$ holds, continuity of $f^{(1)}$ and $f^{(2)}$ ensures that there exists a compact set $\mathcal{K} \subseteq \mathcal{S}_{++}^d$ such that
\[
D_{\mathcal{K}} \leqdef \int_{\mathcal{K}}\left\{f^{(1)}(S) - f^{(2)}(S)\right\}^2 \, \rd S > 0.
\]
Since $n_1/n \to \varrho$ and $n_2/n \to 1 - \varrho$, uniform strong consistency on expanding compact sets yields
\[
\int_{\mathcal{K}}\left\{\hat{f}_{n_1,b_n}^{(1)}(S) - \hat{f}_{n_2,b_n}^{(2)}(S)\right\}^2 \, \rd S \longrightarrow D_{\mathcal{K}}, \qquad \text{a.s.}
\]
Consequently, \eqref{eq:Donald.2} gives
\[
\liminf_{n \to \infty}\frac{T(n_1,n_2)}{n} \geq \varrho(1 - \varrho)D_{\mathcal{K}} > 0, \qquad \text{a.s.}
\]
Thus $T(n_1,n_2)$ grows at least linearly under such a fixed alternative. Under $H_0$, Theorem~\ref{thm:null.shrinking.bandwidth} suggests the oracle asymptotic level-$\alpha$ upper-tail critical value
\[
c_{n,\alpha} = A_d(b_n) \, V_{1,1} + b_n^{-r/4}(v_d \, V_{2,2})^{1/2} z_{1-\alpha},
\]
where $z_{1-\alpha}$ is the $(1-\alpha)$-quantile of the standard normal distribution. By \eqref{eq:Adb.asymptotic}, the first term has order $b_n^{-r/2}$, whereas the second has the smaller order $b_n^{-r/4}$. Hence $c_{n,\alpha} = \OO(b_n^{-r/2}) = \oo(n)$, where the last relation follows directly from \eqref{eq:nb.condition}. Thus the statistic is asymptotically separated from the null scale, and any upper-tail calibration whose critical value is of smaller order than $n$ in probability has power tending to one against the fixed alternatives considered above.
\end{remark}

Finally, we turn to the fixed-bandwidth setting, where $b$ is held constant as $n \to \infty$. In contrast to the common shrinking-bandwidth regime, the statistic $T(n_1,n_2)$ no longer requires an explicit diverging centering: it converges in distribution to the squared norm of a Gaussian element in the Hilbert space. Equivalently, the limit can be written as an (almost surely convergent) weighted $\chi^2$ series limit described by the spectrum of a covariance operator.

\begin{theorem}[Fixed-bandwidth null limit as a weighted chi-square series]\label{thm:null.fixed.bandwidth}\addcontentsline{toc}{subsection}{Theorem~\thetheorem}
Assume $H_0$ holds, i.e., $f^{(1)} = f^{(2)} \equiv f$. Fix the bandwidth $b \in (0,\infty)$. Let $\mathbb{H} \leqdef L^2(\mathcal{S}_{++}^d,\rd S)$ with inner product $\langle g,h\rangle_{\mathbb{H}} = \int_{\mathcal{S}_{++}^d} g(S)h(S) \, \rd S$. Define the feature map
\[
\phi_b(X)(S) \leqdef K_{\nu(b,d), \, bS}(X), \qquad X,S \in \mathcal{S}_{++}^d,
\]
and let $\mathfrak{X}$ be a random matrix with density $f$ supported on $\mathcal{S}_{++}^d$. Assume that
\begin{equation}\label{eq:fixed.b.moment}
V_{1,1} = \int_{\mathcal{S}_{++}^d} f(S) \, |S|^{-(d + 1)/2} \, \rd S < \infty.
\end{equation}
Then $\phi_b(\mathfrak{X})$ is a square-integrable $\mathbb{H}$-valued random element. In particular, $\mu_b \leqdef \EE\{\phi_b(\mathfrak{X})\} \in \mathbb{H}$ is well defined, and the covariance operator $\Sigma_b:\mathbb{H} \to \mathbb{H}$,
\begin{equation}\label{eq:Sigma.b}
\Sigma_b \, g \leqdef \EE\big[\langle \phi_b(\mathfrak{X}) - \mu_b, g\rangle_{\mathbb{H}} (\phi_b(\mathfrak{X}) - \mu_b)\big], \qquad g \in \mathbb{H},
\end{equation}
is well defined, self-adjoint, nonnegative definite and trace-class. Let $(\lambda_j,e_j)_{j=1}^{\infty}$ be an eigen-decomposition of $\Sigma_b$, i.e., $\{e_j\}_{j=1}^{\infty}$ is an orthonormal basis of $\mathbb{H}$ and $\Sigma_b \, e_j = \lambda_j e_j$ with $\lambda_j \geq 0$ and $\sum_{j=1}^{\infty} \lambda_j < \infty$. Let $T(n_1,n_2)$ be the test statistic in \eqref{eq:Donald.2} with $b_1 = b_2 = b$. Then, as $n \to \infty$,
\begin{equation}\label{eq:fixed.bandwidth.limit}
T(n_1,n_2) \rightsquigarrow \sum_{j=1}^{\infty} \lambda_j Z_j^2,
\end{equation}
where $Z_1,Z_2,\ldots \stackrel{\mathrm{iid}}{\sim} \mathcal{N}(0,1)$. The series in \eqref{eq:fixed.bandwidth.limit} converges almost surely and in $L^1$. Equivalently, if $\mathcal{G}$ is a centered Gaussian element in $\mathbb{H}$ with covariance operator $\Sigma_b$, then
\[
T(n_1,n_2) \rightsquigarrow \|\mathcal{G}\|_{\mathbb{H}}^2.
\]
\end{theorem}

\providecommand{\SimulationReplications}{20000}
\providecommand{\SimulationPermutations}{2499}
\providecommand{\SimulationConfigurations}{20}
\providecommand{\SimulationSampleDesigns}{$(20,20)$, $(50,50)$}
\providecommand{\SimulationDimension}{2}
\providecommand{\SimulationCovarianceVectors}{8}
\providecommand{\SimulationBoldCaptionSentence}{For each alternative configuration $\mathrm{A}_i$, $i = 1,2,\ldots,17$, the highest displayed power within each sample-size pair is shown in bold, with all ties at the displayed precision also shown in bold; null entries remain unbold because they estimate the Type I error rate rather than power.}

\providecommand{\SimulationNullMinimum}{4.7}
\providecommand{\SimulationNullMaximum}{5.3}
\providecommand{\SimulationMaximumBandwidthRateDifference}{5.2}
\providecommand{\SimulationBandwidthFits}{1600000}
\providecommand{\SimulationLSCVMinimum}{0.0000}
\providecommand{\SimulationLSCVMaximum}{7351.176}
\providecommand{\SimulationLCVMinimum}{0.0097}
\providecommand{\SimulationLCVMaximum}{79.691}
\providecommand{\SimulationLSCVMedianMinimum}{0.011}
\providecommand{\SimulationLSCVMedianMaximum}{0.267}
\providecommand{\SimulationLCVMedianMinimum}{0.027}
\providecommand{\SimulationLCVMedianMaximum}{0.191}
\providecommand{\SimulationNOneDesignALSCV}{5.3}
\providecommand{\SimulationNOneDesignALCV}{5.1}
\providecommand{\SimulationNOneDesignALukic}{5.0}
\providecommand{\SimulationNOneDesignBLSCV}{4.9}
\providecommand{\SimulationNOneDesignBLCV}{4.9}
\providecommand{\SimulationNOneDesignBLukic}{4.7}
\providecommand{\SimulationNTwoDesignALSCV}{5.1}
\providecommand{\SimulationNTwoDesignALCV}{5.1}
\providecommand{\SimulationNTwoDesignALukic}{5.2}
\providecommand{\SimulationNTwoDesignBLSCV}{5.2}
\providecommand{\SimulationNTwoDesignBLCV}{5.2}
\providecommand{\SimulationNTwoDesignBLukic}{5.0}
\providecommand{\SimulationNThreeDesignALSCV}{5.1}
\providecommand{\SimulationNThreeDesignALCV}{5.1}
\providecommand{\SimulationNThreeDesignALukic}{5.2}
\providecommand{\SimulationNThreeDesignBLSCV}{5.1}
\providecommand{\SimulationNThreeDesignBLCV}{5.1}
\providecommand{\SimulationNThreeDesignBLukic}{5.0}
\providecommand{\SimulationAOneDesignALSCV}{22.5}
\providecommand{\SimulationAOneDesignALCV}{20.9}
\providecommand{\SimulationAOneDesignALukic}{38.0}
\providecommand{\SimulationAOneDesignBLSCV}{43.5}
\providecommand{\SimulationAOneDesignBLCV}{40.8}
\providecommand{\SimulationAOneDesignBLukic}{77.0}
\providecommand{\SimulationATwoDesignALSCV}{28.9}
\providecommand{\SimulationATwoDesignALCV}{28.6}
\providecommand{\SimulationATwoDesignALukic}{5.9}
\providecommand{\SimulationATwoDesignBLSCV}{72.5}
\providecommand{\SimulationATwoDesignBLCV}{72.9}
\providecommand{\SimulationATwoDesignBLukic}{9.8}
\providecommand{\SimulationAThreeDesignALSCV}{26.4}
\providecommand{\SimulationAThreeDesignALCV}{25.8}
\providecommand{\SimulationAThreeDesignALukic}{19.4}
\providecommand{\SimulationAThreeDesignBLSCV}{63.0}
\providecommand{\SimulationAThreeDesignBLCV}{63.0}
\providecommand{\SimulationAThreeDesignBLukic}{47.8}
\providecommand{\SimulationAFourDesignALSCV}{54.1}
\providecommand{\SimulationAFourDesignALCV}{54.7}
\providecommand{\SimulationAFourDesignALukic}{94.3}
\providecommand{\SimulationAFourDesignBLSCV}{88.5}
\providecommand{\SimulationAFourDesignBLCV}{88.2}
\providecommand{\SimulationAFourDesignBLukic}{100.0}
\providecommand{\SimulationAFiveDesignALSCV}{37.0}
\providecommand{\SimulationAFiveDesignALCV}{35.5}
\providecommand{\SimulationAFiveDesignALukic}{39.2}
\providecommand{\SimulationAFiveDesignBLSCV}{67.3}
\providecommand{\SimulationAFiveDesignBLCV}{65.7}
\providecommand{\SimulationAFiveDesignBLukic}{87.6}
\providecommand{\SimulationASixDesignALSCV}{29.1}
\providecommand{\SimulationASixDesignALCV}{27.9}
\providecommand{\SimulationASixDesignALukic}{35.9}
\providecommand{\SimulationASixDesignBLSCV}{58.4}
\providecommand{\SimulationASixDesignBLCV}{58.1}
\providecommand{\SimulationASixDesignBLukic}{74.1}
\providecommand{\SimulationASevenDesignALSCV}{8.7}
\providecommand{\SimulationASevenDesignALCV}{8.4}
\providecommand{\SimulationASevenDesignALukic}{9.4}
\providecommand{\SimulationASevenDesignBLSCV}{14.0}
\providecommand{\SimulationASevenDesignBLCV}{13.6}
\providecommand{\SimulationASevenDesignBLukic}{16.9}
\providecommand{\SimulationAEightDesignALSCV}{11.8}
\providecommand{\SimulationAEightDesignALCV}{11.7}
\providecommand{\SimulationAEightDesignALukic}{8.9}
\providecommand{\SimulationAEightDesignBLSCV}{25.7}
\providecommand{\SimulationAEightDesignBLCV}{25.3}
\providecommand{\SimulationAEightDesignBLukic}{16.7}
\providecommand{\SimulationANineDesignALSCV}{46.9}
\providecommand{\SimulationANineDesignALCV}{47.4}
\providecommand{\SimulationANineDesignALukic}{40.0}
\providecommand{\SimulationANineDesignBLSCV}{84.9}
\providecommand{\SimulationANineDesignBLCV}{85.0}
\providecommand{\SimulationANineDesignBLukic}{92.1}
\providecommand{\SimulationATenDesignALSCV}{17.6}
\providecommand{\SimulationATenDesignALCV}{18.0}
\providecommand{\SimulationATenDesignALukic}{15.6}
\providecommand{\SimulationATenDesignBLSCV}{43.1}
\providecommand{\SimulationATenDesignBLCV}{46.4}
\providecommand{\SimulationATenDesignBLukic}{34.3}
\providecommand{\SimulationAElevenDesignALSCV}{53.5}
\providecommand{\SimulationAElevenDesignALCV}{51.7}
\providecommand{\SimulationAElevenDesignALukic}{21.3}
\providecommand{\SimulationAElevenDesignBLSCV}{95.5}
\providecommand{\SimulationAElevenDesignBLCV}{95.1}
\providecommand{\SimulationAElevenDesignBLukic}{51.6}
\providecommand{\SimulationATwelveDesignALSCV}{9.9}
\providecommand{\SimulationATwelveDesignALCV}{10.8}
\providecommand{\SimulationATwelveDesignALukic}{5.1}
\providecommand{\SimulationATwelveDesignBLSCV}{18.9}
\providecommand{\SimulationATwelveDesignBLCV}{19.9}
\providecommand{\SimulationATwelveDesignBLukic}{5.1}
\providecommand{\SimulationAThirteenDesignALSCV}{41.0}
\providecommand{\SimulationAThirteenDesignALCV}{38.9}
\providecommand{\SimulationAThirteenDesignALukic}{13.6}
\providecommand{\SimulationAThirteenDesignBLSCV}{87.0}
\providecommand{\SimulationAThirteenDesignBLCV}{86.2}
\providecommand{\SimulationAThirteenDesignBLukic}{31.1}
\providecommand{\SimulationAFourteenDesignALSCV}{6.1}
\providecommand{\SimulationAFourteenDesignALCV}{6.3}
\providecommand{\SimulationAFourteenDesignALukic}{5.6}
\providecommand{\SimulationAFourteenDesignBLSCV}{7.6}
\providecommand{\SimulationAFourteenDesignBLCV}{7.7}
\providecommand{\SimulationAFourteenDesignBLukic}{6.7}
\providecommand{\SimulationAFifteenDesignALSCV}{17.1}
\providecommand{\SimulationAFifteenDesignALCV}{17.8}
\providecommand{\SimulationAFifteenDesignALukic}{5.9}
\providecommand{\SimulationAFifteenDesignBLSCV}{42.1}
\providecommand{\SimulationAFifteenDesignBLCV}{47.3}
\providecommand{\SimulationAFifteenDesignBLukic}{8.3}
\providecommand{\SimulationASixteenDesignALSCV}{5.5}
\providecommand{\SimulationASixteenDesignALCV}{5.5}
\providecommand{\SimulationASixteenDesignALukic}{5.7}
\providecommand{\SimulationASixteenDesignBLSCV}{6.6}
\providecommand{\SimulationASixteenDesignBLCV}{6.4}
\providecommand{\SimulationASixteenDesignBLukic}{6.5}
\providecommand{\SimulationASeventeenDesignALSCV}{64.8}
\providecommand{\SimulationASeventeenDesignALCV}{65.5}
\providecommand{\SimulationASeventeenDesignALukic}{57.1}
\providecommand{\SimulationASeventeenDesignBLSCV}{95.4}
\providecommand{\SimulationASeventeenDesignBLCV}{95.3}
\providecommand{\SimulationASeventeenDesignBLukic}{99.2}

\section{Simulation study}\label{sec:power.study}

This section presents a simulation study in dimension $d = 2$ comparing the empirical Type I error rates and powers of the proposed Wishart-KDE test, implemented with pooled LSCV or LCV bandwidth selection, with those of the empirical Laplace-transform test of \citet{Lukic2024Laplace}.

Subsection~\ref{subsec:simulation.competing.test} describes the competing test and the common permutation implementation, Subsection~\ref{subsec:simulation.target.laws} introduces the target laws and configurations, Subsection~\ref{subsec:simulation.design} specifies the Monte Carlo design, and Subsection~\ref{subsec:simulation.numerical.implementation} gives the numerical details. The empirical rejection percentages and selected bandwidths are reported in Subsection~\ref{subsec:simulation.results}, and the main power results are discussed in Subsection~\ref{subsec:simulation.discussion}.

\subsection{Competing test}\label{subsec:simulation.competing.test}

Let $\mathcal{S}_{+}^2$ denote the cone of $2\times 2$ symmetric nonnegative definite matrices. For $U \in \mathcal{S}_{+}^2$, the matrix Laplace transform of a random matrix $\mathfrak{X}$ is $\mathcal{L}_{\mathfrak{X}}(U) = \EE\{\etr(-U\mathfrak{X})\}$. This transform determines the distribution of $\mathfrak{X}$, so the empirical Laplace-transform test of \citet{Lukic2024Laplace} compares the two samples through empirical versions of their transforms. Specifically, let $\smash{\mathfrak{X}^{(1)}_1,\ldots,\mathfrak{X}^{(1)}_{n_1}\stackrel{\mathrm{iid}}{\sim} f^{(1)}}$ and $\smash{\mathfrak{X}^{(2)}_1,\ldots,\mathfrak{X}^{(2)}_{n_2}\stackrel{\mathrm{iid}}{\sim} f^{(2)}}$ be two independent random samples from $f^{(1)}$ and $f^{(2)}$, respectively. Their empirical matrix Laplace transforms are
\[
\widehat{\mathcal{L}}_1(U) = \frac{1}{n_1}\sum_{i=1}^{n_1}\etr(-U\mathfrak{X}_i^{(1)}), \qquad \widehat{\mathcal{L}}_2(U) = \frac{1}{n_2}\sum_{j=1}^{n_2}\etr(-U\mathfrak{X}_j^{(2)}).
\]

For Luki\'c's test, let $\mu_{\nu,\Sigma,\Omega}$ denote a noncentral Wishart probability measure, and let $\mathrm{Id}_2$ denote the $2 \times 2$ identity matrix. The weighting measure is characterized here by the Laplace transform
\[
\ell_{\nu,\Sigma,\Omega}(A) \leqdef \int_{\mathcal{S}_{+}^2}\etr(-UA) \, \mu_{\nu,\Sigma,\Omega}(\rd U) = \frac{\etr\big[-2A\{\mathrm{Id}_2 + 2\Sigma A\}^{-1}\Omega\big]}{|\mathrm{Id}_2 + 2\Sigma A|^{\nu/2}}, \qquad A \in \mathcal{S}_{+}^2.
\]
Here $\nu$ is the degrees-of-freedom parameter, $\Sigma$ is the scale matrix, and $\Omega$ is the noncentrality matrix of the weighting measure. (Equivalently, if $\tilde{\nu}$ denotes the parameter in the representation $\mathrm{NCW}(2\tilde{\nu},\Sigma,\Omega)$ used in Eq.~(3) of \citet{Lukic2024Laplace}, then $\tilde{\nu} = \nu/2$.) These are tuning parameters of the discrepancy and are not parameters of either target distribution. Luki\'c's discrepancy is the integrated squared difference
\begin{equation}\label{eq:Lukic.test}
D_{\nu,\Sigma,\Omega}(n_1,n_2) \leqdef \int_{\mathcal{S}_{+}^2}\{\widehat{\mathcal{L}}_1(U) - \widehat{\mathcal{L}}_2(U)\}^2 \, \mu_{\nu,\Sigma,\Omega}(\rd U).
\end{equation}
This empirical Laplace-transform procedure is referred to as Luki\'c's test below; large values indicate disagreement between the empirical transforms and hence evidence against $H_0 : f^{(1)} = f^{(2)}$.

Because $\etr(-UX)\etr(-UY) = \etr[-U(X + Y)]$, expanding the square produces the pairwise kernel $q_{\nu,\Sigma,\Omega}(X,Y) = \ell_{\nu,\Sigma,\Omega}(X + Y)$. The power study of \citet{Lukic2024Laplace} fixes $\Sigma = \mathrm{Id}_d$ and reports $\nu \in \{1,2,5\}$ and $\Omega \in \{\mathrm{Id}_d,2\mathrm{Id}_d\}$, where $\mathrm{Id}_d$ is the $d \times d$ identity matrix. Accordingly, the present comparison uses $(\nu,\Sigma,\Omega) = (1,\mathrm{Id}_2,\mathrm{Id}_2)$, the combination favored among the settings examined there; see \citet[p.~9347]{Lukic2024Laplace}. For $d = 2$, this full-rank choice of $\Omega$ is admissible because $\nu = d - 1$ lies at the boundary of the noncentral Wishart existence range that permits unrestricted noncentrality. The resulting pairwise kernel is
\[
q(X,Y) = \frac{\etr\big[-2(X + Y)\{\mathrm{Id}_2 + 2(X + Y)\}^{-1}\mathrm{Id}_2\big]}{|\mathrm{Id}_2 + 2(X + Y)|^{1/2}}, \qquad X,Y \in \mathcal{S}_{++}^2.
\]
Setting $n = n_1 + n_2$, the implementation rescales Luki\'c's discrepancy as follows:
\[
\frac{n_1n_2}{n}D_{1,\mathrm{Id}_2,\mathrm{Id}_2}(n_1,n_2) = \frac{n_1n_2}{n}\left\{\frac{1}{n_1^2}\sum_{i,j=1}^{n_1}q(\mathfrak{X}_i^{(1)},\mathfrak{X}_j^{(1)}) + \frac{1}{n_2^2}\sum_{i,j=1}^{n_2}q(\mathfrak{X}_i^{(2)},\mathfrak{X}_j^{(2)}) - \frac{2}{n_1n_2}\sum_{i=1}^{n_1}\sum_{j=1}^{n_2}q(\mathfrak{X}_i^{(1)},\mathfrak{X}_j^{(2)})\right\}.
\]
The prefactor $n_1n_2/n$ is positive and constant over label permutations, so it does not affect the permutation $p$-value.

For either test, let $G$ be the pooled-sample Gram matrix for the Wishart-overlap kernel or $q$. If $\bb{w}$ has entries $1/n_1$ for the first sample and $-1/n_2$ for the second, the observed statistic is proportional to
\[
\frac{n_1n_2}{n}\bb{w}^{\top}G\bb{w}.
\]
Each random relabeling changes only $\bb{w}$, so the three Gram matrices (for the LSCV Wishart-overlap kernel, the LCV Wishart-overlap kernel, and the pairwise kernel $q$) are computed once per replication and the same allocation vectors are used for all three implementations. Because both test statistics are mathematically defined as $V$-statistics rather than $U$-statistics, the diagonal terms of the Gram matrices (representing the self-evaluation of each observation) are retained in the final summations. The parameters $(\nu,\Sigma,\Omega)$ of $q$ remain fixed in every configuration, while the two Wishart bandwidth rules are applied in the same way to the pooled data in every configuration. No tuning value is chosen after inspecting the rejection results for a configuration.

\subsection{Target laws and configurations}\label{subsec:simulation.target.laws}

The design combines Wishart and inverse Wishart laws with sample covariance laws, finite mixtures, randomly rotated Wishart laws, and orthogonally invariant spectral laws. The sample covariance laws are generated from Gaussian, Student $t$, uniform, and fixed-radius parent distributions, and the alternatives are chosen to isolate differences in scale, concentration, orientation, boundary proximity, tail weight, innovation geometry, mixture structure, and spectral dependence. Write $\mathrm{W}_\beta(\Sigma) \equiv \mathrm{Wishart}_2(\beta,\Sigma/\beta)$ for a law with mean $\Sigma$, and write $\mathrm{IW}_\beta \equiv \mathrm{InvWishart}_2\{\beta,(\beta - 3)\mathrm{Id}_2\}$ under the standard inverse Wishart parameterization for a law with mean $\mathrm{Id}_2$. The matrices used to specify the geometric alternatives are
\[
D = \mathrm{diag}(0.5, 1.5), \qquad Q = \frac{1}{\sqrt{2}}\begin{pmatrix}1 & -1 \\ 1 & 1\end{pmatrix}, \qquad R_{\rho} = \begin{pmatrix}1 & \rho \\ \rho & 1\end{pmatrix}.
\]
To define covariance-matrix targets, let $\bb{U}_{a,1},\ldots,\bb{U}_{a,8}$ be iid vectors of the form $\smash{\bb{U}_{a,j} = \{(a - 2)/V_j\}^{1/2}\bb{Z}_j}$, where $\bb{Z}_j \sim \mathcal{N}_2(\bb{0}_2,\mathrm{Id}_2)$, $V_j \sim \chi_a^2$, and all variables are independent, and let $\mathfrak{C}_a$ denote their unbiased sample covariance matrix.

Also let $\mathfrak{C}_{G}$, $\mathfrak{C}_{U}$, and $\mathfrak{C}_{S}$ be the unbiased sample covariance matrices of eight iid standard Gaussian vectors, eight iid vectors with independent $\mathrm{Unif}(-\sqrt{3},\sqrt{3})$ coordinates, and eight iid vectors $\bb{S}_j = \sqrt{2}(\cos\Theta_j,\sin\Theta_j)^{\top}$ with $\Theta_j \sim \mathrm{Unif}(0,2\pi)$, respectively. All three parent distributions have covariance $\mathrm{Id}_2$, and $\mathfrak{C}_{G} \sim \mathrm{W}_7(\mathrm{Id}_2)$. Since $8 > d = 2$, the covariance matrices are positive definite almost surely.

Set $D_{-} = \mathrm{diag}(0.25, 1.75)$ and $D_{+} = \mathrm{diag}(1.75, 0.25)$, and let $\mathfrak{M}$ and $\mathfrak{H}$ follow the mixture laws
\[
\mathfrak{M} \sim \frac{1}{2}\mathrm{W}_5(D_{-}) + \frac{1}{2}\mathrm{W}_5(D_{+}), \qquad \mathfrak{H} \sim \frac{5}{9}\mathrm{W}_3(\mathrm{Id}_2) + \frac{4}{9}\mathrm{W}_{30}(\mathrm{Id}_2),
\]
respectively. For $\eta \in [0,1)$, let $(\zeta_{1,\eta},\zeta_{2,\eta})$ be a bivariate standard normal vector with correlation $\eta$, let $L_{j,\eta} = \exp(\sigma \zeta_{j,\eta} - \sigma^2/2)$, $j \in \{1,2\}$, where $\sigma = \{\log(1.4)\}^{1/2}$, and let $Q_{\theta}$ denote rotation through the angle $\theta$, so $Q = Q_{\pi/4}$. If $\Theta \sim \mathrm{Unif}(0,2\pi)$ is independent of the two Gaussian variables, the law of $\mathfrak{L}_{\eta} \equiv Q_{\Theta}\mathrm{diag}(L_{1,\eta},L_{2,\eta})Q_{\Theta}^{\top}$ is orthogonally invariant, is supported on $\mathcal{S}_{++}^2$, and has expectation $\mathrm{Id}_2$.

To define two laws with the same eigenvalue distribution but different orientation distributions, let $\mathfrak{X}_D,\mathfrak{X}_U \smash{\stackrel{\mathrm{iid}}{\sim}} \mathrm{W}_5(D)$, let $\Theta_D$ be equally likely to be $0$ or $\pi/2$, and let $\Theta_U \sim \mathrm{Unif}(0,2\pi)$, with all these variables independent. Define $\mathfrak{O}_D \equiv Q_{\Theta_D}\mathfrak{X}_D Q_{\Theta_D}^{\top}$ and $\mathfrak{O}_U \equiv Q_{\Theta_U}\mathfrak{X}_U Q_{\Theta_U}^{\top}$. The laws of both $\mathfrak{O}_D$ and $\mathfrak{O}_U$ have expectation $\mathrm{Id}_2$. Finally, let $\mathfrak{P}$ and $\mathfrak{R}$ follow the mean-matched mixture laws
\[
\mathfrak{P} \sim \frac{9}{10}\mathrm{W}_5(0.9\mathrm{Id}_2) + \frac{1}{10}\mathrm{W}_5(1.9\mathrm{Id}_2), \qquad \mathfrak{R} \sim \frac{1}{2}\mathrm{W}_5(R_{0.8}) + \frac{1}{2}\mathrm{W}_5(R_{-0.8}),
\]
respectively. The three null configurations and seventeen alternatives are
\[
\begin{array}{lll}
\mathrm{N}_1: \mathrm{W}_5(\mathrm{Id}_2) \text{ versus } \mathrm{W}_5(\mathrm{Id}_2), & \qquad & \mathrm{N}_2: \mathrm{IW}_8 \text{ versus } \mathrm{IW}_8, \\
\mathrm{N}_3: \mathfrak{C}_3 \text{ versus } \mathfrak{C}_3, && \mathrm{A}_1: \mathrm{W}_5(\mathrm{Id}_2) \text{ versus } \mathrm{W}_5(1.35\mathrm{Id}_2), \\
\mathrm{A}_2: \mathrm{W}_5(\mathrm{Id}_2) \text{ versus } \mathrm{IW}_8, && \mathrm{A}_3: \mathrm{W}_5(\mathrm{Id}_2) \text{ versus } \mathrm{W}_{10}(\mathrm{Id}_2), \\
\mathrm{A}_4: \mathrm{W}_5(D) \text{ versus } \mathrm{W}_5(QDQ^{\top}), && \mathrm{A}_5: \mathrm{W}_5(R_{0.80}) \text{ versus } \mathrm{W}_5(R_{0.90}), \\
\mathrm{A}_6: \mathfrak{C}_3 \text{ versus } \mathfrak{C}_5, && \mathrm{A}_7: \mathfrak{C}_5 \text{ versus } \mathfrak{C}_{10}, \\
\mathrm{A}_8: \mathfrak{C}_{G} \text{ versus } \mathfrak{C}_{U}, && \mathrm{A}_9: \mathrm{W}_5(\mathrm{Id}_2) \text{ versus } \mathfrak{M}, \\
\mathrm{A}_{10}: \mathrm{IW}_8 \text{ versus } \mathrm{IW}_{12}, && \mathrm{A}_{11}: \mathfrak{C}_{G} \text{ versus } \mathfrak{C}_{S}, \\
\mathrm{A}_{12}: \mathrm{W}_5(\mathrm{Id}_2) \text{ versus } \mathfrak{H}, && \mathrm{A}_{13}: \mathrm{W}_5(\mathrm{Id}_2) \text{ versus } \mathfrak{L}_0, \\
\mathrm{A}_{14}: \mathfrak{O}_D \text{ versus } \mathfrak{O}_U, && \mathrm{A}_{15}: \mathfrak{L}_0 \text{ versus } \mathfrak{L}_{0.8}, \\
\mathrm{A}_{16}: \mathrm{W}_5(\mathrm{Id}_2) \text{ versus } \mathfrak{P}, && \mathrm{A}_{17}: \mathrm{W}_5(\mathrm{Id}_2) \text{ versus } \mathfrak{R}.
\end{array}
\]
Several comparisons remove simple first-moment explanations. Mean matching holds in $\mathrm{A}_2$, $\mathrm{A}_3$, and $\mathrm{A}_6$ through $\mathrm{A}_{17}$, although the matching in $\mathrm{A}_2$ does not extend to second moments. Configuration $\mathrm{A}_4$ preserves the scale eigenvalues and applies a fixed rotation. In $\mathrm{A}_6$, the $t_3$ parent has no fourth moment and supplies a heavy-tailed case, whereas both parents in $\mathrm{A}_7$ have finite fourth moments. Configurations $\mathrm{A}_8$ and $\mathrm{A}_{11}$ compare covariance laws generated from parents having the same covariance but different distributions. Configuration $\mathrm{A}_{10}$ isolates concentration within the inverse Wishart family, while $\mathrm{A}_{13}$ compares the baseline Wishart law with an orthogonally invariant spectral law. Configuration $\mathrm{A}_{14}$ keeps the eigenvalue distribution fixed and changes the distribution of orientations. Configuration $\mathrm{A}_{15}$ preserves the two lognormal marginal distributions and changes their dependence, $\mathrm{A}_{16}$ introduces sparse scale contamination, and $\mathrm{A}_{17}$ mixes equal and opposite correlation structures.

Configuration $\mathrm{A}_{12}$ is more stringently matched. Under $\mathrm{W}_\beta(\mathrm{Id}_2)$, the entrywise covariance tensor is $\Cov(X_{ij},X_{k\ell}) = \beta^{-1}(\delta_{ik}\delta_{j\ell} + \delta_{i\ell}\delta_{jk})$. Since $(5/9)/3 + (4/9)/30 = 1/5$, the two laws in $\mathrm{A}_{12}$ have the same matrix expectation and the same complete entrywise covariance tensor, so their difference begins at higher order. The three nulls mirror the diversity of the alternatives by covering a regular Wishart law ($\mathrm{N}_1$), a skewed inverse Wishart law ($\mathrm{N}_2$), and a heavy-tailed covariance law generated by a $t_3$ parent without a fourth moment ($\mathrm{N}_3$).

\subsection{Simulation design}\label{subsec:simulation.design}

Each of the \SimulationConfigurations{} configurations is generated independently \SimulationReplications{} times for the sample-size pairs \SimulationSampleDesigns. Every procedure is tested at level $\alpha = 0.05$ using the same $B = \SimulationPermutations$ random label permutations and the add-one $p$-value in \eqref{eq:add.one.p.value} of Remark~\ref{rem:permutation.calibration}. Sharing the labelings reduces extraneous Monte Carlo variation in the comparison. Under $H_0$, exchangeability makes this calibration finite-sample valid for both discrepancies.

\subsection{Numerical implementation}\label{subsec:simulation.numerical.implementation}

Within each replication, the Wishart bandwidths are selected anew from the pooled observations using the leave-one-out LCV rule and the LSCV rule of \citet{BelzileGenestOuimetRichards2025WishartKDE}, setting the time-series leave-out window parameter to $h = 1$ to match our iid setting. The pooled sample is iid under $H_0$ but not generally under an alternative; in the latter case, the two criteria are used as practical data-driven rules. Each selected bandwidth is held fixed across all label permutations. This is natural because relabeling leaves the pooled sample unchanged; running the pooled bandwidth optimization again during each permutation would be redundant, while estimating separate groupwise bandwidths for each permutation would not only alter the theoretical test statistic being studied, but also significantly increase the computational cost.

The proposed Wishart-KDE test statistic is evaluated from the closed-form overlap in Proposition~\ref{prop:test.statistic}. Determinants are computed by Cholesky decomposition on the log scale, and the overlap Gram matrix is rescaled by one positive factor common to the observed and permuted statistics, which leaves the permutation $p$-value unchanged. All \SimulationBandwidthFits{} bandwidth optimizations (\SimulationConfigurations{} configurations $\times$ \SimulationSampleDesigns{} sample-size pairs $\times$ \SimulationReplications{} replications $\times$ 2 bandwidth selectors) were completed successfully, and no numerical failure occurred. The $\mathsf{R}$ code flags bandwidths below $10^{-7}$ or above $10$ for diagnostic purposes, and any flagged finite value is retained rather than capped.

\subsection{Results}\label{subsec:simulation.results}

Table~\ref{tab:power.study} reports the empirical rejection percentages. Across the three null distributions, the empirical Type I error rates range from \SimulationNullMinimum\% to \SimulationNullMaximum\%. They remain close to the nominal $5\%$ level, including under the skewed inverse Wishart and heavy-tailed $t_3$ covariance nulls. With \SimulationReplications{} replications, the largest Monte Carlo standard error is about $0.4$ percentage points, and the standard error near the nominal level is about $0.2$ percentage points.

Across the whole table, the two bandwidth implementations differ by at most \SimulationMaximumBandwidthRateDifference{} percentage points. In every alternative and sample-size pair except $\mathrm{A}_{16}$ at $(n_1,n_2) = (50,50)$, the LSCV and LCV rejection percentages are either both higher or both lower than Luki\'c's rejection percentage. In that cell, LSCV is slightly higher and LCV slightly lower than Luki\'c's test. The displayed power is higher at $(n_1,n_2) = (50,50)$ in every method and alternative combination except for Luki\'c's test in $\mathrm{A}_{12}$, for which it is \SimulationATwelveDesignBLukic\% in both sample-size pairs.

Table~\ref{tab:bandwidth.study} reports the median selected bandwidth in each cell. For both selectors, the median bandwidth decreases with sample size in every configuration. Across all cells, the LSCV medians range from $\SimulationLSCVMedianMinimum$ to $\SimulationLSCVMedianMaximum$ and the LCV medians range from $\SimulationLCVMedianMinimum$ to $\SimulationLCVMedianMaximum$, but neither selector is uniformly larger.

\begin{table}[!ht]
\caption{Empirical rejection percentages at the $5\%$ level based on \SimulationReplications{} Monte Carlo replications and \SimulationPermutations{} common random label permutations. Within each sample-size block, LSCV and LCV denote the two pooled Wishart-KDE bandwidth rules, and Luki\'c denotes the Laplace-transform test in \eqref{eq:Lukic.test}. \SimulationBoldCaptionSentence}
\label{tab:power.study}
\centering
\renewcommand{\arraystretch}{0.6} % Default is 1.0. Numbers < 1 compress it.
\begin{tabular}{clcccccc}
\toprule
& & \multicolumn{3}{c}{$(20,20)$} & \multicolumn{3}{c}{$(50,50)$} \\
\cmidrule(lr){3-5}\cmidrule(lr){6-8}
Configuration & Feature & LSCV & LCV & Luki\'c & LSCV & LCV & Luki\'c \\
\midrule
$\mathrm{N}_{1}$ & Wishart null & 5.3 & 5.1 & 5.0 & 4.9 & 4.9 & 4.7 \\
$\mathrm{N}_{2}$ & Inverse Wishart null & 5.1 & 5.1 & 5.2 & 5.2 & 5.2 & 5.0 \\
$\mathrm{N}_{3}$ & Heavy-tail covariance null & 5.1 & 5.1 & 5.2 & 5.1 & 5.1 & 5.0 \\
\midrule
$\mathrm{A}_{1}$ & Scale & 22.5 & 20.9 & \textbf{38.0} & 43.5 & 40.8 & \textbf{77.0} \\
$\mathrm{A}_{2}$ & Mean-matched family & \textbf{28.9} & 28.6 & 5.9 & 72.5 & \textbf{72.9} & 9.8 \\
$\mathrm{A}_{3}$ & Concentration & \textbf{26.4} & 25.8 & 19.4 & \textbf{63.0} & \textbf{63.0} & 47.8 \\
$\mathrm{A}_{4}$ & Orientation & 54.1 & 54.7 & \textbf{94.3} & 88.5 & 88.2 & \textbf{100.0} \\
$\mathrm{A}_{5}$ & Near boundary & 37.0 & 35.5 & \textbf{39.2} & 67.3 & 65.7 & \textbf{87.6} \\
$\mathrm{A}_{6}$ & Extreme tail weight & 29.1 & 27.9 & \textbf{35.9} & 58.4 & 58.1 & \textbf{74.1} \\
$\mathrm{A}_{7}$ & Finite fourth moments & 8.7 & 8.4 & \textbf{9.4} & 14.0 & 13.6 & \textbf{16.9} \\
$\mathrm{A}_{8}$ & Gaussian versus uniform & \textbf{11.8} & 11.7 & 8.9 & \textbf{25.7} & 25.3 & 16.7 \\
$\mathrm{A}_{9}$ & Latent mixture structure & 46.9 & \textbf{47.4} & 40.0 & 84.9 & 85.0 & \textbf{92.1} \\
$\mathrm{A}_{10}$ & Inverse Wishart concentration & 17.6 & \textbf{18.0} & 15.6 & 43.1 & \textbf{46.4} & 34.3 \\
$\mathrm{A}_{11}$ & Fixed-radius innovations & \textbf{53.5} & 51.7 & 21.3 & \textbf{95.5} & 95.1 & 51.6 \\
$\mathrm{A}_{12}$ & Matched concentration mixture & 9.9 & \textbf{10.8} & 5.1 & 18.9 & \textbf{19.9} & 5.1 \\
$\mathrm{A}_{13}$ & Spectral lognormal & \textbf{41.0} & 38.9 & 13.6 & \textbf{87.0} & 86.2 & 31.1 \\
$\mathrm{A}_{14}$ & Orientation distribution & 6.1 & \textbf{6.3} & 5.6 & 7.6 & \textbf{7.7} & 6.7 \\
$\mathrm{A}_{15}$ & Eigenvalue dependence & 17.1 & \textbf{17.8} & 5.9 & 42.1 & \textbf{47.3} & 8.3 \\
$\mathrm{A}_{16}$ & Sparse scale contamination & 5.5 & 5.5 & \textbf{5.7} & \textbf{6.7} & 6.4 & 6.5 \\
$\mathrm{A}_{17}$ & Correlation mixture & 64.8 & \textbf{65.5} & 57.1 & 95.4 & 95.3 & \textbf{99.2} \\
\bottomrule
\end{tabular}

\end{table}

\begin{table}[!ht]
\caption{Median pooled Wishart bandwidths across the \SimulationReplications{} Monte Carlo replications.}
\label{tab:bandwidth.study}
\centering
\renewcommand{\arraystretch}{0.6} % Default is 1.0. Numbers < 1 compress it.
\begin{tabular}{ccccc}
\toprule
& \multicolumn{2}{c}{$(20,20)$} & \multicolumn{2}{c}{$(50,50)$} \\
Configuration & LSCV & LCV & LSCV & LCV \\
\midrule
$\mathrm{N}_{1}$ & 0.177 & 0.132 & 0.128 & 0.100 \\
$\mathrm{N}_{2}$ & 0.053 & 0.066 & 0.037 & 0.052 \\
$\mathrm{N}_{3}$ & 0.163 & 0.155 & 0.115 & 0.126 \\
\midrule
$\mathrm{A}_{1}$ & 0.182 & 0.135 & 0.132 & 0.103 \\
$\mathrm{A}_{2}$ & 0.100 & 0.111 & 0.071 & 0.090 \\
$\mathrm{A}_{3}$ & 0.111 & 0.103 & 0.079 & 0.080 \\
$\mathrm{A}_{4}$ & 0.201 & 0.150 & 0.146 & 0.113 \\
$\mathrm{A}_{5}$ & 0.188 & 0.140 & 0.137 & 0.107 \\
$\mathrm{A}_{6}$ & 0.164 & 0.139 & 0.117 & 0.111 \\
$\mathrm{A}_{7}$ & 0.127 & 0.108 & 0.090 & 0.084 \\
$\mathrm{A}_{8}$ & 0.078 & 0.073 & 0.055 & 0.055 \\
$\mathrm{A}_{9}$ & 0.251 & 0.172 & 0.182 & 0.128 \\
$\mathrm{A}_{10}$ & 0.046 & 0.056 & 0.033 & 0.045 \\
$\mathrm{A}_{11}$ & 0.055 & 0.066 & 0.035 & 0.049 \\
$\mathrm{A}_{12}$ & 0.255 & 0.191 & 0.202 & 0.155 \\
$\mathrm{A}_{13}$ & 0.085 & 0.102 & 0.055 & 0.081 \\
$\mathrm{A}_{14}$ & 0.225 & 0.168 & 0.163 & 0.126 \\
$\mathrm{A}_{15}$ & 0.020 & 0.036 & 0.011 & 0.027 \\
$\mathrm{A}_{16}$ & 0.179 & 0.136 & 0.130 & 0.104 \\
$\mathrm{A}_{17}$ & 0.267 & 0.178 & 0.192 & 0.131 \\
\bottomrule
\end{tabular}

\end{table}

\subsection{Discussion}\label{subsec:simulation.discussion}

No implementation dominates uniformly. At $(n_1,n_2) = (50,50)$, a Wishart-KDE implementation attains the largest displayed power in ten alternatives, $\mathrm{A}_2$, $\mathrm{A}_3$, $\mathrm{A}_8$, and $\mathrm{A}_{10}$ through $\mathrm{A}_{16}$, while Luki\'c's test leads in $\mathrm{A}_1$, $\mathrm{A}_4$ through $\mathrm{A}_7$, $\mathrm{A}_9$, and $\mathrm{A}_{17}$. At $(n_1,n_2) = (20,20)$, a Wishart-KDE implementation attains the largest displayed power in eleven alternatives, $\mathrm{A}_2$, $\mathrm{A}_3$, $\mathrm{A}_8$ through $\mathrm{A}_{15}$, and $\mathrm{A}_{17}$, while Luki\'c's test leads in $\mathrm{A}_1$, $\mathrm{A}_4$ through $\mathrm{A}_7$, and $\mathrm{A}_{16}$.

Several patterns persist across the two sample-size pairs. Luki\'c's statistic is markedly stronger for the scale and fixed-rotation alternatives $\mathrm{A}_1$ and $\mathrm{A}_4$, and it also leads in $\mathrm{A}_5$ through $\mathrm{A}_7$ for both sample-size pairs. A Wishart-KDE implementation leads in $\mathrm{A}_2$, $\mathrm{A}_3$, $\mathrm{A}_8$, and $\mathrm{A}_{10}$ through $\mathrm{A}_{15}$ for both sample-size pairs. The ordering reverses with sample size in $\mathrm{A}_9$ and $\mathrm{A}_{17}$, where a Wishart-KDE implementation leads at $(n_1,n_2) = (20,20)$ and Luki\'c's test leads at $(n_1,n_2) = (50,50)$. In $\mathrm{A}_{16}$, Luki\'c's test is largest at $(n_1,n_2) = (20,20)$, while LSCV is slightly higher and LCV slightly lower than Luki\'c's test at $(n_1,n_2) = (50,50)$.

The largest advantages for the Wishart discrepancy occur in several mean-matched or higher-order comparisons. At $(n_1,n_2) = (50,50)$, the largest Wishart-KDE power and the power of Luki\'c's test are \SimulationATwoDesignBLCV\% and \SimulationATwoDesignBLukic\% in the cross-family alternative $\mathrm{A}_2$, \SimulationAElevenDesignBLSCV\% and \SimulationAElevenDesignBLukic\% in the fixed-radius innovation alternative $\mathrm{A}_{11}$, \SimulationAThirteenDesignBLSCV\% and \SimulationAThirteenDesignBLukic\% in the spectral-lognormal alternative $\mathrm{A}_{13}$, and \SimulationAFifteenDesignBLCV\% and \SimulationAFifteenDesignBLukic\% in the eigenvalue-dependence alternative $\mathrm{A}_{15}$. In the exactly mean-and-covariance-matched mixture $\mathrm{A}_{12}$, the corresponding powers are \SimulationATwelveDesignBLCV\% and \SimulationATwelveDesignBLukic\%, providing empirical evidence that the proposed test is sensitive to differences not captured by the mean and entrywise covariance tensor, although its absolute power remains modest. All three powers remain below $8\%$ in $\mathrm{A}_{14}$ and $\mathrm{A}_{16}$ for both sample-size pairs. Conversely, Luki\'c's test reaches \SimulationAOneDesignBLukic\% in $\mathrm{A}_1$ and \SimulationAFourDesignBLukic\% in $\mathrm{A}_4$ at $(n_1,n_2) = (50,50)$, compared with Wishart-KDE maxima of \SimulationAOneDesignBLSCV\% and \SimulationAFourDesignBLSCV\%. The relative performance therefore depends materially on the distributional difference, which supports the proposed procedure as a complementary alternative rather than a uniformly more powerful replacement.

\section{Conclusion and outlook}\label{sec:outlook}

The proposed Wishart-KDE test statistic provides a kernel-based alternative to the empirical Laplace-transform test of \citet{Lukic2024Laplace} for $\mathcal{S}_{++}^d$-valued data. More broadly, it complements the Hankel-transform test of \citet{LukicMilosevic2024AISM}, though we recall that the latter tests orthogonal equivalence in distribution rather than equality of unrestricted positive definite matrix distributions. Its closed-form kernel overlap yields a two-sample $V$-statistic that can be evaluated without numerical integration, and the asymptotic null distribution is derived under both a common shrinking bandwidth and a fixed bandwidth. Because it retains information about relative orientation when $d \geq 2$, it can detect alternatives that differ through eigenvector structure and not merely through eigenvalues, which is important when principal directions carry scientific meaning. In the simulation comparison with the empirical Laplace-transform test, neither test dominates across the examined alternatives. The Wishart-KDE test has particularly large power advantages for the mean-matched cross-family ($\mathrm{A}_2$), fixed-radius innovation ($\mathrm{A}_{11}$), spectral-lognormal ($\mathrm{A}_{13}$), and eigenvalue-dependence ($\mathrm{A}_{15}$) comparisons, and it also has higher power for the exactly mean-and-covariance-matched mixture ($\mathrm{A}_{12}$), although its absolute power in that configuration remains modest. The empirical Laplace-transform test is substantially more powerful for the scale ($\mathrm{A}_1$) and fixed-rotation ($\mathrm{A}_4$) comparisons, while the ordering changes with sample size for the latent-mixture ($\mathrm{A}_9$) and correlation-mixture ($\mathrm{A}_{17}$) alternatives. More broadly, the construction extends classical $L^2$-type two-sample ideas to $\mathcal{S}_{++}^d$ while respecting positive definiteness and avoiding spill-over beyond the cone.

Several extensions deserve investigation. Testing-oriented bandwidth selectors may improve on the practical pooled cross-validation defaults, and theory for unequal bandwidths $b_1 \neq b_2$ or fully data-driven bandwidths remains to be developed. A fuller analysis under fixed alternatives would extend Remark~\ref{rem:fixed.alternatives}, while contiguous alternatives would permit local-power and efficiency comparisons between the proposed test and the empirical Laplace-transform test of \citet{Lukic2024Laplace}. Bootstrap or spectral calibration could also be compared with permutation. Extensions to $k$ samples, change-point detection, and dependent positive definite matrix observations are also natural. Finally, local diagnostics in the spirit of \citet{Duong2013} could identify where the two smoothed densities differ after a global rejection.

\begin{appendices}

\section{Proofs}\label{app:proofs}

\subsection{Proof of Proposition~\ref{prop:test.statistic}}

Recall the definition of $\delta_{i,\ell}$ and $a_{i,\ell}$ in the statement of the proposition. For $i\in \{1,2\}$, we have
\begin{equation}\label{eq:Donald.3}
\hat{f}_{n_i,b_i}^{(i)}(S) - \sum_{\ell=1}^2 \frac{n_{\ell}}{n} \hat{f}_{n_{\ell},b_{\ell}}^{(\ell)}(S)
= \sum_{\ell=1}^2 \delta_{i,\ell} \hat{f}_{n_{\ell},b_{\ell}}^{(\ell)}(S) - \sum_{\ell=1}^2 \frac{n_{\ell}}{n} \hat{f}_{n_{\ell},b_{\ell}}^{(\ell)}(S)
= \sum_{\ell=1}^2 a_{i,\ell} \hat{f}_{n_{\ell},b_{\ell}}^{(\ell)}(S).
\end{equation}
By combining \eqref{eq:Donald.1}, \eqref{eq:Donald.2} and \eqref{eq:Donald.3}, we deduce
\begin{equation}\label{eq:Donald.4}
T(n_1,n_2)
= \sum_{i=1}^2 n_i \sum_{\ell_1=1}^2 \sum_{\ell_2=1}^2 \frac{a_{i,\ell_1} a_{i,\ell_2}}{n_{\ell_1} n_{\ell_2}} \sum_{i_1=1}^{n_{\ell_1}} \sum_{i_2=1}^{n_{\ell_2}} \int_{\mathcal{S}_{++}^d} K_{\nu(b_{\ell_1},d),b_{\ell_1} S}(\mathfrak{X}_{i_1}^{(\ell_1)}) K_{\nu(b_{\ell_2},d),b_{\ell_2} S}(\mathfrak{X}_{i_2}^{(\ell_2)}) \, \rd S.
\end{equation}

Using the expression for the Wishart density in \eqref{eq:Wishart.kernel}, together with
\[
\nu(b,d) = \frac{1}{b} + d + 1, \qquad \frac{\nu(b,d)}{2} - \frac{d+1}{2} = \frac{1}{2b}, \qquad \frac{d \, \nu(b,d)}{2} = \frac{d}{2b} + r(d),
\]
we obtain
\[
\begin{aligned}
&K_{\nu(b_{\ell_1},d),b_{\ell_1} S}(\mathfrak{X}_{i_1}^{(\ell_1)}) \, K_{\nu(b_{\ell_2},d),b_{\ell_2} S}(\mathfrak{X}_{i_2}^{(\ell_2)}) \\
&\quad= \left\{\prod_{j=1}^2 \frac{|\mathfrak{X}_{i_j}^{(\ell_j)}|^{1/(2 b_{\ell_j})}}{(2 b_{\ell_j})^{d/(2 b_{\ell_j}) + r(d)} \Gamma_d\left(1/(2 b_{\ell_j}) + (d+1)/2\right)}\right\} \times \frac{\etr\big\{-S^{-1} \big(b_{\ell_1}^{-1} \mathfrak{X}_{i_1}^{(\ell_1)} + b_{\ell_2}^{-1} \mathfrak{X}_{i_2}^{(\ell_2)}\big) / 2\big\}}{|S|^{1/(2 b_{\ell_1}) + 1/(2 b_{\ell_2}) + (d+1)}}.
\end{aligned}
\]

For brevity, define
\[
\mathfrak{Y}_{\ell_1,\ell_2} \leqdef b_{\ell_1}^{-1}\mathfrak{X}_{i_1}^{(\ell_1)} + b_{\ell_2}^{-1}\mathfrak{X}_{i_2}^{(\ell_2)} \in \mathcal{S}_{++}^d, \qquad
\alpha_{\ell_1,\ell_2} \leqdef \frac{1}{2 b_{\ell_1}} + \frac{1}{2 b_{\ell_2}} + \frac{d+1}{2}.
\]
By \citet{Olkin1998}, the inversion map $U = S^{-1}$ has Jacobian $\rd U = |S|^{-(d+1)} \rd S$, equivalently $\rd S = |U|^{-(d+1)} \rd U$; see also \citet[p.~272]{Anderson2003} and \citet[p.~26, Eq.~(2.44)]{Gindikin1964}. Therefore,
\[
\int_{\mathcal{S}_{++}^d} \frac{\etr\big(-S^{-1}\mathfrak{Y}_{\ell_1,\ell_2}/2\big)}{|S|^{1/(2 b_{\ell_1}) + 1/(2 b_{\ell_2}) + (d+1)}} \, \rd S
= \int_{\mathcal{S}_{++}^d} \etr\big(-U \mathfrak{Y}_{\ell_1,\ell_2}/2\big) \, |U|^{1/(2 b_{\ell_1}) + 1/(2 b_{\ell_2})} \, \rd U \\
= \int_{\mathcal{S}_{++}^d} \etr\big(-U \mathfrak{Y}_{\ell_1,\ell_2}/2\big) \, |U|^{\alpha_{\ell_1,\ell_2} - (d+1)/2} \, \rd U.
\]
Now, for any $C\in \mathcal{S}_{++}^d$ and any $\alpha$ with $\Re(\alpha) > (d-1)/2$, the identity $\smash{\int_{\mathcal{S}_{++}^d}} \etr(-CU) \, |U|^{\alpha-(d+1)/2} \, \rd U = \Gamma_d(\alpha) \, |C|^{-\alpha}$ follows from \eqref{eq:multivariate.gamma} by the congruence transformation $X = C^{1/2} U C^{1/2}$, whose Jacobian is $\rd X = |C|^{(d+1)/2} \rd U$. Applying it with $C = \mathfrak{Y}_{\ell_1,\ell_2}/2$ and $\alpha = \alpha_{\ell_1,\ell_2}$ yields
\[
\int_{\mathcal{S}_{++}^d} \frac{\etr\big(-S^{-1}\mathfrak{Y}_{\ell_1,\ell_2}/2\big)}{|S|^{1/(2 b_{\ell_1}) + 1/(2 b_{\ell_2}) + (d+1)}} \, \rd S
= \Gamma_d(\alpha_{\ell_1,\ell_2}) \left|\frac{\mathfrak{Y}_{\ell_1,\ell_2}}{2}\right|^{-\alpha_{\ell_1,\ell_2}},
\]
with absolute convergence because $\mathfrak{Y}_{\ell_1,\ell_2}\in \mathcal{S}_{++}^d$ and $\alpha_{\ell_1,\ell_2} > (d-1)/2$.

It follows that
\[
\int_{\mathcal{S}_{++}^d} K_{\nu(b_{\ell_1},d),b_{\ell_1} S}(\mathfrak{X}_{i_1}^{(\ell_1)}) K_{\nu(b_{\ell_2},d),b_{\ell_2} S}(\mathfrak{X}_{i_2}^{(\ell_2)}) \, \rd S
= \left\{\prod_{j=1}^2 \frac{|\mathfrak{X}_{i_j}^{(\ell_j)}|^{1/(2 b_{\ell_j})}}{(2 b_{\ell_j})^{d/(2 b_{\ell_j}) + r(d)} \Gamma_d\left(1/(2 b_{\ell_j}) + (d+1)/2\right)}\right\} \Gamma_d(\alpha_{\ell_1,\ell_2}) \left|\frac{\mathfrak{Y}_{\ell_1,\ell_2}}{2}\right|^{-\alpha_{\ell_1,\ell_2}}.
\]
Since
\[
\left|\frac{\mathfrak{Y}_{\ell_1,\ell_2}}{2}\right|^{-\alpha_{\ell_1,\ell_2}}
= 2^{d\alpha_{\ell_1,\ell_2}} |\mathfrak{Y}_{\ell_1,\ell_2}|^{-\alpha_{\ell_1,\ell_2}}, \qquad
d\alpha_{\ell_1,\ell_2}
= \frac{d}{2 b_{\ell_1}} + \frac{d}{2 b_{\ell_2}} + r(d),
\]
the powers of $2$ simplify to $2^{-r(d)}$. Hence
\begin{equation}\label{eq:kappa.b.closed.form}
\begin{aligned}
&\int_{\mathcal{S}_{++}^d} K_{\nu(b_{\ell_1},d),b_{\ell_1} S}(\mathfrak{X}_{i_1}^{(\ell_1)}) K_{\nu(b_{\ell_2},d),b_{\ell_2} S}(\mathfrak{X}_{i_2}^{(\ell_2)}) \, \rd S \\
&\qquad= \frac{1}{2^{r(d)}} \, b_{\ell_1}^{-d/(2 b_{\ell_1}) - r(d)} \, b_{\ell_2}^{-d/(2 b_{\ell_2}) - r(d)} \, \frac{\Gamma_d\left(1/(2 b_{\ell_1}) + 1/(2 b_{\ell_2}) + (d+1)/2\right)}{\Gamma_d\left(1/(2 b_{\ell_1}) + (d+1)/2\right)\Gamma_d\left(1/(2 b_{\ell_2}) + (d+1)/2\right)} \\
&\hspace{35mm}\times \frac{|\mathfrak{X}_{i_1}^{(\ell_1)}|^{1/(2 b_{\ell_1})} |\mathfrak{X}_{i_2}^{(\ell_2)}|^{1/(2 b_{\ell_2})}}
{|b_{\ell_1}^{-1} \mathfrak{X}_{i_1}^{(\ell_1)} + b_{\ell_2}^{-1} \mathfrak{X}_{i_2}^{(\ell_2)}|^{1/(2 b_{\ell_1}) + 1/(2 b_{\ell_2}) + (d+1)/2}}.
\end{aligned}
\end{equation}
By applying this last formula in \eqref{eq:Donald.4}, we obtain
\[
\begin{aligned}
T(n_1,n_2)
&= \frac{1}{2^{r(d)}} \sum_{i=1}^2 n_i \sum_{\ell_1=1}^2 \sum_{\ell_2=1}^2 \frac{a_{i,\ell_1} a_{i,\ell_2}}{n_{\ell_1} n_{\ell_2}} b_{\ell_1}^{-d/(2 b_{\ell_1}) - r(d)} b_{\ell_2}^{-d/(2 b_{\ell_2}) - r(d)} \frac{\Gamma_d(1/(2 b_{\ell_1}) + 1/(2 b_{\ell_2}) + (d+1)/2)}{\Gamma_d(1/(2 b_{\ell_1}) + (d+1)/2) \Gamma_d(1/(2 b_{\ell_2}) + (d+1)/2)} \\[1mm]
&\hspace{55mm}\times \sum_{i_1=1}^{n_{\ell_1}} \sum_{i_2=1}^{n_{\ell_2}} \frac{|\mathfrak{X}_{i_1}^{(\ell_1)}|^{1/(2 b_{\ell_1})} |\mathfrak{X}_{i_2}^{(\ell_2)}|^{1/(2 b_{\ell_2})}}{|b_{\ell_1}^{-1} \mathfrak{X}_{i_1}^{(\ell_1)} + b_{\ell_2}^{-1} \mathfrak{X}_{i_2}^{(\ell_2)}|^{1/(2 b_{\ell_1}) + 1/(2 b_{\ell_2}) + (d+1)/2}}.
\end{aligned}
\]
This concludes the proof.

\subsection{Proof of Theorem~\ref{thm:null.shrinking.bandwidth}}
\setcounter{pstep}{0}

\noindent The proof is divided into steps to make it more digestible. We first rewrite the statistic $T(n_1,n_2)$ as a quadratic form in the overlap kernel $\kappa_b$ (Step~\ref{step:1}), then Hoeffding-center the kernel and decompose the resulting quadratic form into diagonal and off-diagonal components (Step~\ref{step:2}). We identify the expectation of $T(n_1,n_2)$ under $H_0$ (Step~\ref{step:3}), show that the diagonal contribution is negligible after the relevant scaling (Step~\ref{step:4}), and compute the variance of the degenerate off-diagonal term (Step~\ref{step:5}). We then establish a central limit theorem for this degenerate quadratic form by verifying de Jong's conditions (Step~\ref{step:6}). Finally, we identify the centering and variance constants using the technical lemmas gathered in Appendix~\ref{app:tech.lemmas} (Step~\ref{step:7}).

Throughout the proof, write $b = b_n$ for brevity. Also, let $\mathfrak{X},\mathfrak{X}',\mathfrak{X}'',\mathfrak{X}'''$ denote independent random matrices with common density $f$ supported on $\mathcal{S}_{++}^d$.

\step{a $V$-statistic representation via the overlap kernel $\kappa_b$}\label{step:1}
For $X,Y\in \mathcal{S}_{++}^d$, define the overlap kernel
\[
\kappa_b(X,Y) \leqdef \int_{\mathcal{S}_{++}^d} K_{\nu(b,d), \, bS}(X) \, K_{\nu(b,d), \, bS}(Y) \, \rd S.
\]
Substituting \eqref{eq:Donald.1} into \eqref{eq:Donald.2} with $b_1 = b_2 = b$, expanding the square, and interchanging sums and integrals yields
\begin{equation}\label{eq:T.kappa.representation}
T(n_1,n_2) = \frac{n_1n_2}{n}\Bigg[\frac{1}{n_1^2} \sum_{i=1}^{n_1} \sum_{j=1}^{n_1}\kappa_b\big(\mathfrak{X}^{(1)}_i,\mathfrak{X}^{(1)}_j\big)
+ \frac{1}{n_2^2} \sum_{i=1}^{n_2} \sum_{j=1}^{n_2}\kappa_b\big(\mathfrak{X}^{(2)}_i,\mathfrak{X}^{(2)}_j\big) - \frac{2}{n_1n_2} \sum_{i=1}^{n_1} \sum_{j=1}^{n_2}\kappa_b\big(\mathfrak{X}^{(1)}_i,\mathfrak{X}^{(2)}_j\big)\Bigg].
\end{equation}

Under $H_0$, the pooled sample is iid. Indeed, let
\[
\mathfrak{Z}_i \leqdef
\begin{cases}
\mathfrak{X}^{(1)}_i, & i\in[n_1], \\
\mathfrak{X}^{(2)}_{i - n_1}, & i\in \{n_1 + 1,\ldots,n\},
\end{cases}
\qquad \text{so that } \mathfrak{Z}_1,\ldots,\mathfrak{Z}_n\stackrel{\mathrm{iid}}{\sim} f ~\text{under $H_0$}.
\]
Define the weights
\begin{equation}\label{eq:alpha.weights}
\alpha_i \leqdef
\begin{cases}
\displaystyle \frac{\sqrt{n_1n_2/n}}{n_1} = \sqrt{\frac{n_2}{nn_1}}, & i\in[n_1], \\[3mm]
\displaystyle -\frac{\sqrt{n_1n_2/n}}{n_2} = -\sqrt{\frac{n_1}{nn_2}}, & i\in \{n_1 + 1,\ldots,n\}.
\end{cases}
\end{equation}
A direct calculation gives
\begin{equation}\label{eq:alpha.sums}
\sum_{i=1}^n \alpha_i = 0, \qquad
\sum_{i=1}^n \alpha_i^2 = \frac{n_2}{n} + \frac{n_1}{n}=1, \qquad
\max_{1\leq i\leq n}\alpha_i^2 \ll \frac{1}{n}.
\end{equation}
Expanding the quadratic form $\sum_{i,j=1}^n \alpha_i\alpha_j\kappa_b(\mathfrak{Z}_i,\mathfrak{Z}_j)$ using \eqref{eq:alpha.weights} yields exactly \eqref{eq:T.kappa.representation}, hence
\begin{equation}\label{eq:T.quadratic.form}
T(n_1,n_2) = \sum_{i=1}^n \sum_{j=1}^n \alpha_i\alpha_j \, \kappa_b(\mathfrak{Z}_i,\mathfrak{Z}_j).
\end{equation}

\step{centering the kernel (degeneracy) and separating diagonal/off-diagonal parts}\label{step:2}
Let
\[
m_b(X) \leqdef \EE\{\kappa_b(X,\mathfrak{X})\}, \qquad
m_{b,0} \leqdef \EE\{\kappa_b(\mathfrak{X},\mathfrak{X}')\}.
\]
Define the Hoeffding-centered kernel
\begin{equation}\label{eq:h.b.def}
h_b(X,Y) \leqdef \kappa_b(X,Y) - m_b(X) - m_b(Y) + m_{b,0}.
\end{equation}
Then, for every fixed $X\in \mathcal{S}_{++}^d$,
\begin{equation}\label{eq:degeneracy}
\EE\{h_b(X,\mathfrak{X})\}
= \EE\{\kappa_b(X,\mathfrak{X})\} - m_b(X) - \EE\{m_b(\mathfrak{X})\} + m_{b,0}
= m_b(X) - m_b(X) - m_{b,0} + m_{b,0}
= 0,
\end{equation}
so the kernel is \emph{degenerate} of order $1$.

Using \eqref{eq:h.b.def} in \eqref{eq:T.quadratic.form} gives
\[
T(n_1,n_2)
= \sum_{i,j=1}^n \alpha_i\alpha_j \, h_b(\mathfrak{Z}_i,\mathfrak{Z}_j) + \sum_{i,j=1}^n \alpha_i\alpha_j\{m_b(\mathfrak{Z}_i) + m_b(\mathfrak{Z}_j) - m_{b,0}\}.
\]
The second term vanishes \emph{exactly} because $\sum_{i=1}^n \alpha_i = 0$ in \eqref{eq:alpha.sums}:
\[
\sum_{i,j=1}^n \alpha_i\alpha_j \, m_b(\mathfrak{Z}_i) = \left(\sum_{j=1}^n \alpha_j\right)\left(\sum_{i=1}^n \alpha_i m_b(\mathfrak{Z}_i)\right) = 0, \qquad
\sum_{i,j=1}^n \alpha_i\alpha_j \, m_b(\mathfrak{Z}_j) = 0, \qquad
\sum_{i,j=1}^n \alpha_i\alpha_j \, m_{b,0} = m_{b,0} \left(\sum_{i=1}^n \alpha_i\right)^2 = 0.
\]
Hence, $T(n_1,n_2) = \sum_{i=1}^n \sum_{j=1}^n \alpha_i\alpha_j \, h_b(\mathfrak{Z}_i,\mathfrak{Z}_j)$. Now split into diagonal and off-diagonal parts:
\begin{equation}\label{eq:T.D.U}
T(n_1,n_2) = D_n(b) + U_n(b),
\end{equation}
where
\[
D_n(b) \leqdef \sum_{i=1}^n \alpha_i^2 \, h_b(\mathfrak{Z}_i,\mathfrak{Z}_i), \qquad
U_n(b) \leqdef 2\sum_{1\leq i < j\leq n}\alpha_i\alpha_j \, h_b(\mathfrak{Z}_i,\mathfrak{Z}_j).
\]

\step{expectation of $T(n_1,n_2)$ under $H_0$}\label{step:3}
By \eqref{eq:degeneracy}, note that $\EE\{h_b(\mathfrak{Z}_i,\mathfrak{Z}_j)\} = 0$ for $i \neq j$, hence $\EE\{U_n(b)\} = 0$. Therefore,
\[
\EE\{T(n_1,n_2)\} = \EE\{D_n(b)\} = \left(\sum_{i=1}^n \alpha_i^2\right)\EE\{h_b(\mathfrak{X},\mathfrak{X})\} = \EE\{h_b(\mathfrak{X},\mathfrak{X})\},
\]
where we used $\sum_{i=1}^n \alpha_i^2 = 1$ from \eqref{eq:alpha.sums}. Since $h_b(X,X) = \kappa_b(X,X) - 2 m_b(X) + m_{b,0}$, taking expectations and using $\EE\{m_b(\mathfrak{X})\} = m_{b,0}$ gives
\begin{equation}\label{eq:mean.tau.b}
\tau_b \leqdef \EE\{T(n_1,n_2)\} = \EE\{h_b(\mathfrak{X},\mathfrak{X})\} = \EE\{\kappa_b(\mathfrak{X},\mathfrak{X})\} - m_{b,0}.
\end{equation}
In particular,
\[
m_{b,0} = \int_{\mathcal{S}_{++}^d} \big(\EE\{K_{\nu(b,d),bS}(\mathfrak{X})\}\big)^2 \rd S,
\]
where the expectation inside the square is the smoothed target density. Under the stated boundedness and continuity assumptions on $f$, $m_{b,0} \ll 1$ as $b \downarrow 0$; indeed $m_{b,0} = \EE\{m_b(\mathfrak{X})\} \leq \sup_{X\in \mathcal{S}_{++}^d} m_b(X) \ll 1$ by the argument leading to \eqref{eq:int.kappa.dy} in the proof of Lemma~\ref{lem:h.cycle.fourth.moment}. Its exact asymptotic behavior is irrelevant for \eqref{eq:clt.null.shrinking}; all that matters here is the bound $m_{b,0} \ll 1$, which implies $b^{r/4}m_{b,0} \to 0$.

\step{the diagonal fluctuation $D_n(b) - \tau_b$ is negligible after scaling}\label{step:4}
Write
\[
W_i \leqdef h_b(\mathfrak{Z}_i,\mathfrak{Z}_i) - \EE\{h_b(\mathfrak{X},\mathfrak{X})\}, \qquad i \in [n].
\]
Then $W_1,\ldots,W_n$ are iid with mean $0$, and $D_n(b) - \tau_b = \sum_{i=1}^n \alpha_i^2 \, W_i$. Hence, by independence and \eqref{eq:alpha.sums},
\begin{equation}\label{eq:Var.Dn}
\Var\{D_n(b) - \tau_b\}
= \sum_{i=1}^n \alpha_i^4 \, \Var(W_1)
\leq \left(\max_{1\leq i\leq n}\alpha_i^2\right)\left(\sum_{i=1}^n \alpha_i^2\right)\EE\{W_1^2\}
\ll \frac{1}{n} \, \EE\{h_b(\mathfrak{X},\mathfrak{X})^2\}.
\end{equation}
By Lemma~\ref{lem:h.centered.p.moment.control} with $p = 2$, whose finite-integrability hypotheses here are precisely $V_{2,2} < \infty$ from \eqref{eq:moment.conditions} together with $\int_{\mathcal{S}_{++}^d} f(S) \, |S|^{-(d + 1)} \, \rd S < \infty$, which we assumed in \eqref{eq:moment.conditions.2}, we have $\EE\{h_b(\mathfrak{X},\mathfrak{X})^2\} \ll b^{-r}$ as $b \downarrow 0$. Inserting this into \eqref{eq:Var.Dn} gives
\[
\Var\{D_n(b) - \tau_b\} \ll \frac{b^{-r}}{n}.
\]
Therefore, by Chebyshev's inequality and \eqref{eq:nb.condition}, $b^{r/4} \, \{D_n(b) - \tau_b\} \stackrel{\PP}{\longrightarrow} 0$. As a consequence of the decomposition in \eqref{eq:T.D.U}, the asymptotic distribution of $b^{r/4}\{T(n_1,n_2) - \tau_b\}$ is the same as that of $b^{r/4}U_n(b)$.

\step{variance of the degenerate quadratic form $U_n(b)$}\label{step:5}
We claim that
\begin{equation}\label{eq:Var.Un}
\Var\{U_n(b)\} = 4 \sum_{1\leq i < j\leq n}\alpha_i^2\alpha_j^2 \, \EE\{h_b(\mathfrak{X},\mathfrak{X}')^2\}.
\end{equation}
Indeed, expand
\[
\Var\{U_n(b)\} = 4 \sum_{1 \leq i < j \leq n} \sum_{1 \leq k < \ell \leq n} \alpha_i\alpha_j\alpha_k\alpha_{\ell} \, \Cov(h_b(\mathfrak{Z}_i,\mathfrak{Z}_j),h_b(\mathfrak{Z}_k,\mathfrak{Z}_{\ell})).
\]
If $\{i,j\}\cap\{k,\ell\} = \varnothing$, the two terms are independent, hence the covariance is $0$. If exactly one index is shared, say $k = i$ and $\ell\notin\{i,j\}$, then using degeneracy \eqref{eq:degeneracy},
\[
\EE[h_b(\mathfrak{Z}_i,\mathfrak{Z}_j) \, h_b(\mathfrak{Z}_i,\mathfrak{Z}_{\ell})]
= \EE\left[ \, \EE\{h_b(\mathfrak{Z}_i,\mathfrak{Z}_j) \, h_b(\mathfrak{Z}_i,\mathfrak{Z}_{\ell}) \mid \mathfrak{Z}_i\}\right]
= \EE\left[ \, \EE\{h_b(\mathfrak{Z}_i,\mathfrak{Z}_j) \mid \mathfrak{Z}_i\} \, \EE\{h_b(\mathfrak{Z}_i,\mathfrak{Z}_{\ell}) \mid \mathfrak{Z}_i\}\right]
= 0,
\]
and since each term has mean $0$, the covariance is $0$. Therefore the only nonzero terms are those with $(k,\ell) = (i,j)$, which yields \eqref{eq:Var.Un}.

Now compute the weight sum:
\begin{equation}\label{eq:alpha.weight.sum}
\sum_{1\leq i < j\leq n}\alpha_i^2\alpha_j^2
= \frac{1}{2}\left\{\left(\sum_{i=1}^n \alpha_i^2\right)^2 - \sum_{i=1}^n \alpha_i^4\right\}
= \frac{1}{2}\left\{1 - \sum_{i=1}^n \alpha_i^4\right\}
= \frac{1}{2} + \OO\left(\frac{1}{n}\right),
\end{equation}
because $\sum_{i=1}^n \alpha_i^2 = 1$ and $\sum_{i=1}^n \alpha_i^4 \leq (\max_{1\leq i \leq n} \alpha_i^2)\sum_{i=1}^n \alpha_i^2 \ll 1/n$ by \eqref{eq:alpha.sums}. Hence, combining \eqref{eq:Var.Un} and \eqref{eq:alpha.weight.sum},
\begin{equation}\label{eq:Var.Un.asymp}
\Var\{U_n(b)\} = \left(2 + \oo(1)\right) \EE\{h_b(\mathfrak{X},\mathfrak{X}')^2\}.
\end{equation}

\step{a central limit theorem for $U_n(b)$ under \eqref{eq:nb.condition}}\label{step:6}
We verify de Jong's row-negligibility and normalized fourth-moment conditions. For $1 \leq i < j \leq n$, let
\[
\xi_{ij} \leqdef 2\alpha_i\alpha_j \, h_b(\mathfrak{Z}_i,\mathfrak{Z}_j),
\]
and extend this to a symmetric array by setting $\xi_{ji} \leqdef \xi_{ij}$ for $j < i$ and $\xi_{ii} \leqdef 0$. Then $U_n(b) = \sum_{1 \leq i < j \leq n} \xi_{ij}$. By the symmetry of $h_b$ and the degeneracy in \eqref{eq:degeneracy},
\[
\EE(\xi_{ij}\mid \mathfrak{Z}_i) = 0, \qquad
\EE(\xi_{ij}\mid \mathfrak{Z}_j) = 0, \qquad 1 \leq i \neq j \leq n,
\]
so the array is clean. De Jong's central limit theorem for clean quadratic forms \citep[see][Theorem~2.1]{deJong1987} implies
\begin{equation}\label{eq:CLT.U}
\frac{U_n(b)}{\sqrt{\Var\{U_n(b)\}}} \rightsquigarrow \mathcal{N}(0,1),
\end{equation}
provided that, as $n \to \infty$,
\begin{equation}
(a) ~~ \frac{\max_{1 \leq i \leq n} \sum_{j=1}^n \EE(\xi_{ij}^2)}{\Var\{U_n(b)\}} \longrightarrow 0, \qquad \quad
(b) ~~ \frac{\EE\{U_n(b)^4\}}{\Var\{U_n(b)\}^2} \longrightarrow 3.
\end{equation}

\medskip
\noindent
\emph{Verification of de Jong's condition $(a)$.} For each $i \in \{1,\ldots,n\}$,
\[
\sum_{j=1}^n \EE(\xi_{ij}^2)
= 4 \alpha_i^2 \sum_{\substack{1 \leq j \leq n \\ j \neq i}} \alpha_j^2 \, \EE\{h_b(\mathfrak{X},\mathfrak{X}')^2\}
\leq 4 \alpha_i^2 \, \EE\{h_b(\mathfrak{X},\mathfrak{X}')^2\},
\]
hence, by \eqref{eq:alpha.sums},
\[
\max_{1\leq i\leq n} \sum_{j=1}^n \EE(\xi_{ij}^2)
\leq 4 \, \EE\{h_b(\mathfrak{X},\mathfrak{X}')^2\} \max_{1\leq i\leq n} \alpha_i^2
\ll \frac{1}{n} \EE\{h_b(\mathfrak{X},\mathfrak{X}')^2\}.
\]
Since $\Var\{U_n(b)\} = (2 + \oo(1)) \, \EE\{h_b(\mathfrak{X},\mathfrak{X}')^2\}$ by \eqref{eq:Var.Un.asymp}, we obtain
\[
\frac{\max_{1 \leq i \leq n} \sum_{j=1}^n \EE(\xi_{ij}^2)}{\Var\{U_n(b)\}} \ll \frac{1}{n} \to 0,
\]
completing the verification of de Jong's condition $(a)$.

\medskip
\noindent
\emph{Verification of de Jong's condition $(b)$.}
First record the moment orders needed below. By Lemma~\ref{lem:kappa.offdiag.p.moment} with $p = 2$, which uses the finiteness of $V_{2,2}$ in \eqref{eq:moment.conditions},
\[
\EE\{\kappa_b(\mathfrak{X},\mathfrak{X}')^2\} \asymp b^{-r/2}.
\]
Moreover, boundedness of $f$ implies $\smash{\sup_{X\in \mathcal{S}_{++}^d}m_b(X) + |m_{b,0}| \ll 1}$ uniformly for small $b$ (see the argument leading to \eqref{eq:int.kappa.dy} in the proof of Lemma~\ref{lem:h.cycle.fourth.moment}), hence
\[
\EE\{h_b(\mathfrak{X},\mathfrak{X}')^2\} = \EE\{\kappa_b(\mathfrak{X},\mathfrak{X}')^2\} + \OO(1) \asymp b^{-r/2},
\]
and in particular
\begin{equation}\label{eq:var.est}
\Var\{U_n(b)\}^2 \asymp b^{-r},
\end{equation}
by \eqref{eq:Var.Un.asymp}. Under the additional integrability assumption $\smash{\int_{\mathcal{S}_{++}^d} f(S)^2 \, |S|^{-3(d + 1)/2} \, \rd S < \infty}$, which we assumed in \eqref{eq:moment.conditions.2}, Lemma~\ref{lem:kappa.offdiag.p.moment} with $p = 4$ yields $\EE\{\kappa_b(\mathfrak{X},\mathfrak{X}')^4\} \ll b^{-3r/2}$. Since $\smash{\sup_{X\in \mathcal{S}_{++}^d}m_b(X) \ll 1}$ and $|m_{b,0}| \ll 1$ uniformly for small $b$, we have $|h_b(X,Y)|^4 \ll \kappa_b(X,Y)^4 + 1$ uniformly in $(X,Y)$, and therefore
\begin{equation}\label{eq:h4.order}
\EE\{h_b(\mathfrak{X},\mathfrak{X}')^4\} \ll b^{-3r/2}.
\end{equation}

Because the array $(\xi_{ij})$ is clean, the fourth-moment decomposition for clean quadratic forms (see \citet[Table~1]{deJong1987}) yields
\begin{equation}\label{eq:U4.deJong.decomp}
\begin{aligned}
\EE\{U_n(b)^4\}
&= G_{I,n} + 6 \, G_{II,n} + 12 \, G_{III,n} + 24 \, G_{IV,n} + 6 \, G_{V,n} \\
&= G_{I,n} + 6 \, G_{II,n} + 12 \, G_{III,n} + 24 \, G_{IV,n} + 3 \, \Var\{U_n(b)\}^2 + \oo\left(\Var\{U_n(b)\}^2\right),
\end{aligned}
\end{equation}
where, using the symmetry $\xi_{ji} = \xi_{ij}$, the graph-sums in Table~1 of \citet{deJong1987} are
\[
\begin{aligned}
G_{I,n} &\leqdef \sum_{1 \leq i < j \leq n} \EE(\xi_{ij}^4), \qquad
G_{II,n} \leqdef \sum_{1 \leq i < j < k \leq n}\big\{\EE(\xi_{ij}^2\xi_{ik}^2) + \EE(\xi_{ij}^2\xi_{jk}^2) + \EE(\xi_{ik}^2\xi_{jk}^2)\big\}, \\
G_{III,n} &\leqdef \sum_{1 \leq i < j < k \leq n}\big\{\EE(\xi_{ij}^2\xi_{ik}\xi_{jk}) + \EE(\xi_{ik}^2\xi_{ij}\xi_{jk}) + \EE(\xi_{jk}^2\xi_{ij}\xi_{ik})\big\}, \\
G_{IV,n} &\leqdef \sum_{1 \leq i < j < k < \ell \leq n}\big\{\EE(\xi_{ij}\xi_{ik}\xi_{j\ell}\xi_{k\ell}) + \EE(\xi_{ij}\xi_{i\ell}\xi_{jk}\xi_{k\ell}) + \EE(\xi_{ik}\xi_{i\ell}\xi_{jk}\xi_{j\ell})\big\}, \\
G_{V,n} &\leqdef \sum_{1 \leq i < j < k < \ell \leq n}\big\{\EE(\xi_{ij}^2\xi_{k\ell}^2) + \EE(\xi_{ik}^2\xi_{j\ell}^2) + \EE(\xi_{i\ell}^2\xi_{jk}^2)\big\}.
\end{aligned}
\]
The second equality in \eqref{eq:U4.deJong.decomp} is due to
\[
G_{V,n} = \frac{1}{2} \Var\{U_n(b)\}^2 + \oo\left(\Var\{U_n(b)\}^2\right),
\]
which is a consequence of condition $(a)$; see \citet[Eq.~(6)]{deJong1987}. We show that each term $G_{I,n},G_{II,n},G_{III,n},G_{IV,n}$ is $\oo(\Var\{U_n(b)\}^2)$.

\smallskip
\noindent
\emph{(i) The $G_{I,n}$ term.}
Since $\EE(\xi_{ij}^4) = 16\alpha_i^4\alpha_j^4 \, \EE\{h_b(\mathfrak{X},\mathfrak{X}')^4\}$,
\[
G_{I,n}
\ll \EE\{h_b(\mathfrak{X},\mathfrak{X}')^4\} \left(\sum_{i=1}^n \alpha_i^4\right)^2
\ll \frac{1}{n^2 b^{3r/2}},
\]
by \eqref{eq:h4.order} and \eqref{eq:alpha.sums}. Since $\Var\{U_n(b)\}^2 \asymp b^{-r}$ by \eqref{eq:var.est}, it follows from \eqref{eq:nb.condition} that
\[
\frac{G_{I,n}}{\Var\{U_n(b)\}^2} \ll \frac{1}{n^2 b^{r/2}} = \oo(1).
\]

\smallskip
\noindent
\emph{(ii) The $G_{II,n}$ term (star graphs).}
A typical summand in $G_{II,n}$ has the form $\EE(\xi_{ij}^2\xi_{ik}^2)$ with $i,j,k$ distinct:
\[
\EE(\xi_{ij}^2\xi_{ik}^2)
= 16 \alpha_i^4 \alpha_j^2 \alpha_k^2 \, \EE\left[h_b(\mathfrak{Z}_i,\mathfrak{Z}_j)^2 \, h_b(\mathfrak{Z}_i,\mathfrak{Z}_k)^2\right]
\leq 16 \alpha_i^4 \alpha_j^2 \alpha_k^2 \, \EE\{h_b(\mathfrak{X},\mathfrak{X}')^4\},
\]
by Cauchy--Schwarz. Summing over indices yields
\[
G_{II,n}
\ll \EE\{h_b(\mathfrak{X},\mathfrak{X}')^4\} \sum_{i=1}^n \alpha_i^4\sum_{1\leq j < k\leq n}\alpha_j^2\alpha_k^2
\ll \frac{1}{n b^{3r/2}},
\]
using \eqref{eq:h4.order}, and using \eqref{eq:alpha.sums} to deduce that $\sum_{i=1}^n \alpha_i^4 \ll 1/n$ and $\sum_{1\leq j < k\leq n}\alpha_j^2\alpha_k^2 \ll 1$. Since $\Var\{U_n(b)\}^2 \asymp b^{-r}$ by \eqref{eq:var.est}, it follows from \eqref{eq:nb.condition} that
\[
\frac{G_{II,n}}{\Var\{U_n(b)\}^2} \ll \frac{1}{n b^{r/2}} = \oo(1).
\]

\smallskip
\noindent
\emph{(iii) The $G_{III,n}$ term.}
Since $|G_{III,n}| \leq G_{II,n}$ by \citet[Eq.~(3)]{deJong1987}, we get
\[
\frac{G_{III,n}}{\Var\{U_n(b)\}^2} = \oo(1),
\]
from $(ii)$.

\smallskip
\noindent
\emph{(iv) The $G_{IV,n}$ term (cycle graphs).}
A representative cycle summand has the form $\EE(\xi_{12}\xi_{13}\xi_{42}\xi_{43})$, which equals
\[
16 \, \alpha_1^2\alpha_4^2\alpha_2^2\alpha_3^2 \, \EE\left[h_b(\mathfrak{X},\mathfrak{X}') \, h_b(\mathfrak{X},\mathfrak{X}'') \, h_b(\mathfrak{X}''',\mathfrak{X}') \, h_b(\mathfrak{X}''',\mathfrak{X}'')\right],
\]
with $\mathfrak{X},\mathfrak{X}',\mathfrak{X}'',\mathfrak{X}''' \stackrel{\mathrm{iid}}{\sim} f$. By Lemma~\ref{lem:h.cycle.fourth.moment}, which again uses our assumption $V_{2,2} < \infty$ from \eqref{eq:moment.conditions}, we have
\[
\EE\left[h_b(\mathfrak{X},\mathfrak{X}') \, h_b(\mathfrak{X},\mathfrak{X}'') \, h_b(\mathfrak{X}''',\mathfrak{X}') \, h_b(\mathfrak{X}''',\mathfrak{X}'')\right]
\ll b^{-r/2}, \qquad b \downarrow 0.
\]
Since $\sum_{1\leq i < j < k < \ell\leq n}\alpha_i^2\alpha_j^2\alpha_k^2\alpha_\ell^2 \ll 1$, it follows that $|G_{IV,n}| \ll b^{-r/2}$, and therefore, using \eqref{eq:var.est},
\[
\left|\frac{G_{IV,n}}{\Var\{U_n(b)\}^2}\right| \ll b^{r/2} = \oo(1).
\]

Combining $(i)$--$(iv)$ in \eqref{eq:U4.deJong.decomp} yields
\[
\frac{\EE[U_n(b)^4]}{\Var\{U_n(b)\}^2} \to 3,
\]
completing the verification of de Jong's condition $(b)$.

By Step~\ref{step:4}, $b^{r/4}\{T(n_1,n_2) - \tau_b\}$ has the same limit as $b^{r/4}U_n(b)$. Hence, using \eqref{eq:CLT.U},
\begin{equation}\label{eq:CLT.centered.by.tau}
\frac{b^{r/4}\{T(n_1,n_2) - \tau_b\}}{\sqrt{b^{r/2} \, \Var\{U_n(b)\}}} \rightsquigarrow \mathcal{N}(0,1).
\end{equation}

\step{identify the centering and the limit of $b^{r/2} \, \Var\{U_n(b)\}$}\label{step:7}
We first treat the centering. By \eqref{eq:mean.tau.b} and Lemma~\ref{lem:kappa.diagonal.p.moment} with $p = 1$, which uses $V_{1,1} = \int_{\mathcal{S}_{++}^d} f(S) \, |S|^{-(d + 1)/2} \, \rd S < \infty$ from \eqref{eq:moment.conditions}, we have $\EE\{\kappa_b(\mathfrak{X},\mathfrak{X})\} = A_d(b) \, V_{1,1}$. Hence
\[
\tau_b
= \EE\{\kappa_b(\mathfrak{X},\mathfrak{X})\} - m_{b,0}
= A_d(b) \, V_{1,1} - m_{b,0}.
\]
Since $m_{b,0} \ll 1$ as explained in Step~\ref{step:3}, we obtain
\begin{equation}\label{eq:centering.shift}
b^{r/4}\{\tau_b - A_d(b) \, V_{1,1}\} = - \, b^{r/4}m_{b,0} \longrightarrow 0.
\end{equation}

We now treat the variance. By \eqref{eq:Var.Un.asymp}, it is enough to identify the limit of $b^{r/2} \, \EE\{h_b(\mathfrak{X},\mathfrak{X}')^2\}$. Let
\[
a_b(X,Y) \leqdef m_b(X) + m_b(Y) - m_{b,0},
\]
so that $h_b(X,Y) = \kappa_b(X,Y) - a_b(X,Y)$. By boundedness of $f$, the argument leading to \eqref{eq:int.kappa.dy} in the proof of Lemma~\ref{lem:h.cycle.fourth.moment} implies that $\sup_{X,Y\in \mathcal{S}_{++}^d} |a_b(X,Y)| \ll 1$ for all sufficiently small $b$, hence $\EE\{a_b(\mathfrak{X},\mathfrak{X}')^2\} \ll 1$ and
\[
\left|\EE\{\kappa_b(\mathfrak{X},\mathfrak{X}') \, a_b(\mathfrak{X},\mathfrak{X}')\}\right| \leq \left(\sup_{X,Y\in \mathcal{S}_{++}^d} |a_b(X,Y)|\right) \EE\{\kappa_b(\mathfrak{X},\mathfrak{X}')\} \ll 1.
\]
Expanding the square gives
\[
\EE\{h_b(\mathfrak{X},\mathfrak{X}')^2\}
= \EE\{\kappa_b(\mathfrak{X},\mathfrak{X}')^2\} - 2 \, \EE\{\kappa_b(\mathfrak{X},\mathfrak{X}') \, a_b(\mathfrak{X},\mathfrak{X}')\} + \EE\{a_b(\mathfrak{X},\mathfrak{X}')^2\}
= \EE\{\kappa_b(\mathfrak{X},\mathfrak{X}')^2\} + \OO(1).
\]
Therefore, since $b^{r/2} \downarrow 0$, $b^{r/2} \, \EE\{h_b(\mathfrak{X},\mathfrak{X}')^2\} = b^{r/2} \, \EE\{\kappa_b(\mathfrak{X},\mathfrak{X}')^2\} + \oo(1)$. Applying Lemma~\ref{lem:kappa.offdiag.p.moment} with $p = 2$, again using $V_{2,2} < \infty$ from \eqref{eq:moment.conditions}, yields
\[
b^{r/2} \, \EE\{\kappa_b(\mathfrak{X},\mathfrak{X}')^2\} \longrightarrow c_d^2 \, J_{d,2} \, V_{2,2}, \qquad b \downarrow 0,
\]
where $V_{2,2} = \int_{\mathcal{S}_{++}^d} f(S)^2 \, |S|^{-(d + 1)/2} \, \rd S$ and
\[
J_{d,2}
= \left(\sqrt{\frac{8\pi}{2}}\right)^{d}\left(\sqrt{\frac{4\pi}{2}}\right)^{d(d - 1)/2}
= (2\sqrt{\pi})^d(\sqrt{2\pi})^{d(d - 1)/2}
\equiv I_d.
\]
Combining with \eqref{eq:Var.Un.asymp} gives
\[
b^{r/2} \, \Var\{U_n(b)\}
= b^{r/2}\left(2 + \oo(1)\right) \EE\{h_b(\mathfrak{X},\mathfrak{X}')^2\}
\longrightarrow 2 c_d^2 I_d V_{2,2}
\equiv v_d V_{2,2}, \qquad b \downarrow 0.
\]
Therefore, \eqref{eq:CLT.centered.by.tau} implies $b^{r/4}\{T(n_1,n_2) - \tau_b\} \rightsquigarrow \mathcal{N}(0, v_d V_{2,2})$. Finally, by \eqref{eq:centering.shift},
\[
b^{r/4}\{T(n_1,n_2) - A_d(b)V_{1,1}\} = b^{r/4}\{T(n_1,n_2) - \tau_b\} + b^{r/4}\{\tau_b - A_d(b)V_{1,1}\},
\]
and the second term on the right-hand side converges to $0$ by \eqref{eq:centering.shift}. Applying Slutsky's theorem then yields
\[
b^{r/4}\{T(n_1,n_2) - A_d(b)V_{1,1}\} \rightsquigarrow \mathcal{N}(0, v_d V_{2,2}),
\]
which is \eqref{eq:clt.null.shrinking}. This completes the proof of Theorem~\ref{thm:null.shrinking.bandwidth}.

\subsection{Proof of Theorem~\ref{thm:null.fixed.bandwidth}}

For every $X \in \mathcal{S}_{++}^d$, Lemma~\ref{lem:kappa.diagonal.p.moment} gives
\[
\|\phi_b(X)\|_{\mathbb{H}}^2
= \int_{\mathcal{S}_{++}^d} K_{\nu(b,d), \, bS}(X)^2 \, \rd S
= \kappa_b(X,X)
= A_d(b) \, |X|^{-(d + 1)/2}
< \infty,
\]
so $\phi_b(X) \in \mathbb{H}$. We begin by rewriting $T(n_1,n_2)$ as a squared norm in $\mathbb{H}$. Write $n_{\mathrm{eff}} \leqdef n_1n_2/n$. For $\ell \in \{1,2\}$, define the $\mathbb{H}$-valued sample mean $\smash{\bar\phi_{\ell} \leqdef \frac{1}{n_\ell} \sum_{i=1}^{n_\ell}\phi_b(\mathfrak{X}_i^{(\ell)})}$. By the definition \eqref{eq:Donald.1} (with $b_\ell = b$), we have $\smash{\hat{f}_{n_\ell,b}^{(\ell)}(S)} = \bar\phi_{\ell}(S)$ for every $S \in \mathcal{S}_{++}^d$. Hence,
\begin{equation}\label{eq:T.as.norm}
T(n_1,n_2)
= \frac{n_1n_2}{n} \int_{\mathcal{S}_{++}^d}\big(\bar\phi_1(S) - \bar\phi_2(S)\big)^2 \, \rd S
= n_{\mathrm{eff}} \|\bar\phi_1 - \bar\phi_2\|_{\mathbb{H}}^2
= \|\Delta_n\|_{\mathbb{H}}^2, \qquad \Delta_n \leqdef \sqrt{n_{\mathrm{eff}}} \, (\bar\phi_1 - \bar\phi_2).
\end{equation}

Next, we justify that $\phi_b(\mathfrak{X})$ is an $\mathbb{H}$-valued random element with finite second moment. First, by the explicit closed form \eqref{eq:kappa.closed.form.appendix}, the overlap kernel $(X,Y) \mapsto \kappa_b(X,Y)$ is continuous on $\mathcal{S}_{++}^d \times \mathcal{S}_{++}^d$. Therefore,
\[
\|\phi_b(X) - \phi_b(Y)\|_{\mathbb{H}}^2 = \kappa_b(X,X) + \kappa_b(Y,Y) - 2\kappa_b(X,Y) \longrightarrow 0 \qquad \text{as} ~~ Y \to X,
\]
so $X \mapsto \phi_b(X)$ is continuous from $\mathcal{S}_{++}^d$ into $\mathbb{H}$. In particular, $\phi_b(\mathfrak{X})$ is Borel measurable. Also, by Lemma~\ref{lem:kappa.diagonal.p.moment} and the finiteness of $V_{1,1}$ assumed in \eqref{eq:fixed.b.moment},
\[
\EE[\|\phi_b(\mathfrak{X})\|_{\mathbb{H}}^2]
= \EE[\kappa_b(\mathfrak{X},\mathfrak{X})]
= A_d(b) \, V_{1,1}
< \infty.
\]
Hence, by Cauchy--Schwarz, $\EE[\|\phi_b(\mathfrak{X})\|_{\mathbb{H}}] < \infty$, so $\mu_b = \EE[\phi_b(\mathfrak{X})]$ exists as a Bochner integral in $\mathbb{H}$ (meaning $\|\mu_b\|_{\mathbb{H}} < \infty$). Moreover,
\begin{equation}\label{eq:element.H}
\EE[\|\phi_b(\mathfrak{X}) - \mu_b\|_{\mathbb{H}}^2]
\leq 2 \, \EE[\|\phi_b(\mathfrak{X})\|_{\mathbb{H}}^2] + 2 \, \|\mu_b\|_{\mathbb{H}}^2
< \infty.
\end{equation}
Set $Z \leqdef \phi_b(\mathfrak{X}) - \mu_b$. For every $g \in \mathbb{H}$, Cauchy--Schwarz and \eqref{eq:element.H} give $\EE[\|\langle Z,g\rangle_{\mathbb{H}} Z\|_{\mathbb{H}}] \leq \|g\|_{\mathbb{H}} \, \EE[\|Z\|_{\mathbb{H}}^2] < \infty$, so the Bochner expectation defining $\Sigma_b \, g$ in \eqref{eq:Sigma.b} exists. Moreover, $\|\Sigma_b \, g\|_{\mathbb{H}} \leq \EE[|\langle Z,g\rangle_{\mathbb{H}}| \, \|Z\|_{\mathbb{H}}] \leq \|g\|_{\mathbb{H}} \, \EE[\|Z\|_{\mathbb{H}}^2]$. Thus, $\Sigma_b$ is a well-defined bounded linear operator on $\mathbb{H}$.

We now establish a Hilbert-space central limit theorem for each sample mean. We first note that $\mathbb{H} = L^2(\mathcal{S}_{++}^d,\rd S)$ is a separable Hilbert space. By the central limit theorem in separable Hilbert spaces \citep[see, e.g.,][Theorem~10.5]{MR2814399}, it follows that, as $n_\ell \to \infty$,
\begin{equation}\label{eq:HCLT.each.sample}
\sqrt{n_\ell} \, (\bar\phi_{\ell} - \mu_b) \rightsquigarrow \mathcal{G}_{\ell}, \qquad \ell \in \{1,2\},
\end{equation}
where $\mathcal{G}_\ell$ is a centered Gaussian element in $\mathbb{H}$ with covariance operator $\Sigma_b:\mathbb{H} \to \mathbb{H}$. For completeness, for any $g,h \in \mathbb{H}$,
\[
\Cov\big(\langle \mathcal{G}_\ell,g\rangle_{\mathbb{H}}, \langle \mathcal{G}_\ell,h\rangle_{\mathbb{H}}\big)
= \langle \Sigma_b \, g, h\rangle_{\mathbb{H}}
= \EE\big[\langle \phi_b(\mathfrak{X}) - \mu_b,g\rangle_{\mathbb{H}} \langle \phi_b(\mathfrak{X}) - \mu_b,h\rangle_{\mathbb{H}}\big].
\]
Since the two samples are independent for every $n$, the pair $(\sqrt{n_1} \, (\bar\phi_1 - \mu_b), \, \sqrt{n_2} \, (\bar\phi_2 - \mu_b))$ has independent coordinates for every $n$. Therefore, together with \eqref{eq:HCLT.each.sample}, the joint law converges in $\mathbb{H} \times \mathbb{H}$ to the product law:
\begin{equation}\label{eq:joint.convergence}
\big(\sqrt{n_1} \, (\bar\phi_1 - \mu_b), \, \sqrt{n_2} \, (\bar\phi_2 - \mu_b)\big) \rightsquigarrow (\mathcal{G}_1,\mathcal{G}_2),
\end{equation}
where $\mathcal{G}_1$ and $\mathcal{G}_2$ are independent centered Gaussian elements in $\mathbb{H}$ with covariance operator $\Sigma_b$.

We now derive a central limit theorem for $\Delta_n$. By adding/subtracting $\mu_b$ in the definition of $\Delta_n$ in \eqref{eq:T.as.norm},
\[
\Delta_n = \sqrt{n_{\mathrm{eff}}} \, (\bar\phi_1 - \mu_b) - \sqrt{n_{\mathrm{eff}}} \, (\bar\phi_2 - \mu_b).
\]
Rewrite the coefficients as
\[
\sqrt{n_{\mathrm{eff}}} \, (\bar\phi_1 - \mu_b) = \sqrt{\frac{n_2}{n}} \sqrt{n_1} \, (\bar\phi_1 - \mu_b), \qquad
\sqrt{n_{\mathrm{eff}}} \, (\bar\phi_2 - \mu_b) = \sqrt{\frac{n_1}{n}} \sqrt{n_2} \, (\bar\phi_2 - \mu_b).
\]
Since $n_1/n \to \varrho$ and $n_2/n \to 1 - \varrho$ by \eqref{eq:n1.n2.n}, Slutsky's theorem together with the preceding joint convergence in \eqref{eq:joint.convergence} yields
\begin{equation}\label{eq:Delta.limit}
\Delta_n \rightsquigarrow \sqrt{1 - \varrho} \, \mathcal{G}_1 - \sqrt{\varrho} \, \mathcal{G}_2 \stackrel{\mathrm{law}}{=} \mathcal{G},
\end{equation}
where $\mathcal{G}$ is a centered Gaussian element in $\mathbb{H}$ with covariance operator $(1 - \varrho)\Sigma_b + \varrho\Sigma_b = \Sigma_b$.

We then apply the continuous mapping theorem to obtain the limit of $T(n_1,n_2)$. The mapping $g \mapsto \|g\|_{\mathbb{H}}^2$ is continuous on $\mathbb{H}$. Therefore, by \eqref{eq:T.as.norm}, \eqref{eq:Delta.limit} and the continuous mapping theorem,
\begin{equation}\label{eq:T.to.normG}
T(n_1,n_2) = \|\Delta_n\|_{\mathbb{H}}^2 \rightsquigarrow \|\mathcal{G}\|_{\mathbb{H}}^2.
\end{equation}

Finally, we give a spectral representation of $\|\mathcal{G}\|_{\mathbb{H}}^2$ as a weighted $\chi^2$ series. By construction,
\[
\langle \Sigma_b g,h\rangle_{\mathbb{H}}
= \EE[\langle Z,g\rangle_{\mathbb{H}}\langle Z,h\rangle_{\mathbb{H}}]
= \langle g,\Sigma_b h\rangle_{\mathbb{H}}, \qquad g,h \in \mathbb{H},
\]
so $\Sigma_b$ is self-adjoint; also, $\langle \Sigma_b g,g\rangle_{\mathbb{H}} = \EE[\langle Z,g\rangle_{\mathbb{H}}^2] \geq 0$ for $g \in \mathbb{H}$, so $\Sigma_b$ is nonnegative definite. It remains to prove that $\Sigma_b$ is trace-class. Let $\{u_j\}_{j=1}^{\infty}$ be any orthonormal basis of $\mathbb{H}$. Then, by \eqref{eq:element.H} and Tonelli's theorem,
\[
\sum_{j=1}^{\infty} \langle \Sigma_b \, u_j, u_j\rangle_{\mathbb{H}}
= \sum_{j=1}^{\infty} \EE\big[\langle \phi_b(\mathfrak{X}) - \mu_b, u_j\rangle_{\mathbb{H}}^2\big]
= \EE\left[\sum_{j=1}^{\infty} \langle \phi_b(\mathfrak{X}) - \mu_b, u_j\rangle_{\mathbb{H}}^2\right]
= \EE\big[\|\phi_b(\mathfrak{X}) - \mu_b\|_{\mathbb{H}}^2\big] < \infty,
\]
which proves that $\Sigma_b$ is trace-class. In particular, $\Sigma_b$ is compact. By the spectral theorem, its eigenvectors associated with the nonzero eigenvalues, together with an orthonormal basis of $\ker(\Sigma_b)$ whose elements are assigned eigenvalue $0$, form an orthonormal basis $\{e_j\}_{j=1}^{\infty}$ of $\mathbb{H}$. Write $\Sigma_b \, e_j = \lambda_j e_j$, where $\lambda_j \geq 0$ and $\sum_{j=1}^{\infty} \lambda_j = \tr(\Sigma_b) < \infty$.

Now set $Y_j \leqdef \langle \mathcal{G},e_j\rangle_{\mathbb{H}}$. Since $\mathcal{G}$ is a centered Gaussian element in $\mathbb{H}$, the vector $(Y_1,\ldots,Y_m)$ is jointly Gaussian for each $m$. Moreover,
\[
\Cov(Y_j,Y_k) = \EE(Y_jY_k) = \langle \Sigma_b \, e_j, e_k\rangle_{\mathbb{H}}
= \lambda_j \, \langle e_j,e_k\rangle_{\mathbb{H}}
= \lambda_j \, \ind_{\{j = k\}}.
\]
Hence, the variables $Y_1,Y_2,\ldots$ are independent with $Y_j \sim \mathcal{N}(0,\lambda_j)$. Therefore, if $Z_1,Z_2,\ldots \smash{\stackrel{\mathrm{iid}}{\sim}} \mathcal{N}(0,1)$, then we have $(Y_1,Y_2,\ldots) \smash{\stackrel{\mathrm{law}}{=}} (\sqrt{\lambda_1}Z_1,\sqrt{\lambda_2}Z_2,\ldots)$. Using Parseval's identity,
\[
\|\mathcal{G}\|_{\mathbb{H}}^2
= \sum_{j=1}^{\infty} \langle \mathcal{G},e_j\rangle_{\mathbb{H}}^2
= \sum_{j=1}^{\infty} Y_j^2 \stackrel{\mathrm{law}}{=} \sum_{j=1}^{\infty} \lambda_j Z_j^2.
\]
Since $\sum_{j=1}^{\infty} \lambda_j < \infty$ and $\EE(Z_j^2)=1$, the series $\sum_{j=1}^{\infty} \lambda_j Z_j^2$ converges almost surely and in $L^1$ (by monotone convergence for nonnegative partial sums and $\smash{\EE[\sum_{j=1}^{\infty} \lambda_j Z_j^2] = \sum_{j=1}^{\infty} \lambda_j < \infty}$). Combining this with \eqref{eq:T.to.normG} yields \eqref{eq:fixed.bandwidth.limit}. This completes the proof of Theorem~\ref{thm:null.fixed.bandwidth}.

\section{Technical lemmas}\label{app:tech.lemmas}

The first lemma provides an explicit diagonal evaluation of the overlap kernel
\begin{equation}\label{eq:kappa.def.recall}
\kappa_b(X,Y) = \int_{\mathcal{S}_{++}^d} K_{\nu(b,d), \, bS}(X) \, K_{\nu(b,d), \, bS}(Y) \, \rd S, \qquad X,Y\in \mathcal{S}_{++}^d,
\end{equation}
and it yields the small-$b$ asymptotics of the diagonal moments $\EE\{\kappa_b(\mathfrak{X},\mathfrak{X})^p\}$.

\begin{lemma}\label{lem:kappa.diagonal.p.moment}\addcontentsline{toc}{subsection}{Lemma~\thelemma}
Fix $d\in \N$. Then, for every $b\in(0,\infty)$ and every $X\in \mathcal{S}_{++}^d$,
\begin{equation}\label{eq:kappa.diagonal.identity}
\kappa_b(X,X) = A_d(b) \, |X|^{-(d+1)/2},
\end{equation}
with
\begin{equation}\label{eq:Adb.def}
A_d(b) \leqdef b^{-r} \, 2^{-d/b-2r} \, \frac{\Gamma_d\left(1/b + (d+1)/2\right)}{\Gamma_d\left(1/(2b) + (d+1)/2\right)^2}.
\end{equation}
Then, as $b\downarrow 0$,
\begin{equation}\label{eq:Adb.asymptotic}
A_d(b) = c_d \, b^{-r/2}\{1 + \OO(b)\}, \qquad c_d \leqdef 2^{-d(d+2)/2} \, \pi^{-r/2}.
\end{equation}
In particular, let $p\in \N$ and let $\mathfrak{X}\sim f$ be a random matrix with density $f$ supported on $\mathcal{S}_{++}^d$. Assume the weighted integrability condition
\begin{equation}\label{eq:Mp.diag.def}
V_{1,p} \leqdef \int_{\mathcal{S}_{++}^d} f(S) \, |S|^{-p(d+1)/2} \, \rd S < \infty.
\end{equation}
Then, for every $b\in(0,\infty)$,
\begin{equation}\label{eq:Ekappa.diag.p.exact}
\EE\{\kappa_b(\mathfrak{X},\mathfrak{X})^p\} = A_d(b)^p \, V_{1,p},
\end{equation}
and, as $b\downarrow 0$,
\begin{equation}\label{eq:Ekappa.diag.p.asymptotic}
\EE\{\kappa_b(\mathfrak{X},\mathfrak{X})^p\} = c_d^{ \, p} \, b^{-pr/2} \, V_{1,p} \{1 + \OO(b)\}.
\end{equation}
\end{lemma}

\begin{proof}[Proof of Lemma~\ref{lem:kappa.diagonal.p.moment}]
By \eqref{eq:kappa.b.closed.form} in the proof of Proposition~\ref{prop:test.statistic} with $b_{\ell_1} = b_{\ell_2} = b$ and $X = Y$, the overlap kernel $\kappa_b$ on the diagonal admits the closed form
\[
\kappa_b(X,X)
= \frac{b^{-r}}{2^{r}} \frac{\Gamma_d\left(1/b + (d+1)/2\right)}{\Gamma_d\left(1/(2b) + (d+1)/2\right)^2} \frac{|X|^{1/(2b)} \, |X|^{1/(2b)}}{|2X|^{1/b + (d+1)/2}}
= A_d(b) \, |X|^{-(d+1)/2}, \qquad X\in \mathcal{S}_{++}^d,
\]
which is \eqref{eq:kappa.diagonal.identity}.

Next, we want to prove \eqref{eq:Adb.asymptotic}. Let $z \leqdef 1/b\to\infty$ as $b\downarrow 0$. Using the product representation of the multivariate gamma function (see \eqref{eq:multivariate.gamma}) in the ratio appearing in \eqref{eq:Adb.def} yields
\[
\frac{\Gamma_d\left(z + (d+1)/2\right)}{\Gamma_d\left(z/2 + (d+1)/2\right)^2}
= \pi^{-d(d-1)/4} \prod_{i=1}^d \frac{\Gamma(z + c_i)}{\Gamma(z/2 + c_i)^2}, \qquad c_i \leqdef \frac{d-i+2}{2}.
\]

We now apply Stirling's formula. For fixed $c > 0$, $\Gamma(w) = \sqrt{2\pi} \, w^{w-1/2}e^{-w}\{1 + \OO(1/w)\}$ as $w\to\infty$, so that
\[
\Gamma(z + c) = \sqrt{2\pi} \, z^{z + c-1/2}e^{-z}\{1 + \OO(1/z)\}, \qquad
\Gamma(z/2 + c) = \sqrt{2\pi} \, (z/2)^{z/2 + c-1/2}e^{-z/2}\{1 + \OO(1/z)\}, \qquad z\to\infty.
\]
Therefore,
\[
\frac{\Gamma(z + c)}{\Gamma(z/2 + c)^2}
= \frac{\sqrt{2\pi} \, z^{z + c-1/2}e^{-z}\{1 + \OO(1/z)\}}{(2\pi) \, (z/2)^{z + 2c-1}e^{-z}\{1 + \OO(1/z)\}}
= \frac{2^{z + 2c-1}}{\sqrt{2\pi}} \, z^{-c + 1/2}\{1 + \OO(1/z)\}.
\]
Multiplying over $i=1,\dots,d$, with $c = c_i$, gives
\[
\frac{\Gamma_d\left(z + (d+1)/2\right)}{\Gamma_d\left(z/2 + (d+1)/2\right)^2}
= \pi^{-d(d-1)/4} \left(\frac{2^{z}}{\sqrt{2\pi}}\right)^d 2^{\sum_{i=1}^d (2c_i-1)} z^{-\sum_{i=1}^d c_i + d/2}\{1 + \OO(1/z)\}.
\]
Compute the sums:
\[
\sum_{i=1}^d (2c_i-1) = \sum_{i=1}^d (d-i+1) = \frac{d(d+1)}{2} = r, \qquad
\sum_{i=1}^d c_i = \sum_{i=1}^d \frac{d-i+2}{2} = \frac{d(d+3)}{4},
\]
hence
\[
-\sum_{i=1}^d c_i + \frac{d}{2} = -\frac{d(d+3)}{4} + \frac{2d}{4} = -\frac{d(d+1)}{4} = -\frac{r}{2}.
\]
Also,
\[
\pi^{-d(d-1)/4}\left(\frac{1}{\sqrt{2\pi}}\right)^d = 2^{-d/2} \, \pi^{-d(d-1)/4-d/2} = 2^{-d/2} \, \pi^{-r/2}.
\]
Putting these together, we obtain
\begin{equation}\label{eq:ratio.mv.gamma.asymp}
\frac{\Gamma_d\left(z + (d+1)/2\right)}{\Gamma_d\left(z/2 + (d+1)/2\right)^2} = 2^{-d/2} \, \pi^{-r/2} \, 2^{dz} \, 2^{r} \, z^{-r/2} \{1 + \OO(1/z)\}.
\end{equation}
Substituting $z = 1/b$ in \eqref{eq:ratio.mv.gamma.asymp} yields
\begin{equation}\label{eq:ratio.mv.gamma.asymp.b}
\frac{\Gamma_d\left(1/b + (d+1)/2\right)}{\Gamma_d\left(1/(2b) + (d+1)/2\right)^2} = 2^{-d/2} \, \pi^{-r/2} \, 2^{d/b} \, 2^{r} \, b^{r/2} \{1 + \OO(b)\}.
\end{equation}

Finally, insert \eqref{eq:ratio.mv.gamma.asymp.b} into the definition \eqref{eq:Adb.def}:
\[
A_d(b)
= b^{-r} \, 2^{-d/b-2r} \left[2^{-d/2} \, \pi^{-r/2} \, 2^{d/b} \, 2^{r} \, b^{r/2} \{1 + \OO(b)\}\right]
= b^{-r/2} \, 2^{-r-d/2} \, \pi^{-r/2}\{1 + \OO(b)\}.
\]
Since $r + d/2 = d(d+2)/2$, we have $2^{-r-d/2} = 2^{-d(d+2)/2}$, which proves \eqref{eq:Adb.asymptotic}.

Finally, \eqref{eq:Ekappa.diag.p.exact} follows from \eqref{eq:kappa.diagonal.identity} by raising both sides to the power $p$ and taking expectations:
\[
\EE\{\kappa_b(\mathfrak{X},\mathfrak{X})^p\}
= A_d(b)^p \, \EE\big\{|\mathfrak{X}|^{-p(d+1)/2}\big\}
= A_d(b)^p \int_{\mathcal{S}_{++}^d} f(S) \, |S|^{-p(d+1)/2} \, \rd S
= A_d(b)^p \, V_{1,p},
\]
where the finiteness is guaranteed by \eqref{eq:Mp.diag.def}. Combining this identity with \eqref{eq:Adb.asymptotic}
yields \eqref{eq:Ekappa.diag.p.asymptotic}.
\end{proof}

The next lemma establishes the small-$b$ asymptotics of the off-diagonal moments $\EE\{\kappa_b(\mathfrak{X},\mathfrak{X}')^p\}$ (with $\smash{\mathfrak{X},\mathfrak{X}' \stackrel{\mathrm{iid}}{\sim} f}$ independent), by showing that the overlap kernel localizes in a $\sqrt{b}$-neighborhood of the diagonal $X = Y$ and yields an explicit Gaussian constant.

\begin{lemma}\label{lem:kappa.offdiag.p.moment}\addcontentsline{toc}{subsection}{Lemma~\thelemma}
Fix $d \in \N$ and let $p \in \N$. Recall that $\kappa_b$ is the overlap kernel in \eqref{eq:kappa.def.recall}. Let $\mathfrak{X},\mathfrak{X}'$ be independent random matrices with common density $f$ supported on $\mathcal{S}_{++}^d$.
Assume that $f$ is bounded and continuous on $\mathcal{S}_{++}^d$, and that the weighted $L^2$-moment
\begin{equation}\label{eq:M2p.offdiag.def}
V_{2,p} \leqdef \int_{\mathcal{S}_{++}^d} f(S)^2 \, |S|^{-(p - 1)(d + 1)/2} \, \rd S < \infty.
\end{equation}
Define the Gaussian integral
\begin{equation}\label{eq:Jdp.def}
J_{d,p} \leqdef \int_{\mathcal{S}^{d}} \exp\left\{-\frac{p}{8}\tr(U^2)\right\} \, \rd U.
\end{equation}
Then, as $b \downarrow 0$,
\begin{equation}\label{eq:Ekappa.offdiag.p.asymptotic}
\EE\{\kappa_b(\mathfrak{X},\mathfrak{X}')^p\} = c_d^{ \, p} \, J_{d,p} \, b^{-(p - 1)r/2} \, V_{2,p} \, \{1 + \oo(1)\},
\end{equation}
where $c_d = 2^{-d(d + 2)/2}\pi^{-r/2}$ as in \eqref{eq:Adb.asymptotic}. Moreover, the integral \eqref{eq:Jdp.def} evaluates explicitly as
\begin{equation}\label{eq:Jdp.closed.form}
J_{d,p}
= \left(\sqrt{\frac{8\pi}{p}}\right)^{d}\left(\sqrt{\frac{4\pi}{p}}\right)^{d(d - 1)/2}
= 2^{d(d + 2)/2} \, \pi^{r/2} \, p^{-r/2}.
\end{equation}
\end{lemma}

\begin{proof}[Proof of Lemma~\ref{lem:kappa.offdiag.p.moment}]
Throughout the proof, let
\[
q \leqdef \frac{(p - 1)(d + 1)}{2}, \qquad
I_b(X) \leqdef \int_{\mathcal{S}_{++}^d}\kappa_b(X,Y)^p f(Y) \, \rd Y, \qquad X \in \mathcal{S}_{++}^d.
\]
Then
\begin{equation}\label{eq:kappa.b.I.b}
\EE\{\kappa_b(\mathfrak{X},\mathfrak{X}')^p\} = \int_{\mathcal{S}_{++}^d} f(X) \, I_b(X) \, \rd X.
\end{equation}

We start from the closed-form expression for $\kappa_b$, rewrite it in exponential form using $\Phi(X,Y)$ below, then localize the $(X,Y)$-integral in a $\sqrt{b}$-neighborhood of the diagonal $X = Y$ via a matrix change of variables, control the tails, and finally evaluate the resulting Gaussian integral.

By \eqref{eq:kappa.b.closed.form} in the proof of Proposition~\ref{prop:test.statistic} with $b_{\ell_1} = b_{\ell_2} = b$, the overlap kernel admits the closed form
\begin{equation}\label{eq:kappa.closed.form.appendix}
\kappa_b(X,Y)
= \frac{b^{-r}}{2^{r}} \frac{\Gamma_d\left(1/b + (d + 1)/2\right)}{\Gamma_d\left(1/(2b) + (d + 1)/2\right)^2} \frac{|X|^{1/(2b)} \, |Y|^{1/(2b)}}{|X + Y|^{1/b + (d + 1)/2}}, \qquad X,Y \in \mathcal{S}_{++}^d.
\end{equation}
Define
\begin{equation}\label{eq:Phi.def.appendix}
\Phi(X,Y) \leqdef \log|X + Y| - \frac{1}{2}\log|X| - \frac{1}{2}\log|Y| - d\log 2, \qquad X,Y \in \mathcal{S}_{++}^d.
\end{equation}
By concavity of the log-determinant \citep[see, e.g.,][Theorem~7.6.6]{HornJohnson2013}, $\Phi(X,Y) \geq 0$ with equality if and only if $X = Y$. Rewriting the ratio of determinants in \eqref{eq:kappa.closed.form.appendix} gives $|X|^{1/(2b)}|Y|^{1/(2b)}/|X + Y|^{1/b} = 2^{-d/b} \exp\{-\Phi(X,Y)/b\}$. Hence,
\begin{equation}\label{eq:kappa.b.Phi}
\kappa_b(X,Y) = C_d(b) \, |X + Y|^{-(d + 1)/2} \exp\left\{-\frac{1}{b}\Phi(X,Y)\right\},
\end{equation}
where
\[
C_d(b) \leqdef b^{-r} \, 2^{-r - d/b} \, \frac{\Gamma_d\left(1/b + (d + 1)/2\right)}{\Gamma_d\left(1/(2b) + (d + 1)/2\right)^2}.
\]
Using \eqref{eq:ratio.mv.gamma.asymp.b} (from the proof of Lemma~\ref{lem:kappa.diagonal.p.moment}), we obtain
\[
C_d(b) = b^{-r} \, 2^{-r - d/b}\left\{2^{-d/2}\pi^{-r/2} \, 2^{d/b} \, 2^{r} \, b^{r/2}\{1 + \OO(b)\}\right\} = 2^{-d/2}\pi^{-r/2} \, b^{-r/2} \{1 + \OO(b)\}.
\]
Set $\tilde{c}_d \leqdef 2^{-d/2}\pi^{-r/2}$. Then
\begin{equation}\label{eq:kappa.exp.form.appendix.simplified}
\kappa_b(X,Y) = \tilde{c}_d \, b^{-r/2} \, |X + Y|^{-(d + 1)/2} \exp\left\{-\frac{1}{b}\Phi(X,Y)\right\}\{1 + \OO(b)\},
\end{equation}
uniformly over $(X,Y)$ ranging in any fixed compact subset of $\mathcal{S}_{++}^d\times\mathcal{S}_{++}^d$.

Fix a compact set $K \subseteq \mathcal{S}_{++}^d$. We shall derive an asymptotic formula for $I_b(X)$ uniformly for $X \in K$. For $b > 0$, let
\[
\mathcal{D}_b \leqdef \{U \in \mathcal{S}^d : I + \sqrt{b} \, U \in \mathcal{S}_{++}^d\}.
\]
For every $X \in K$, parameterize $Y \in \mathcal{S}_{++}^d$ by
\begin{equation}\label{eq:Y.param.appendix}
Y = X^{1/2}(I + \sqrt{b} \, U)X^{1/2}, \qquad U \in \mathcal{D}_b.
\end{equation}
This defines a bijection from $\mathcal{D}_b$ onto $\mathcal{S}_{++}^d$. Moreover, $|Y| = |X| \, |I + \sqrt{b} \, U|$ and $|X + Y| = |X| \, |2I + \sqrt{b} \, U|$. For all sufficiently small $b$, $\{U \in \mathcal{S}^d : \|U\|_F \leq b^{-1/4}\} \subseteq \mathcal{D}_b$, because $\lambda_{\min}(I + \sqrt{b} \, U) \geq 1 - b^{1/4} > 0$ on this set.
A Taylor expansion of $\log|\cdot|$ at the identity yields, uniformly for $\|U\|_F \leq b^{-1/4}$,
\[
\begin{aligned}
\log|I + \sqrt{b} \, U|
&= \sqrt{b} \, \tr(U) - \frac{b}{2}\tr(U^2) + \OO(b^{3/2}\|U\|_F^3), \\
\log|2I + \sqrt{b} \, U|
&= d\log 2 + \frac{\sqrt{b}}{2}\tr(U) - \frac{b}{8}\tr(U^2) + \OO(b^{3/2}\|U\|_F^3).
\end{aligned}
\]
Subtracting as in \eqref{eq:Phi.def.appendix} shows that, uniformly for $\|U\|_F \leq b^{-1/4}$,
\begin{equation}\label{eq:Phi.expansion.appendix}
\Phi(X,Y) = \frac{b}{8}\tr(U^2) + \OO(b^{3/2}\|U\|_F^3).
\end{equation}
Also,
\begin{equation}\label{eq:det.factor.p.appendix}
|X + Y|^{-p(d + 1)/2}
= |X|^{-p(d + 1)/2} \, |2I + \sqrt{b} \, U|^{-p(d + 1)/2}
= 2^{-pr} \, |X|^{-p(d + 1)/2} \{1 + \OO(\sqrt{b}\|U\|_F)\},
\end{equation}
uniformly for $\|U\|_F \leq b^{-1/4}$.
Combining \eqref{eq:kappa.exp.form.appendix.simplified}, \eqref{eq:Phi.expansion.appendix}, and \eqref{eq:det.factor.p.appendix} gives the local approximation: for any fixed $M > 0$,
\begin{equation}\label{eq:kappa.p.local.appendix}
\kappa_b\big(X, \, X^{1/2}(I + \sqrt{b} \, U)X^{1/2}\big)^p
= c_d^{ \, p} \, b^{-pr/2} \, |X|^{-p(d + 1)/2} \exp\left\{-\frac{p}{8}\tr(U^2)\right\} \{1 + \oo(1)\},
\end{equation}
uniformly over $X \in K$ and $\|U\|_F \leq M$, where the $\oo(1)$ depends on $M$, and we used $\tilde{c}_d^{ \, p} \, 2^{-pr} = c_d^{ \, p}$.

Fix $M > 0$ and let $\delta \in (0,1)$. For all sufficiently small $b$, the first three sets below are contained in $\mathcal{D}_b$, since $\lambda_{\min}(I + \sqrt{b} \, U) \geq 1 - \sqrt{b}\|U\|_F > 0$ whenever $\sqrt{b}\|U\|_F \leq \delta$. We therefore decompose $\mathcal{D}_b$ into the regions
\[
\|U\|_F \leq M, \qquad M < \|U\|_F \leq b^{-1/4}, \qquad b^{-1/4} < \|U\|_F \leq \delta/\sqrt{b}, \qquad \mathcal{D}_b \cap \{\|U\|_F > \delta/\sqrt{b}\},
\]
where $\delta$ will be chosen small enough below.

On the main region $\|U\|_F \leq M$, continuity of $f$ and the compactness of $K$ imply
\begin{equation}\label{eq:continuity.compact}
f\big(X^{1/2}(I + \sqrt{b} \, U)X^{1/2}\big) = f(X) + \oo(1),
\end{equation}
uniformly for $X \in K$ and $\|U\|_F \leq M$. The Jacobian of \eqref{eq:Y.param.appendix} is
\begin{equation}\label{eq:Jacobian.appendix}
\rd Y = b^{r/2} \, |X|^{(d + 1)/2} \, \rd U;
\end{equation}
see, e.g., \citet[Theorem~2.1.6]{Muirhead1982}. Using \eqref{eq:kappa.p.local.appendix}, \eqref{eq:continuity.compact}, and \eqref{eq:Jacobian.appendix} yields
\[
\int_{\|U\|_F\leq M}\kappa_b(X,Y)^p f(Y) \, \rd Y = c_d^{ \, p} \, b^{-(p - 1)r/2} \, |X|^{-q}\{f(X)J_{d,p}(M) + \oo(1)\},
\]
uniformly for $X \in K$, where
\[
J_{d,p}(M) \leqdef \int_{\|U\|_F\leq M} \exp\left\{-\frac{p}{8}\tr(U^2)\right\} \, \rd U.
\]

For the intermediate region $M < \|U\|_F \leq b^{-1/4}$, \eqref{eq:Phi.expansion.appendix} implies that for $b$ small enough (depending on $M$),
\begin{equation}\label{eq:exponent.lower.bound}
\frac{\Phi(X,Y)}{b} \geq \frac{1}{16}\tr(U^2), \qquad \text{on } \{M < \|U\|_F \leq b^{-1/4}\},
\end{equation}
uniformly for $X \in K$. Moreover, on the region $\{\|U\|_F \leq b^{-1/4}\}$, we have $\sqrt{b}\|U\|_F \leq b^{1/4}$, so $|2I + \sqrt{b} \, U|^{-p(d + 1)/2}$ is bounded above and below by positive constants for all sufficiently small $b$. Using \eqref{eq:kappa.exp.form.appendix.simplified}, \eqref{eq:exponent.lower.bound}, the boundedness of $f$, and \eqref{eq:Jacobian.appendix}, we obtain
\[
\sup_{X\in K} \int_{M < \|U\|_F\leq b^{-1/4}}\kappa_b(X,Y)^p f(Y) \, \rd Y
\ll_K b^{-(p - 1)r/2} \int_{\|U\|_F > M} \exp\left\{-\frac{p}{16}\tr(U^2)\right\} \rd U.
\]
The Gaussian tail on the right tends to $0$ as $M \to \infty$.

Next, choose $\delta \in (0,1)$ small enough that the Taylor expansion \eqref{eq:Phi.expansion.appendix} remains valid whenever $\sqrt{b} \, \|U\|_F \leq \delta$, and so that the remainder term there is bounded by $(b/16)\tr(U^2)$. Then, on the region $b^{-1/4} < \|U\|_F \leq \delta/\sqrt{b}$, we still have $\Phi(X,Y)/b \geq \tr(U^2)/16$, uniformly for $X \in K$, and therefore
\[
\sup_{X\in K} \int_{b^{-1/4} < \|U\|_F\leq \delta/\sqrt{b}}\kappa_b(X,Y)^p f(Y) \, \rd Y
\ll_K b^{-(p - 1)r/2} \int_{\|U\|_F > b^{-1/4}} \exp\left\{-\frac{p}{16}\tr(U^2)\right\} \rd U
= \oo\big(b^{-(p - 1)r/2}\big).
\]

It remains to handle the genuinely far region $\mathcal{D}_b \cap \{\|U\|_F > \delta/\sqrt{b}\}$. For this, define
\[
\Psi(T) \leqdef \Phi(I,T)
= \log|I + T| - \frac12\log|T| - d\log 2, \qquad T \in \mathcal{S}_{++}^d.
\]
If $\lambda_1(T),\ldots,\lambda_d(T)$ are the eigenvalues of $T$, then
\[
\Psi(T) = \sum_{j=1}^d \log\left(\frac{1 + \lambda_j(T)}{2\sqrt{\lambda_j(T)}}\right).
\]
Hence $\Psi(T) \geq 0$, with equality only at $T = I$, and moreover $\Psi(T) \to \infty$ as $T$ approaches the boundary of $\mathcal{S}_{++}^d$ or $\|T\|_F \to \infty$. Therefore
\[
m_\delta \leqdef \inf_{\|T - I\|_F\geq \delta}\Psi(T) > 0.
\]
Now write $Y = X^{1/2}TX^{1/2}$, so that $\Phi(X,Y) = \Psi(T)$ and $\rd Y = |X|^{(d + 1)/2} \rd T$. On $\mathcal{D}_b \cap \{\|U\|_F > \delta/\sqrt{b}\}$, equivalently on $\{T \in \mathcal{S}_{++}^d : \|T - I\|_F > \delta\}$, we have $\Psi(T) \geq m_\delta$. Hence, using \eqref{eq:kappa.b.Phi},
\[
\int_{\mathcal{D}_b\cap\{\|U\|_F > \delta/\sqrt{b}\}}\kappa_b(X,Y)^p f(Y) \, \rd Y
\leq \|f\|_\infty \, C_d(b)^p \, e^{-p m_\delta/(2b)} \, |X|^{-q} \int_{\mathcal{S}_{++}^d}|I + T|^{-p(d + 1)/2}e^{-p\Psi(T)/(2b)} \, \rd T.
\]
Using $e^{-p\Psi(T)/(2b)} = 2^{pd/(2b)}|T|^{p/(4b)}|I + T|^{-p/(2b)}$, the last integral equals
\[
2^{pd/(2b)} \int_{\mathcal{S}_{++}^d} |T|^{p/(4b)}|I + T|^{-p/(2b) - p(d + 1)/2} \, \rd T.
\]
For all sufficiently small $b$, the matrix beta integral \citep[see, e.g.,][\S35.3(ii)]{Richards2010}, applied with $a = p/(4b) + (d + 1)/2$ and $c = p/(4b) + q$, shows that this is
\[
2^{pd/(2b)} \frac{\Gamma_d\left(p/(4b) + (d + 1)/2\right) \Gamma_d\left(p/(4b) + q\right)}{\Gamma_d\left(p/(2b) + p(d + 1)/2\right)},
\]
where both beta parameters exceed $(d - 1)/2$. This expression is $\OO(b^{r/2})$ by Stirling's formula. Since $C_d(b)^p = \OO(b^{-pr/2})$ and $e^{-p m_\delta/(2b)} \to 0$ as $b \downarrow 0$, it follows that
\[
\sup_{X\in K} \int_{\mathcal{D}_b\cap\{\|U\|_F > \delta/\sqrt{b}\}}\kappa_b(X,Y)^p f(Y) \, \rd Y = \oo\big(b^{-(p - 1)r/2}\big).
\]

Combining the four regions, letting $b \downarrow 0$ first and then $M \to \infty$, we obtain
\[
\sup_{X\in K} \left|I_b(X) - c_d^{ \, p} b^{-(p - 1)r/2} f(X) |X|^{-q} J_{d,p}\right| = \oo\big(b^{-(p - 1)r/2}\big).
\]

We now integrate over $X$. Let
\[
K_m \leqdef \{X \in \mathcal{S}_{++}^d:\ \lambda_{\min}(X) \geq m^{-1},\ \lambda_{\max}(X) \leq m\}, \qquad m \in \N.
\]
Then $K_m \subseteq \mathcal{S}_{++}^d$ and $K_m \uparrow \mathcal{S}_{++}^d$. For each fixed $m$, the preceding uniform estimate on $K_m$ gives
\begin{equation}\label{eq:convergence.K.m}
b^{(p - 1)r/2}\int_{K_m} f(X)I_b(X) \, \rd X \longrightarrow c_d^{ \, p}J_{d,p}\int_{K_m} f(X)^2|X|^{-q} \, \rd X.
\end{equation}

It remains to control the contribution of $K_m^c$. For this, define
\begin{equation}\label{eq:M.d.p.b}
M_{d,p}(b) \leqdef \int_{\mathcal{S}_{++}^d}\kappa_b(X,Y)^p |Y|^q \, \rd Y.
\end{equation}
Using the closed form expression of the overlap kernel in \eqref{eq:kappa.closed.form.appendix}, the congruence change of variables $Y = X^{1/2}TX^{1/2}$, with $\rd Y = |X|^{(d + 1)/2} \rd T$, and the matrix beta integral \citep[see, e.g.,][\S35.3(ii)]{Richards2010}, this quantity is finite and independent of $X$ for all sufficiently small $b$. The beta parameters are $a = p/(2b) + p(d + 1)/2$ and $c = p/(2b)$, which exceed $(d - 1)/2$ for all sufficiently small $b$, and
\begin{equation}\label{eq:M.d.p.b.OG}
M_{d,p}(b)
= \left\{\frac{b^{-r}}{2^{r}} \frac{\Gamma_d\left(1/b + (d + 1)/2\right)}{\Gamma_d\left(1/(2b) + (d + 1)/2\right)^2}\right\}^p \frac{\Gamma_d\left(p/(2b) + p(d + 1)/2\right) \Gamma_d\left(p/(2b)\right)}{\Gamma_d\left(p/b + p(d + 1)/2\right)}.
\end{equation}
From \eqref{eq:ratio.mv.gamma.asymp.b}, the term inside the braces is $\OO(b^{-r/2} 2^{d/b})$ as $b \downarrow 0$, so the first factor is $\OO(b^{-pr/2} 2^{pd/b})$. For the second factor, applying Stirling's formula to the product representation of $\Gamma_d$ yields
\[
\Gamma_d(z + c)
\asymp \prod_{i=1}^d \left( \sqrt{2\pi} \, z^{z + c - i/2} e^{-z} \right)
\asymp z^{dz + dc - r/2} e^{-dz}, \qquad z \to \infty.
\]
Taking $z = p/b$, the second factor in \eqref{eq:M.d.p.b.OG} behaves asymptotically as
\[
\frac{(z/2)^{dz/2 + dp(d + 1)/2 - r/2} e^{-dz/2} \cdot (z/2)^{dz/2 - r/2} e^{-dz/2}}{z^{dz + dp(d + 1)/2 - r/2} e^{-dz}}
\asymp z^{-r/2} 2^{-dz}
= p^{-r/2} b^{r/2} 2^{-pd/b}
\asymp b^{r/2} 2^{-pd/b}.
\]
The exponential terms $2^{pd/b}$ and $2^{-pd/b}$ cancel out perfectly, leaving
\begin{equation}\label{eq:M.d.p.b.bound}
M_{d,p}(b) \ll b^{-pr/2} 2^{pd/b} \times b^{r/2} 2^{-pd/b} = b^{-(p - 1)r/2}, \qquad b \downarrow 0.
\end{equation}

Now, by Cauchy--Schwarz,
\[
I_b(X)^2
= \left(\int \kappa_b(X,Y)^p f(Y) \, \rd Y\right)^2
\leq \left(\int \kappa_b(X,Y)^p |Y|^q \, \rd Y\right) \left(\int \kappa_b(X,Y)^p f(Y)^2|Y|^{-q} \, \rd Y\right).
\]
Thus
\[
I_b(X)^2 \leq M_{d,p}(b)\int_{\mathcal{S}_{++}^d}\kappa_b(X,Y)^p f(Y)^2|Y|^{-q} \, \rd Y.
\]
Multiply this last expression by $|X|^q$ and integrate over $X \in \mathcal{S}_{++}^d$. Using Tonelli's theorem, and the symmetry $\kappa_b(X,Y) = \kappa_b(Y,X)$, we get
\[
\int_{\mathcal{S}_{++}^d} I_b(X)^2 |X|^q \, \rd X
\leq M_{d,p}(b)^2 \int_{\mathcal{S}_{++}^d} f(Y)^2|Y|^{-q} \, \rd Y
= M_{d,p}(b)^2 \, V_{2,p}.
\]
Therefore, again by Cauchy--Schwarz,
\[
\int_{K_m^c} f(X)I_b(X) \, \rd X
\leq \left(\int_{K_m^c} f(X)^2|X|^{-q} \, \rd X\right)^{1/2} \left(\int_{\mathcal{S}_{++}^d} I_b(X)^2 |X|^q \, \rd X\right)^{1/2}
\leq M_{d,p}(b) \, V_{2,p}^{1/2} \left(\int_{K_m^c} f(X)^2 |X|^{-q} \, \rd X\right)^{1/2}.
\]
Since $b^{(p - 1)r/2}M_{d,p}(b) \ll 1$ by \eqref{eq:M.d.p.b.bound}, there exists $b_0 > 0$ such that
\begin{equation}\label{eq:tail.bound.K.m.c}
\sup_{0 < b < b_0} b^{(p - 1)r/2}\int_{K_m^c} f(X)I_b(X) \, \rd X
\ll V_{2,p}^{1/2} \left(\int_{K_m^c} f(X)^2|X|^{-q} \, \rd X\right)^{1/2},
\end{equation}
and the right-hand side tends to $0$ as $m \to \infty$ by \eqref{eq:M2p.offdiag.def}.

Combining the convergence on $K_m$ in \eqref{eq:convergence.K.m} with the tail bound on $K_m^c$ in \eqref{eq:tail.bound.K.m.c}, and then letting $m \to \infty$, it follows from \eqref{eq:kappa.b.I.b} that
\[
b^{(p - 1)r/2} \, \EE\{\kappa_b(\mathfrak{X},\mathfrak{X}')^p\}
\longrightarrow c_d^{ \, p}J_{d,p}\int_{\mathcal{S}_{++}^d} f(X)^2|X|^{-q} \, \rd X
= c_d^{ \, p}J_{d,p}V_{2,p},
\]
which is exactly \eqref{eq:Ekappa.offdiag.p.asymptotic}.

It remains to compute $J_{d,p}$. Since for $U \in \mathcal{S}^d$, $\tr(U^2) = \sum_{i=1}^d U_{ii}^2 + 2\sum_{1\leq i < j\leq d}U_{ij}^2$, the integral in \eqref{eq:Jdp.def} factorizes:
\[
J_{d,p}
= \left(\prod_{i=1}^d \int_{-\infty}^{\infty}e^{-(p/8)u^2}\rd u\right) \left(\prod_{1\leq i < j\leq d} \int_{-\infty}^{\infty}e^{-(p/4)u^2}\rd u\right)
= \left(\sqrt{\frac{8\pi}{p}}\right)^{d}\left(\sqrt{\frac{4\pi}{p}}\right)^{d(d - 1)/2},
\]
which is \eqref{eq:Jdp.closed.form}.
\end{proof}

The third lemma controls the $p$th absolute moments of the Hoeffding-centered kernel $h_b$ on and off the diagonal in terms of the corresponding moments of $\kappa_b$, showing that centering does not increase the order of magnitude. Recall the overlap kernel $\kappa_b$ from \eqref{eq:kappa.def.recall}, and recall $m_b(X) = \EE\{\kappa_b(X,\mathfrak{X}')\}$, $m_{b,0} = \EE\{\kappa_b(\mathfrak{X},\mathfrak{X}')\}$, and the Hoeffding-centered kernel
\begin{equation}\label{eq:hb.def.appendix}
h_b(X,Y) = \kappa_b(X,Y) - m_b(X) - m_b(Y) + m_{b,0}, \qquad X,Y \in \mathcal{S}_{++}^d.
\end{equation}

\begin{lemma}\label{lem:h.centered.p.moment.control}\addcontentsline{toc}{subsection}{Lemma~\thelemma}
Fix $d \in \N$ and let $p \in \N$. Let $\mathfrak{X},\mathfrak{X}'$ be iid random matrices with common density $f$ supported on $\mathcal{S}_{++}^d$. Assume that $f$ is bounded and continuous on $\mathcal{S}_{++}^d$, and that the integrability conditions \eqref{eq:Mp.diag.def} and \eqref{eq:M2p.offdiag.def} hold, i.e.,
\begin{equation}\label{eq:finite.asump}
V_{1,p} = \int_{\mathcal{S}_{++}^d} f(S) \, |S|^{-p(d + 1)/2} \, \rd S < \infty, \qquad
V_{2,p} = \int_{\mathcal{S}_{++}^d} f(S)^2 \, |S|^{-(p - 1)(d + 1)/2} \, \rd S < \infty,
\end{equation}
so that Lemmas~\ref{lem:kappa.diagonal.p.moment} and \ref{lem:kappa.offdiag.p.moment} apply for this~$p$. Then, as $b \downarrow 0$,
\begin{equation}\label{eq:hb.diag.p.order}
\EE\big\{|h_b(\mathfrak{X},\mathfrak{X})|^p\big\} \ll b^{-pr/2}, \qquad
\EE\big\{|h_b(\mathfrak{X},\mathfrak{X}')|^p\big\} \ll b^{-(p - 1)r/2}.
\end{equation}
\end{lemma}

\begin{proof}[Proof of Lemma~\ref{lem:h.centered.p.moment.control}]
We repeatedly use the elementary inequality
\begin{equation}\label{eq:lp.sum.bound}
|a_1 + \cdots + a_k|^p \leq k^{p - 1} \sum_{i=1}^k |a_i|^p, \qquad p \geq 1,
\end{equation}
and Jensen's inequality in conditional form.

For each fixed $b \in (0,\infty)$, determinant monotonicity on $\mathcal{S}_{++}^d$ gives $|X + Y|^{1/b} \geq |X|^{1/(2b)}|Y|^{1/(2b)}$ and $|X + Y|^{(d + 1)/2} \geq |X|^{(d + 1)/2}$. Consequently, the closed-form expression for $\kappa_b$ in \eqref{eq:kappa.closed.form.appendix} implies $\kappa_b(X,Y)^p \ll_b |X|^{-p(d + 1)/2}$ for $X,Y \in \mathcal{S}_{++}^d$. Hence $\EE\{\kappa_b(\mathfrak{X},\mathfrak{X}')^p\} < \infty$ by the finiteness of $V_{1,p}$ in \eqref{eq:finite.asump}, so the conditional expectations below are well defined.

We first control the centering terms. Given that $m_b(\mathfrak{X}) = \EE\{\kappa_b(\mathfrak{X},\mathfrak{X}')\mid \mathfrak{X}\}$, Jensen's inequality yields
\begin{equation}\label{eq:mb.p.bound.appendix}
\EE\{|m_b(\mathfrak{X})|^p\}
= \EE\Big[\big|\EE\{\kappa_b(\mathfrak{X},\mathfrak{X}')\mid \mathfrak{X}\}\big|^p\Big]
\leq \EE\Big[\EE\{|\kappa_b(\mathfrak{X},\mathfrak{X}')|^p\mid \mathfrak{X}\}\Big]
= \EE\{\kappa_b(\mathfrak{X},\mathfrak{X}')^p\}.
\end{equation}
Also, since $m_{b,0} = \EE\{\kappa_b(\mathfrak{X},\mathfrak{X}')\}$ is constant,
\begin{equation}\label{eq:mb1.p.bound.appendix}
|m_{b,0}|^p
= \big|\EE\{\kappa_b(\mathfrak{X},\mathfrak{X}')\}\big|^p
\leq \EE\{\kappa_b(\mathfrak{X},\mathfrak{X}')^p\},
\end{equation}
by Jensen's inequality.

For the diagonal moment, note that $h_b(X,X) = \kappa_b(X,X) - 2 m_b(X) + m_{b,0}$. Apply \eqref{eq:lp.sum.bound} with $k = 3$:
\[
|h_b(\mathfrak{X},\mathfrak{X})|^p
\leq 3^{p - 1}\big(\kappa_b(\mathfrak{X},\mathfrak{X})^p + 2^p|m_b(\mathfrak{X})|^p + |m_{b,0}|^p\big),
\]
and use \eqref{eq:mb.p.bound.appendix}--\eqref{eq:mb1.p.bound.appendix} to obtain
\begin{equation}\label{eq:hb.diag.p.bound}
\EE\{|h_b(\mathfrak{X},\mathfrak{X})|^p\}
\leq 3^{p - 1}\big(\EE\{\kappa_b(\mathfrak{X},\mathfrak{X})^p\} + (2^p + 1) \, \EE\{\kappa_b(\mathfrak{X},\mathfrak{X}')^p\}\big).
\end{equation}

For the off-diagonal moment, use \eqref{eq:hb.def.appendix} with $(X,Y) = (\mathfrak{X},\mathfrak{X}')$ and apply \eqref{eq:lp.sum.bound} with $k = 4$:
\[
|h_b(\mathfrak{X},\mathfrak{X}')|^p
\leq 4^{p - 1}\big(\kappa_b(\mathfrak{X},\mathfrak{X}')^p + |m_b(\mathfrak{X})|^p + |m_b(\mathfrak{X}')|^p + |m_{b,0}|^p\big).
\]
Taking expectations and using \eqref{eq:mb.p.bound.appendix}--\eqref{eq:mb1.p.bound.appendix} gives
\begin{equation}\label{eq:hb.offdiag.p.bound}
\EE\{|h_b(\mathfrak{X},\mathfrak{X}')|^p\}
\leq 4^{p - 1}\big(\EE\{\kappa_b(\mathfrak{X},\mathfrak{X}')^p\} + 3 \, \EE\{\kappa_b(\mathfrak{X},\mathfrak{X}')^p\}\big)
= 4^p \, \EE\{\kappa_b(\mathfrak{X},\mathfrak{X}')^p\}.
\end{equation}

Finally, the claim \eqref{eq:hb.diag.p.order} follows by combining \eqref{eq:hb.diag.p.bound}--\eqref{eq:hb.offdiag.p.bound} with Lemma~\ref{lem:kappa.diagonal.p.moment} (using the assumption $V_{1,p} < \infty$ from \eqref{eq:finite.asump}), which gives $\EE\{\kappa_b(\mathfrak{X},\mathfrak{X})^p\} \ll b^{-pr/2}$, and Lemma~\ref{lem:kappa.offdiag.p.moment} (using the assumption $V_{2,p} < \infty$ from \eqref{eq:finite.asump}), which gives $\EE\{\kappa_b(\mathfrak{X},\mathfrak{X}')^p\} \ll b^{-(p - 1)r/2}$.
\end{proof}

The fourth lemma provides a bound on a cycle-type fourth-order product of the Hoeffding-centered kernel $h_b$.

\begin{lemma}\label{lem:h.cycle.fourth.moment}\addcontentsline{toc}{subsection}{Lemma~\thelemma}
Fix $d \in \N$. Let $\mathfrak{X},\mathfrak{X}',\mathfrak{X}'',\mathfrak{X}'''$ be iid random matrices with common density $f$ supported on $\mathcal{S}_{++}^d$. Assume that $f$ is bounded on $\mathcal{S}_{++}^d$ and that $V_{2,2} = \int_{\mathcal{S}_{++}^d} f(S)^2 \, |S|^{-(d + 1)/2} \, \rd S < \infty$. Then, as $b \downarrow 0$,
\begin{equation}\label{eq:h.cycle.bound}
\EE\left[h_b(\mathfrak{X},\mathfrak{X}') \, h_b(\mathfrak{X},\mathfrak{X}'') \, h_b(\mathfrak{X}''',\mathfrak{X}') \, h_b(\mathfrak{X}''',\mathfrak{X}'')\right]
\ll b^{-r/2}.
\end{equation}
\end{lemma}

\begin{proof}[Proof of Lemma~\ref{lem:h.cycle.fourth.moment}]
Let $\mu(\rd S) \leqdef f(S) \, \rd S$ and define the integral operator $\mathcal{H}_b$ on $L^2(\mu)$ by
\[
(\mathcal{H}_b \, g)(X) \leqdef \int_{\mathcal{S}_{++}^d} h_b(X,Y) \, g(Y) \, \mu(\rd Y).
\]

We first bound $\|\mathcal{H}_b\|_{\mathrm{op}}$ uniformly in small $b$ by a Schur-type argument. Using $\kappa_b \geq 0$, $m_b \geq 0$, and $m_{b,0} \geq 0$,
\[
|h_b(X,Y)| \leq \kappa_b(X,Y) + m_b(X) + m_b(Y) + m_{b,0}.
\]
Integrating with respect to $\mu(\rd Y)$ gives
\[
\int_{\mathcal{S}_{++}^d} |h_b(X,Y)| \, \mu(\rd Y)
\leq \int_{\mathcal{S}_{++}^d}\kappa_b(X,Y) \, \mu(\rd Y) + m_b(X) + \int_{\mathcal{S}_{++}^d}m_b(Y) \, \mu(\rd Y) + m_{b,0}
= 2 m_b(X) + 2 m_{b,0},
\]
since $\int_{\mathcal{S}_{++}^d}\kappa_b(X,Y) \, \mu(\rd Y) = m_b(X)$ and $\int_{\mathcal{S}_{++}^d}m_b(Y) \, \mu(\rd Y) = \EE\{m_b(\mathfrak{X})\} = m_{b,0}$. Thus it suffices to bound $\sup_{X\in \mathcal{S}_{++}^d} m_b(X)$ and $m_{b,0}$.

By boundedness of $f$ and $\kappa_b \geq 0$,
\[
m_b(X)
= \int_{\mathcal{S}_{++}^d}\kappa_b(X,Y) \, f(Y) \, \rd Y
\leq \|f\|_{\infty} \int_{\mathcal{S}_{++}^d}\kappa_b(X,Y) \, \rd Y.
\]
Using the closed-form expression for $\kappa_b$ in \eqref{eq:kappa.closed.form.appendix} and the congruence change of variables $Y = X^{1/2}TX^{1/2}$ (whose Jacobian is $\rd Y = |X|^{(d + 1)/2} \rd T$ on $\mathcal{S}_{++}^d$), one checks that for all sufficiently small $b$, $\int_{\mathcal{S}_{++}^d}\kappa_b(X,Y) \rd Y$ is independent of $X$ and equals
\[
\int_{\mathcal{S}_{++}^d}\kappa_b(X,Y) \, \rd Y
= \frac{b^{-r}}{2^r} \, \frac{\Gamma_d(1/(2b))}{\Gamma_d\left(1/(2b) + (d + 1)/2\right)}.
\]
Indeed, after the change $Y = X^{1/2}TX^{1/2}$ the $X$-dependence cancels, and the remaining $T$-integral is a multivariate beta integral:
for $\Re(a),\Re(c) > (d - 1)/2$,
\[
\int_{\mathcal{S}_{++}^d} |T|^{a - (d + 1)/2} \, |I + T|^{-(a + c)} \, \rd T
= \frac{\Gamma_d(a)\Gamma_d(c)}{\Gamma_d(a + c)},
\]
applied with $a = 1/(2b) + (d + 1)/2$ and $c = 1/(2b)$, which is valid for all sufficiently small $b$ since then $a,c > (d - 1)/2$. Therefore,
\begin{equation}\label{eq:int.kappa.dy}
\sup_{X\in \mathcal{S}_{++}^d} m_b(X)
\leq \|f\|_{\infty} \, \frac{b^{-r}}{2^r} \, \frac{\Gamma_d(1/(2b))}{\Gamma_d\left(1/(2b) + (d + 1)/2\right)}.
\end{equation}
Using the product representation of $\Gamma_d$ and Stirling's formula (as in Lemma~\ref{lem:kappa.diagonal.p.moment}), the multivariate-gamma ratio satisfies
\[
\frac{\Gamma_d(z)}{\Gamma_d\left(z + (d + 1)/2\right)} = z^{-r} \{1 + \OO(1/z)\}, \qquad z \to \infty,
\]
so with $z = 1/(2b)$ we obtain
\[
\frac{b^{-r}}{2^r} \, \frac{\Gamma_d(1/(2b))}{\Gamma_d\left(1/(2b) + (d + 1)/2\right)}
= \frac{b^{-r}}{2^r} \, \left(\frac{1}{2b}\right)^{-r} \{1 + \OO(b)\}
= 1 + \OO(b),
\]
and in particular the right-hand side of \eqref{eq:int.kappa.dy} is bounded uniformly for sufficiently small $b$.
Since $m_{b,0} = \EE\{m_b(\mathfrak{X})\} \leq \sup_{X\in \mathcal{S}_{++}^d} m_b(X)$, we conclude that $\smash{\sup_{X\in \mathcal{S}_{++}^d} \int_{\mathcal{S}_{++}^d}|h_b(X,Y)| \, \mu(\rd Y) \ll 1}$ and $\smash{\sup_{Y\in \mathcal{S}_{++}^d} \int_{\mathcal{S}_{++}^d}|h_b(X,Y)| \, \mu(\rd X) \ll 1}$, for all sufficiently small $b$. By the Schur test,
\begin{equation}\label{eq:Schur.test}
\|\mathcal{H}_b\|_{\mathrm{op}} \ll 1.
\end{equation}

Next we show that $\mathcal{H}_b$ is Hilbert--Schmidt and that
\begin{equation}\label{eq:hb.hs.bound}
\|\mathcal{H}_b\|_{\mathrm{HS}}^2 = \EE\{h_b(\mathfrak{X},\mathfrak{X}')^2\} \ll b^{-r/2}.
\end{equation}
Using $|h_b(X,Y)|^2 \leq 4 (\kappa_b(X,Y)^2 + m_b(X)^2 + m_b(Y)^2 + m_{b,0}^2)$, we obtain
\[
\EE\{h_b(\mathfrak{X},\mathfrak{X}')^2\}
\leq 4 \, \big(\EE\{\kappa_b(\mathfrak{X},\mathfrak{X}')^2\} + 2 \, \EE\{m_b(\mathfrak{X})^2\} + m_{b,0}^2\big).
\]
Given that $m_b(\mathfrak{X}) = \EE\{\kappa_b(\mathfrak{X},\mathfrak{X}')\mid \mathfrak{X}\}$ and $m_{b,0} = \EE\{\kappa_b(\mathfrak{X},\mathfrak{X}')\}$, Jensen's inequality yields $\EE\{m_b(\mathfrak{X})^2\} \leq \EE\{\kappa_b(\mathfrak{X},\mathfrak{X}')^2\}$ and $m_{b,0}^2 \leq \EE\{\kappa_b(\mathfrak{X},\mathfrak{X}')^2\}$. Therefore
\begin{equation}\label{eq:h.by.kappa.second.moment}
\EE\{h_b(\mathfrak{X},\mathfrak{X}')^2\} \leq 16 \, \EE\{\kappa_b(\mathfrak{X},\mathfrak{X}')^2\}.
\end{equation}

Set $q = (d + 1)/2$ and recall from \eqref{eq:M.d.p.b} and \eqref{eq:M.d.p.b.bound} that $M_{d,2}(b) = \smash{\int_{\mathcal{S}_{++}^d}\kappa_b(X,Y)^2 \, |Y|^{q} \, \rd Y}$, and that $M_{d,2}(b) \ll b^{-r/2}$ as $b \downarrow 0$. Now, by Cauchy--Schwarz and the symmetry $\kappa_b(X,Y) = \kappa_b(Y,X)$, we get
\[
\begin{aligned}
\EE\{\kappa_b(\mathfrak{X},\mathfrak{X}')^2\}
&= \iint_{\mathcal{S}_{++}^d\times\mathcal{S}_{++}^d}\kappa_b(X,Y)^2 f(X)f(Y) \, \rd X \, \rd Y \\
&\leq \left(\iint \kappa_b(X,Y)^2 f(X)^2|X|^{-q}|Y|^{q} \, \rd X \, \rd Y\right)^{1/2} \left(\iint \kappa_b(X,Y)^2 f(Y)^2|Y|^{-q}|X|^{q} \, \rd X \, \rd Y\right)^{1/2} \\
&= M_{d,2}(b) \, V_{2,2}
\ll b^{-r/2}.
\end{aligned}
\]
Combining this with \eqref{eq:h.by.kappa.second.moment} proves \eqref{eq:hb.hs.bound}.

We now return to the cycle term. Before applying conditioning, we verify absolute integrability. Define $q_b(X) \leqdef \int_{\mathcal{S}_{++}^d} h_b(X,Z)^2 \, \mu(\rd Z)$. By Tonelli's theorem, conditional independence, and the Cauchy--Schwarz inequality,
\[
\begin{aligned}
&\EE\left[\left|h_b(\mathfrak{X},\mathfrak{X}') \, h_b(\mathfrak{X},\mathfrak{X}'') \, h_b(\mathfrak{X}''',\mathfrak{X}') \, h_b(\mathfrak{X}''',\mathfrak{X}'')\right|\right] \\
&\qquad= \EE\left[\left\{\int_{\mathcal{S}_{++}^d}|h_b(\mathfrak{X},Z) \, h_b(\mathfrak{X}''',Z)| \, \mu(\rd Z)\right\}^2\right]
\leq \EE\{q_b(\mathfrak{X})q_b(\mathfrak{X}''')\}
= \left[\EE\{q_b(\mathfrak{X})\}\right]^2
= \|\mathcal{H}_b\|_{\mathrm{HS}}^4
< \infty.
\end{aligned}
\]
Thus, the following conditioning and Fubini steps are justified. Since $\mathfrak{X}',\mathfrak{X}''$ are iid and independent of $(\mathfrak{X},\mathfrak{X}''')$, conditioning gives
\begin{equation}\label{eq:last.display}
\begin{aligned}
\EE\left[h_b(\mathfrak{X},\mathfrak{X}') \, h_b(\mathfrak{X},\mathfrak{X}'') \, h_b(\mathfrak{X}''',\mathfrak{X}') \, h_b(\mathfrak{X}''',\mathfrak{X}'')\right]
&= \EE\left[\left\{\int_{\mathcal{S}_{++}^d} h_b(\mathfrak{X},Z) \, h_b(\mathfrak{X}''',Z) \, \mu(\rd Z)\right\}^2\right] \\
&= \iint_{\mathcal{S}_{++}^d\times\mathcal{S}_{++}^d}
\left\{\int_{\mathcal{S}_{++}^d} h_b(X,Z) \, h_b(X',Z) \, \mu(\rd Z)\right\}^2 \mu(\rd X) \, \mu(\rd X').
\end{aligned}
\end{equation}
Because $\mathcal{H}_b$ is Hilbert--Schmidt, the operator $\mathcal{R}_b \leqdef \mathcal{H}_b\mathcal{H}_b^{\ast}$ is nonnegative definite and trace class, hence also Hilbert--Schmidt. Its kernel at $(X,X')$ is
\[
\int_{\mathcal{S}_{++}^d} h_b(X,Z) \, h_b(X',Z) \, \mu(\rd Z),
\]
so the right-hand side of \eqref{eq:last.display} equals $\|\mathcal{R}_b\|_{\mathrm{HS}}^2 = \tr(\mathcal{R}_b^2)$. Writing the eigenvalues of $\mathcal{R}_b$ as $(\lambda_j)_{j=1}^{\infty}$ yields
\[
\tr(\mathcal{R}_b^2)
= \sum_{j=1}^{\infty} \lambda_j^2 \leq \left(\sup_{j\geq 1} \lambda_j\right)\sum_{j=1}^{\infty} \lambda_j
= \|\mathcal{R}_b\|_{\mathrm{op}} \, \tr(\mathcal{R}_b).
\]
Moreover, $\|\mathcal{R}_b\|_{\mathrm{op}} = \|\mathcal{H}_b\|_{\mathrm{op}}^2$ and $\tr(\mathcal{R}_b) = \|\mathcal{H}_b\|_{\mathrm{HS}}^2$, hence
\begin{equation}\label{eq:cycle.trace.bound}
\left|\EE\left[h_b(\mathfrak{X},\mathfrak{X}') \, h_b(\mathfrak{X},\mathfrak{X}'') \, h_b(\mathfrak{X}''',\mathfrak{X}') \, h_b(\mathfrak{X}''',\mathfrak{X}'')\right]\right|
\leq \|\mathcal{H}_b\|_{\mathrm{op}}^2 \, \|\mathcal{H}_b\|_{\mathrm{HS}}^2.
\end{equation}
Inserting the bounds \eqref{eq:Schur.test} and \eqref{eq:hb.hs.bound} into \eqref{eq:cycle.trace.bound} yields \eqref{eq:h.cycle.bound}.
\end{proof}

\end{appendices}

\section*{Reproducibility}\label{app:reproducibility}
\addcontentsline{toc}{section}{Reproducibility}

The code required to implement our proposed test and generate all tables in the simulation study is publicly accessible at \href{https://github.com/FredericOuimetMcGill/WishartKDETwoSampleTest}{https://github.com/FredericOuimetMcGill/WishartKDETwoSampleTest}.

\section*{Statement of AI use}
\addcontentsline{toc}{section}{Statement of AI use}

The simulation study in Section~\ref{sec:power.study} was conducted with the help of ChatGPT 5.6 Sol. The author takes full accountability for the correctness of the $\mathsf{R}$ code, the study design and the conclusions.

\section*{Funding}
\addcontentsline{toc}{section}{Funding}

The author is supported by the Natural Sciences and Engineering Research Council of Canada (NSERC) through Discovery Grant RGPIN-2026-04471 and Discovery Launch Supplement DGECR-2026-00449.

\section*{Acknowledgments}
\addcontentsline{toc}{section}{Acknowledgments}

We thank Donald Richards (Penn State University) for providing an earlier version of the proof of Proposition~\ref{prop:test.statistic}. We are also grateful to the referee for the careful reading and constructive feedback, which helped improve the paper.

%\section*{References}
\addcontentsline{toc}{section}{References}

\small

% This line reduces the gap between references
\setlength{\bibsep}{0pt plus 0ex}

\bibliographystyle{myjmva}
\bibliography{bib}

@article {AhmadCerrito1993,
  author   = {Ahmad, I. A. and Cerrito, P. B.},
  title    = {Goodness of fit tests based on the {$L_2$}-norm of multivariate probability density functions},
  fjournal = {Journal of Nonparametric Statistics},
  journal  = {J. Nonparametr. Statist.},
  volume   = {2},
  number   = {2},
  pages    = {169--181},
  year     = {1993},
  doi      = {10.1080/10485259308832550},
  mrnumber = {1256380}
}

@article {AndersenBollerslevDieboldLabys2003,
  author   = {Andersen, T. G. and Bollerslev, T. and Diebold, F. X. and Labys, P.},
  title    = {Modeling and forecasting realized volatility},
  fjournal = {Econometrica. Journal of the Econometric Society},
  journal  = {Econometrica},
  volume   = {71},
  number   = {2},
  pages    = {579--625},
  year     = {2003},
  doi      = {10.1111/1468-0262.00418},
  mrnumber = {1958138}
}

@article {AndersonHallTitterington1994,
  author   = {Anderson, N. H. and Hall, P. and Titterington, D. M.},
  title    = {Two-sample test statistics for measuring discrepancies between two multivariate probability density functions using kernel-based density estimates},
  fjournal = {Journal of Multivariate Analysis},
  journal  = {J. Multivariate Anal.},
  volume   = {50},
  number   = {1},
  pages    = {41--54},
  year     = {1994},
  doi      = {10.1006/jmva.1994.1033},
  mrnumber = {1292607}
}

@book {Anderson2003,
  author    = {Anderson, T. W.},
  title     = {An {I}ntroduction to {M}ultivariate {S}tatistical {A}nalysis},
  edition   = {Third},
  publisher = {Wiley-Interscience},
  address   = {Hoboken, NJ},
  pages     = {xx+721},
  year      = {2003},
  isbn      = {0-471-36091-0},
  mrnumber  = {1990662}
}

@article {BabiluaNadaraya2019,
  author   = {Babilua, P. K. and Nadaraya, E. A.},
  title    = {On one homogeneity test based on the kernel-type estimators of the distribution density},
  fjournal = {Ukrainian Mathematical Journal},
  journal  = {Ukrainian Math. J.},
  volume   = {71},
  number   = {4},
  pages    = {505--518},
  year     = {2019},
  doi      = {10.1007/s11253-019-01660-5},
  mrnumber = {3940208}
}

@article {BaringhausKolbe2015,
  author   = {Baringhaus, L. and Kolbe, D.},
  title    = {Two-sample tests based on empirical {H}ankel transforms},
  fjournal = {Statistical Papers},
  journal  = {Statist. Papers},
  volume   = {56},
  number   = {3},
  pages    = {597--617},
  year     = {2015},
  doi      = {10.1007/s00362-014-0599-1},
  mrnumber = {3369421}
}

@article {BaringhausTaherizadeh2010,
  author   = {Baringhaus, L. and Taherizadeh, F.},
  title    = {Empirical {H}ankel transforms and its applications to goodness-of-fit tests},
  fjournal = {Journal of Multivariate Analysis},
  journal  = {J. Multivariate Anal.},
  volume   = {101},
  number   = {6},
  pages    = {1445--1457},
  year     = {2010},
  doi      = {10.1016/j.jmva.2009.12.002},
  mrnumber = {2609505}
}

@article {BaringhausTaherizadeh2013,
  author   = {Baringhaus, L. and Taherizadeh, F.},
  title    = {A {K-S} type test for exponentiality based on empirical {H}ankel transforms},
  fjournal = {Communications in Statistics. Theory and Methods},
  journal  = {Comm. Statist. Theory Methods},
  volume   = {42},
  number   = {20},
  pages    = {3781--3792},
  year     = {2013},
  doi      = {10.1080/03610926.2011.639003},
  mrnumber = {3170965}
}

@article {BarndorffNielsenShephard2004,
  author   = {Barndorff-Nielsen, O. E. and Shephard, N.},
  title    = {Econometric analysis of realized covariation: high frequency based covariance, regression, and correlation in financial economics},
  fjournal = {Econometrica. Journal of the Econometric Society},
  journal  = {Econometrica},
  volume   = {72},
  number   = {3},
  pages    = {885--925},
  year     = {2004},
  doi      = {10.1111/j.1468-0262.2004.00515.x},
  mrnumber = {2051439}
}

@article {BasserMattielloLeBihan1994,
  author   = {Basser, P. J. and Mattiello, J. and {Le Bihan}, D.},
  title    = {{MR} diffusion tensor spectroscopy and imaging},
  fjournal = {Biophysical Journal},
  journal  = {Biophys. J.},
  volume   = {66},
  number   = {1},
  pages    = {259--267},
  year     = {1994},
  doi      = {10.1016/S0006-3495(94)80775-1}
}

@article {BasserPierpaoli1996,
  author   = {Basser, P. J. and Pierpaoli, C.},
  title    = {Microstructural and physiological features of tissues elucidated by quantitative-diffusion-tensor {MRI}},
  fjournal = {Journal of Magnetic Resonance, Series B},
  journal  = {J. Magn. Reson. Ser. B},
  volume   = {111},
  number   = {3},
  pages    = {209--219},
  year     = {1996},
  doi      = {10.1006/jmrb.1996.0086}
}

@article {Beaulieu2002,
  author   = {Beaulieu, C.},
  title    = {The basis of anisotropic water diffusion in the nervous system -- a technical review},
  fjournal = {NMR in Biomedicine},
  journal  = {NMR Biomed.},
  volume   = {15},
  number   = {7--8},
  pages    = {435--455},
  year     = {2002},
  doi      = {10.1002/nbm.782}
}

@misc {BelzileGenestOuimetRichards2025WishartKDE,
  author       = {Belzile, L. R. and Genest, C. and Ouimet, F. and Richards, D.},
  title        = {{Wishart} kernel density estimation for strongly mixing time series on the cone of positive definite matrices},
  howpublished = {arXiv preprint arXiv:2512.08232v2},
  year         = {2026},
  doi          = {10.48550/arXiv.2512.08232}
}

@article {Chen2000,
  author   = {Chen, S. X.},
  title    = {Probability density function estimation using gamma kernels},
  fjournal = {Annals of the Institute of Statistical Mathematics},
  journal  = {Ann. Inst. Statist. Math.},
  volume   = {52},
  number   = {3},
  pages    = {471--480},
  year     = {2000},
  doi      = {10.1023/A:1004165218295},
  mrnumber = {1794247}
}

@article {ConturoEtAl1999,
  author   = {Conturo, T. E. and Lori, N. F. and Cull, T. S. and Akbudak, E. and Snyder, A. Z. and Shimony, J. S. and McKinstry, R. C. and Burton, H. and Raichle, M. E.},
  title    = {Tracking neuronal fiber pathways in the living human brain},
  fjournal = {Proceedings of the National Academy of Sciences of the United States of America},
  journal  = {Proc. Natl. Acad. Sci. USA},
  volume   = {96},
  number   = {18},
  pages    = {10422--10427},
  year     = {1999},
  doi      = {10.1073/pnas.96.18.10422}
}

@article {deJong1987,
  author   = {de Jong, P.},
  title    = {A central limit theorem for generalized quadratic forms},
  fjournal = {Probability Theory and Related Fields},
  journal  = {Probab. Theory Related Fields},
  volume   = {75},
  number   = {2},
  pages    = {261--277},
  year     = {1987},
  doi      = {10.1007/BF00354037},
  mrnumber = {885466}
}

@article {Duong2013,
  author   = {Duong, T.},
  title    = {Local significant differences from nonparametric two-sample tests},
  fjournal = {Journal of Nonparametric Statistics},
  journal  = {J. Nonparametr. Stat.},
  volume   = {25},
  number   = {3},
  pages    = {635--645},
  year     = {2013},
  doi      = {10.1080/10485252.2013.810217},
  mrnumber = {3174288}
}

@article {EllingsonEtAl2017,
  author   = {Ellingson, L. A. and Groisser, D. and Osborne, D. E. and Patrangenaru, V. and Schwartzman, A.},
  title    = {Nonparametric bootstrap of sample means of positive-definite matrices with an application to diffusion-tensor-imaging data analysis},
  fjournal = {Communications in Statistics. Simulation and Computation},
  journal  = {Comm. Statist. Simulation Comput.},
  volume   = {46},
  number   = {6},
  pages    = {4851--4879},
  year     = {2017},
  doi      = {10.1080/03610918.2015.1136413},
  mrnumber = {3672582}
}

@article {GarreauArlot2018Consistent,
  author   = {Garreau, D. and Arlot, S.},
  title    = {Consistent change-point detection with kernels},
  fjournal = {Electronic Journal of Statistics},
  journal  = {Electron. J. Stat.},
  volume   = {12},
  number   = {2},
  pages    = {4440--4486},
  year     = {2018},
  doi      = {10.1214/18-EJS1513},
  mrnumber = {3892345}
}

@article {Gindikin1964,
  author   = {Gindikin, S. G.},
  title    = {Analysis in homogeneous domains},
  fjournal = {Russian Mathematical Surveys},
  journal  = {Russian Math. Surveys},
  volume   = {19},
  number   = {4},
  pages    = {1--89},
  year     = {1964},
  doi      = {10.1070/RM1964v019n04ABEH001153},
  mrnumber = {171941}
}

@article {GourierouxJasiakSufana2009,
  author   = {Gouri{\'e}roux, C. and Jasiak, J. and Sufana, R.},
  title    = {The {W}ishart autoregressive process of multivariate stochastic volatility},
  fjournal = {Journal of Econometrics},
  journal  = {J. Econometrics},
  volume   = {150},
  number   = {2},
  pages    = {167--181},
  year     = {2009},
  doi      = {10.1016/j.jeconom.2008.12.016},
  mrnumber = {2535514}
}

@book {GuptaNagar2000,
  author    = {Gupta, A. K. and Nagar, D. K.},
  title     = {Matrix {V}ariate {D}istributions},
  volume    = {104},
  publisher = {Chapman \& Hall/CRC},
  address   = {Boca Raton, FL},
  pages     = {x+367},
  year      = {2000},
  isbn      = {1-58488-046-5},
  doi       = {10.1201/9780203749289},
  mrnumber  = {1738933}
}

@phdthesis {Hadjicosta2019PhD,
  author  = {Hadjicosta, E.},
  title   = {Integral {T}ransform {M}ethods in {G}oodness-of-{F}it {T}esting},
  school  = {The Pennsylvania State University},
  address = {University Park, PA},
  year    = {2019}
}

@article {MR4169380,
  author   = {Hadjicosta, E. and Richards, D.},
  title    = {Integral transform methods in goodness-of-fit testing, {II}: the {W}ishart distributions},
  fjournal = {Annals of the Institute of Statistical Mathematics},
  journal  = {Ann. Inst. Statist. Math.},
  volume   = {72},
  number   = {6},
  pages    = {1317--1370},
  year     = {2020},
  doi      = {10.1007/s10463-019-00737-z},
  mrnumber = {4169380}
}

@article {HlavkaHuskovaMeintanis2020Paired,
  author   = {Hl{\'a}vka, Z. and Hu{\v{s}}kov{\'a}, M. and Meintanis, S. G.},
  title    = {Change-point methods for multivariate time-series: paired vectorial observations},
  fjournal = {Statistical Papers},
  journal  = {Statist. Papers},
  volume   = {61},
  number   = {4},
  pages    = {1351--1383},
  year     = {2020},
  doi      = {10.1007/s00362-020-01175-3},
  mrnumber = {4127478}
}

@article {HofmannScholkopfSmola2008,
  author   = {Hofmann, T. and Sch{\"o}lkopf, B. and Smola, A. J.},
  title    = {Kernel methods in machine learning},
  fjournal = {The Annals of Statistics},
  journal  = {Ann. Statist.},
  volume   = {36},
  number   = {3},
  pages    = {1171--1220},
  year     = {2008},
  doi      = {10.1214/009053607000000677},
  mrnumber = {2418654}
}

@book {HornJohnson2013,
  author    = {Horn, R. A. and Johnson, C. R.},
  title     = {Matrix {A}nalysis},
  edition   = {Second},
  publisher = {Cambridge University Press},
  address   = {Cambridge},
  pages     = {xviii+643},
  year      = {2013},
  isbn      = {978-0-521-54823-6},
  doi       = {10.1017/CBO9781139020411},
  mrnumber  = {2978290}
}

@article {HuskovaMeintanis2006ECF,
  author   = {Hu{\v{s}}kov{\'a}, M. and Meintanis, S. G.},
  title    = {Change point analysis based on empirical characteristic functions},
  fjournal = {Metrika. International Journal for Theoretical and Applied Statistics},
  journal  = {Metrika},
  volume   = {63},
  number   = {2},
  pages    = {145--168},
  year     = {2006},
  doi      = {10.1007/s00184-005-0008-9},
  mrnumber = {2242537}
}

@article {HuskovaMeintanis2006Ranks,
  author   = {Hu{\v{s}}kov{\'a}, M. and Meintanis, S. G.},
  title    = {Change-point analysis based on empirical characteristic functions of ranks},
  fjournal = {Sequential Analysis. Design Methods \& Applications},
  journal  = {Sequential Anal.},
  volume   = {25},
  number   = {4},
  pages    = {421--436},
  year     = {2006},
  doi      = {10.1080/07474940600934888},
  mrnumber = {2271923}
}

@book {MR2814399,
  author    = {Ledoux, M. and Talagrand, M.},
  title     = {Probability in {B}anach {S}paces: {I}soperimetry and {P}rocesses},
  publisher = {Springer-Verlag},
  address   = {Berlin},
  pages     = {xii+480},
  year      = {2011},
  isbn      = {978-3-642-20211-7},
  doi       = {10.1007/978-3-642-20212-4},
  mrnumber  = {2814399}
}

@article {Li1996,
  author   = {Li, Q.},
  title    = {Nonparametric testing of closeness between two unknown distribution functions},
  fjournal = {Econometric Reviews},
  journal  = {Econometric Rev.},
  volume   = {15},
  number   = {3},
  pages    = {261--274},
  year     = {1996},
  doi      = {10.1080/07474939608800355},
  mrnumber = {1410879}
}

@article {Li1999,
  author   = {Li, Q.},
  title    = {Nonparametric testing the similarity of two unknown density functions: local power and bootstrap analysis},
  fjournal = {Journal of Nonparametric Statistics},
  journal  = {J. Nonparametr. Statist.},
  volume   = {11},
  number   = {1--3},
  pages    = {189--213},
  year     = {1999},
  doi      = {10.1080/10485259908832780},
  mrnumber = {1719461}
}

@article {LiMaasoumiRacine2009,
  author   = {Li, Q. and Maasoumi, E. and Racine, J. S.},
  title    = {A nonparametric test for equality of distributions with mixed categorical and continuous data},
  fjournal = {Journal of Econometrics},
  journal  = {J. Econometrics},
  volume   = {148},
  number   = {2},
  pages    = {186--200},
  year     = {2009},
  doi      = {10.1016/j.jeconom.2008.10.007},
  mrnumber = {2500656}
}

@article {LiXieDaiSong2019ScanB,
  author   = {Li, S. and Xie, Y. and Dai, H. and Song, L.},
  title    = {Scan {$B$}-statistic for kernel change-point detection},
  fjournal = {Sequential Analysis. Design Methods \& Applications},
  journal  = {Sequential Anal.},
  volume   = {38},
  number   = {4},
  pages    = {503--544},
  year     = {2019},
  doi      = {10.1080/07474946.2019.1686886},
  mrnumber = {4057156}
}

@article {Lukic2024Laplace,
  author   = {Luki{\'c}, {\v{Z}}.},
  title    = {A {L}aplace transform-based test for the equality of positive semidefinite matrix distributions},
  fjournal = {Univerzitet u Ni{\v{s}}u. Prirodno-Matemati{\v{c}}ki Fakultet. Filomat},
  journal  = {Filomat},
  volume   = {38},
  number   = {26},
  pages    = {9343--9359},
  year     = {2024},
  doi      = {10.2298/FIL2426343L},
  mrnumber = {4852661}
}

@article {LukicMilosevic2024AISM,
  author   = {Luki{\'c}, {\v{Z}}. and Milo{\v{s}}evi{\'c}, B.},
  title    = {A novel two-sample test within the space of symmetric positive definite matrix distributions and its application in finance},
  fjournal = {Annals of the Institute of Statistical Mathematics},
  journal  = {Ann. Inst. Statist. Math.},
  volume   = {76},
  number   = {5},
  pages    = {797--820},
  year     = {2024},
  doi      = {10.1007/s10463-024-00902-z},
  mrnumber = {4798673}
}

@article {LukicMilosevic2024ChangePoint,
  author   = {Luki{\'c}, {\v{Z}}. and Milo{\v{s}}evi{\'c}, B.},
  title    = {Change-point analysis for matrix data: the empirical {H}ankel transform approach},
  fjournal = {Statistical Papers},
  journal  = {Statist. Papers},
  volume   = {65},
  number   = {9},
  pages    = {5955--5980},
  year     = {2024},
  doi      = {10.1007/s00362-024-01596-4},
  mrnumber = {4827983}
}

@article {MattesonJames2014Nonparametric,
  author   = {Matteson, D. S. and James, N. A.},
  title    = {A nonparametric approach for multiple change point analysis of multivariate data},
  fjournal = {Journal of the American Statistical Association},
  journal  = {J. Amer. Statist. Assoc.},
  volume   = {109},
  number   = {505},
  pages    = {334--345},
  year     = {2014},
  doi      = {10.1080/01621459.2013.849605},
  mrnumber = {3180567}
}

@book {Muirhead1982,
  author    = {Muirhead, R. J.},
  title     = {Aspects of {M}ultivariate {S}tatistical {T}heory},
  publisher = {John Wiley \& Sons, Inc.},
  address   = {New York},
  pages     = {xix+673},
  year      = {1982},
  isbn      = {0-471-09442-0},
  doi       = {10.1002/9780470316559},
  mrnumber  = {652932}
}

@article {Olkin1998,
  author   = {Olkin, I.},
  title    = {The density of the inverse and pseudo-inverse of a random matrix},
  fjournal = {Statistics \& Probability Letters},
  journal  = {Statist. Probab. Lett.},
  volume   = {38},
  number   = {2},
  pages    = {131--135},
  year     = {1998},
  doi      = {10.1016/S0167-7152(97)00163-6},
  mrnumber = {1627918}
}

@article {OsborneEtAl2013,
  author   = {Osborne, D. E. and Patrangenaru, V. and Ellingson, L. and Groisser, D. and Schwartzman, A.},
  title    = {Nonparametric two-sample tests on homogeneous {R}iemannian manifolds, {C}holesky decompositions and diffusion tensor image analysis},
  fjournal = {Journal of Multivariate Analysis},
  journal  = {J. Multivariate Anal.},
  volume   = {119},
  pages    = {163--175},
  year     = {2013},
  doi      = {10.1016/j.jmva.2013.04.006},
  mrnumber = {3061421}
}

@article {MR4358612,
  author   = {Ouimet, F.},
  title    = {A symmetric matrix-variate normal local approximation for the {W}ishart distribution and some applications},
  fjournal = {Journal of Multivariate Analysis},
  journal  = {J. Multivariate Anal.},
  volume   = {189},
  pages    = {Paper No. 104923, 17 pp.},
  year     = {2022},
  doi      = {10.1016/j.jmva.2021.104923},
  mrnumber = {4358612}
}

@article {PierpaoliJezzardBasserBarnettDiChiro1996,
  author   = {Pierpaoli, C. and Jezzard, P. and Basser, P. J. and Barnett, A. and {Di Chiro}, G.},
  title    = {Diffusion tensor {MR} imaging of the human brain},
  fjournal = {Radiology},
  journal  = {Radiology},
  volume   = {201},
  number   = {3},
  pages    = {637--648},
  year     = {1996},
  doi      = {10.1148/radiology.201.3.8939209}
}

@book {RasmussenWilliams2006,
  author    = {Rasmussen, C. E. and Williams, C. K. I.},
  title     = {Gaussian {P}rocesses for {M}achine {L}earning},
  publisher = {MIT Press},
  address   = {Cambridge, MA},
  pages     = {xviii+248},
  year      = {2006},
  isbn      = {978-0-262-18253-9},
  doi       = {10.7551/mitpress/3206.001.0001},
  mrnumber  = {2514435}
}

@incollection {Richards2010,
  author    = {Richards, D. S. P.},
  title     = {Functions of matrix argument},
  booktitle = {{NIST} {H}andbook of {M}athematical {F}unctions},
  editor    = {Olver, F. W. J. and Lozier, D. W. and Boisvert, R. F. and Clark, C. W.},
  publisher = {Cambridge University Press},
  address   = {Cambridge},
  chapter   = {35},
  pages     = {767--774},
  year      = {2010},
  isbn      = {978-0-521-14063-8},
  mrnumber  = {2655375}
}

@book {ScholkopfSmola2002,
  author    = {Sch{\"o}lkopf, B. and Smola, A. J.},
  title     = {Learning with {K}ernels: {S}upport {V}ector {M}achines, {R}egularization, {O}ptimization, and {B}eyond},
  publisher = {The MIT Press},
  address   = {Cambridge, MA},
  pages     = {xviii+626},
  year      = {2001},
  isbn      = {978-0-262-19475-4},
  doi       = {10.7551/mitpress/4175.001.0001}
}

@article {Schwartzman2016Lognormal,
  author   = {Schwartzman, A.},
  title    = {Lognormal distributions and geometric averages of symmetric positive definite matrices},
  fjournal = {International Statistical Review. Revue Internationale de Statistique},
  journal  = {Int. Stat. Rev.},
  volume   = {84},
  number   = {3},
  pages    = {456--486},
  year     = {2016},
  doi      = {10.1111/insr.12113},
  mrnumber = {3580425}
}

@article {SchwartzmanDoughertyTaylor2010,
  author   = {Schwartzman, A. and Dougherty, R. F. and Taylor, J. E.},
  title    = {Group comparison of eigenvalues and eigenvectors of diffusion tensors},
  fjournal = {Journal of the American Statistical Association},
  journal  = {J. Amer. Statist. Assoc.},
  volume   = {105},
  number   = {490},
  pages    = {588--599},
  year     = {2010},
  doi      = {10.1198/jasa.2010.ap07291},
  mrnumber = {2724844}
}

@article {WhitcherEtAl2007,
  author   = {Whitcher, B. and Wisco, J. J. and Hadjikhani, N. and Tuch, D. S.},
  title    = {Statistical group comparison of diffusion tensors via multivariate hypothesis testing},
  fjournal = {Magnetic Resonance in Medicine},
  journal  = {Magn. Reson. Med.},
  volume   = {57},
  number   = {6},
  pages    = {1065--1074},
  year     = {2007},
  doi      = {10.1002/mrm.21229}
}

@article {YouPark2021,
  author   = {You, K. and Park, H.-J.},
  title    = {Re-visiting {R}iemannian geometry of symmetric positive definite matrices for the analysis of functional connectivity},
  fjournal = {NeuroImage},
  journal  = {NeuroImage},
  volume   = {225},
  pages    = {117464},
  year     = {2021},
  doi      = {10.1016/j.neuroimage.2020.117464}
}

@article {YouPark2022,
  author   = {You, K. and Park, H.-J.},
  title    = {Geometric learning of functional brain network on the correlation manifold},
  fjournal = {Scientific Reports},
  journal  = {Sci. Rep.},
  volume   = {12},
  number   = {1},
  pages    = {17752},
  year     = {2022},
  doi      = {10.1038/s41598-022-21376-0}
}

\end{document}